\documentclass[11pt,twoside]{article}

\usepackage{fullpage}

\input{commands_final}

\makeatletter
\long\def\@makecaption#1#2{
        \vskip 0.8ex
        \setbox\@tempboxa\hbox{\small {\bf #1:} #2}
        \parindent 1.5em  
        \dimen0=\hsize
        \advance\dimen0 by -3em
        \ifdim \wd\@tempboxa >\dimen0
                \hbox to \hsize{
                        \parindent 0em
                        \hfil 
                        \parbox{\dimen0}{\def\baselinestretch{0.96}\small
                                {\bf #1.} #2
                                } 
                        \hfil}
        \else \hbox to \hsize{\hfil \box\@tempboxa \hfil}
        \fi
        }
\makeatother

\begin{document}

\begin{center}

  {\bf{\LARGE{ \mbox{Isotonic regression with unknown permutations:} Statistics, computation, and adaptation}}}

\vspace*{.2in}

{\large{
\begin{tabular}{ccc}
Ashwin Pananjady$^\star$ and Richard J. Samworth$^\dagger$
\end{tabular}
}}
\vspace*{.2in}

\begin{tabular}{c}
$^\star$Schools of Industrial \& Systems Engineering and Electrical \& Computer Engineering, \\
Georgia Institute of Technology \\
$^\dagger$Statistical Laboratory, University of Cambridge
\end{tabular}

\vspace*{.2in}

\today

\vspace*{.2in}

\large{\emph{Dedicated to the memory of Matthew Brennan}}
\vspace*{.2in}

\begin{abstract}
Motivated by models for multiway comparison data,
we consider the problem of estimating a
coordinate-wise isotonic function on the domain
$[0, 1]^d$ from noisy observations collected on a uniform lattice,
but where the design points have been permuted along each dimension.
While the univariate and bivariate versions
of this problem have received significant
attention, our focus is on the multivariate case $d \geq 3$.
We study both the minimax risk of estimation (in empirical 
$L_2$ loss) and the fundamental limits of adaptation (quantified
by the adaptivity index) to a family of piecewise constant
functions. We
provide a computationally efficient Mirsky partition estimator that is
minimax optimal 
while also achieving the smallest
adaptivity index possible for polynomial time procedures. Thus, 
from a worst-case perspective and in sharp
contrast to the bivariate
case, 
the latent permutations in the model do not introduce
significant computational difficulties over and above
vanilla isotonic regression. 
On the other hand, the fundamental limits of adaptation are 
significantly different with and
without unknown permutations: Assuming a hardness conjecture from
average-case complexity theory, a statistical-computational gap manifests
in the former case.
In a complementary direction, we show that natural modifications of
existing estimators
fail to satisfy at least one of the desiderata of optimal worst-case
statistical performance, computational efficiency, and fast adaptation.
Along the way to showing our results, we improve adaptation results in
the special case $d = 2$ and establish some properties of estimators
for vanilla isotonic regression, both of which may be of independent interest.
\end{abstract}
\end{center}


\section{Introduction} \label{sec:intro}

Consider the problem of estimating a degree $d$, real-valued tensor $\thetastar \in \real^{n_1 \times \cdots \times n_d}$, whose entries are observed with noise. As in many problems in high-dimensional statistics, this tensor estimation problem requires a prohibitively large number of observations to solve without the imposition of further structure, and consequently, many structural constraints have been placed in particular applications of tensor estimation. For instance, low ``rank” structure is common in chemistry and neuroscience applications~\citep{andersen2003practical,mocks1988decomposing}, blockwise constant structure is common in applications to clustering and classification of relational data~\citep{zhou2007learning},  sparsity is commonly used in data mining applications~\citep{kolda2005higher}, and variants and combinations of such assumptions have also appeared in other contexts~\citep{zhou2015bayesian}.
In this paper, we study a flexible, nonparametric structural assumption that generalizes parametric assumptions in applications of tensor models to discrete choice data.

Suppose we are interested in modeling ordinal data, which arises in applications ranging from information
retrieval~\citep{dwork2001rank} and assortment optimization~\citep{kok2008assortment} to recommender systems~\citep{baltrunas2010group} and crowdsourcing~\citep{chen2013pairwise}; in a generic such problem,
we have $n_1$ ``items'', subsets of which are evaluated using a multiway comparison. In particular, each datum takes the form of a tuple containing $d$ of these items, and a single item that is chosen from the tuple as the ``winner" of this comparison.
Such data can be represented using a stochastic model of choice: For each tuple $A$ and each item $i \in A$, suppose that $i$ wins the comparison with probability $p(i, A)$. The winner of each comparison is then modeled as a random variable; equivalently, the overall statistical model is described by a $d$-dimensional mean tensor $\thetastar \in \real^{n_1 \times \cdots \times n_1}$, where $\thetastar(i_1, \ldots, i_d) = p(i_1, (i_1, \ldots, i_d) )$, and our data consist of noisy observations of entries of this tensor.
Imposing sensible constraints on the tensor $\thetastar$ in these
applications goes back to classical, axiomatic work on the subject due to~\citet{luce1959individual} and~\citet{plackett1975analysis}.
A natural and flexible assumption is given by \emph{simple scalability}~\citep{krantz1965scaling,mcfadden1981econometric,tversky1972elimination}: this states that each of the $n_1$ items can be associated with some scalar utility (item $i$ with utility $u_i$), and that the comparison probability is given by
\begin{align} \label{eq:mono}
\thetastar(i_1,\ldots, i_d) = f(u_{i_1}, \ldots, u_{i_d}),
\end{align}
where $f$ is a non-decreasing function of its first argument and a coordinate-wise non-increasing function of the remaining arguments.
Operationally, an item should not have a lower chance of being chosen as a winner if---all else remaining equal---its utility were to be increased.

There are many models that satisfy the nonparametric simple scalability assumption, in particular, \emph{parametric} assumptions in which a specific form of the function $f$ is posited. The simplest parameterization is given by $f(u_1, \ldots, u_d) = u_1 / \sum_{j = 1}^d u_j$, which dates back to~\citet{luce1959individual}. A logarithmic transformation of Luce's parameterization leads to the multinomial logit (MNL) model, which has seen tremendous popularity in applications ranging from transportation~\citep{ben1985discrete} to marketing~\citep{chandukala2008choice}. See, e.g.,~\citet{mcfadden1973conditional} for a classical but comprehensive introduction to this class of models. However the parametric assumptions of the MNL model have been called into question by a line of work showing that greater flexibility in modeling can lead to improved results in many applications (see, e.g.,~\citet{farias2013nonparametric} and references therein).

The simple scalability (SS) assumption has also been extensively explored when $d = 2$, i.e., in the pairwise comparison case. In this case, the nonparametric SS assumption is equivalent to \emph{strong stochastic transitivity}, and a long line of work~\citep{marschak1957experimental,mclaughlin1965stochastic,fishburn1973binary} has studied its empirical properties. In particular, the parametric MNL model specialized to this case corresponds to the popular Bradley--Terry--Luce model~\citep{bradley52rank,luce1959individual}, and nonparametric models are known to be significantly more robust to misspecification than their parametric counterparts in common applications~\citep{marschak1957experimental,ballinger1997decisions}.

Let us return now to the SS assumption in the general case, and state an equivalent formulation in terms of structure on the tensor $\thetastar$. For two vectors of equal dimension, let $x \preceq y$ denote that $x - y \leq 0$ entrywise, and let $\pi$ denote any permutation of $[n_1]$ that orders the utilities in the sense that $u_{\pi(1)} \leq \cdots \leq u_{\pi(n_1)}$. 
The monotonicity of the function $f$ in the SS assumption~\eqref{eq:mono} ensures that 
whenever $(i_1, \ldots, i_d) \preceq (i_1', \ldots, i_d')$, we have
\begin{align} \label{eq:structure-first}
\thetastar\bigl(\pi(i_1), \pi^{-1}(i_2), \ldots, \pi^{-1}(i_d) \bigr) &= f \bigl(u_{\pi(i_1)}, u_{\pi^{-1}(i_2)}, \ldots, u_{\pi^{-1}(i_d)} \bigr) \notag \\
&\leq f \bigl(u_{\pi(i'_1)}, u_{\pi^{-1}(i'_2)}, \ldots, u_{\pi^{-1}(i'_d)} \bigr) \notag \\
&= \thetastar\bigl(\pi(i'_1), \pi^{-1}(i_2'), \ldots, \pi^{-1}(i'_d) \bigr).
\end{align}
Crucially, since the utilities themselves are latent, the permutation $\pi$ is \emph{unknown}---indeed, it represents the ranking that must be estimated from our data---and so $\thetastar$ is a coordinate-wise isotonic tensor with unknown permutations. In the multiway comparison problem, this tensor represents the stochastic model underlying our data, and accurate knowledge of these probabilities is useful, for instance, in informing pricing and revenue management decisions in assortment optimization applications~\citep{kok2008assortment}.

While multiway comparisons form our primary motivation, the flexibility afforded by nonparametric models
with latent permutations has also been noticed and exploited in other applications. For instance, in psychometric item-response theory, the
Mokken model---which corresponds to imposing structure of the form~\eqref{eq:structure-first} when $d = 2$---is known to be significantly more robust to misspecification that the parametric Rasch model; see~\citet{van2003mokken} for an introduction and survey. In crowd-labeling, the permutation-based model~\citep{ShaBalWai16} has seen empirical success in applications where the parametric Dawid--Skene model~\citep{DawSke79} imposes stringent assumptions. 
Besides these, there are also several other examples of tensor estimation problems in which parametric structure is frequently assumed; for example, in click modeling~\citep{craswell2008experimental} and random hypergraph models~\citep{ghoshdastidar2017consistency,angelini2015spectral}. Similarly to before, nonparametric structure has the potential to generalize and lend flexibility to these parametric models.

It is worth noting that in many of the aforementioned applications, the underlying objects can be clustered into near identical sets. For example, there is evidence that such ``indifference sets" of items exist in crowdsourcing (see~\citet[Figure 1]{ShaBalWai16-2} for an illuminating example) and peer review applications involving comparison data~\citep{nguyen2014codewebs}; clustering is often used in the application of psychometric evaluation methods~\citep{hardouin2004clustering}, and many models for communities in hypergraphs posit the existence of such clusters of nodes~\citep{abbe2013conditional,ghoshdastidar2017consistency}. For a precise mathematical definition of indifference sets and how they induce further structure in the tensor $\thetastar$, see Section~\ref{sec:setup}. Whenever such additional structure exists, it is conceivable that estimation can be performed in a more sample-efficient manner; we will precisely quantify such a phenomenon in our exposition in Section~\ref{sec:main-results}.

Using these applications as motivation, our goal in this paper is to study the tensor estimation problem under the nonparametric structural assumptions~\eqref{eq:structure-first} of monotonicity constraints and unknown permutations.

\subsection{Related work}

Regression problems with unknown permutations were classically studied in applications to record-linkage~\citep{degroot1980estimation}, and similar models have witnessed recent interest driven by other modern applications in machine learning and signal processing; see, e.g.,~\citet{collier2016minimax,unnikrishnan2018unlabeled,pananjady2017denoising,pananjady2017linear, hsu2017linear,abid2018stochastic,behr2017minimax} for theoretical results and applications. We focus our discussion on the sub-class of such problems involving monotonic shape-constraints and (vector/matrix/tensor) estimation. When $d = 1$, the assumption~\eqref{eq:structure-first} corresponds to the ``uncoupled'' or ``shuffled'' univariate isotonic regression problem~\citep{carpentier16b}. Here, an estimator based on Wasserstein deconvolution is known to attain the minimax rate $\log \log n / \log n$ in (normalized) squared $\ell_2$-error for estimation of the underlying (sorted) vector of length $n$~\citep{rigollet2019uncoupled}. In a recent paper,~\citet{balabdaoui2020unlinked} considered a closely related problem, with a focus on isolating the effect of the noise distribution on the deconvolution procedure. A multivariate version of this problem (estimating multiple isotonic functions under a common unknown permutation of coordinates) has also been studied under the moniker of ``statistical seriation'', and has been shown to have applications to archaeology and metagenomics~\citep{FlaMaoRig16,ma2019optimal}.

The case $d = 2$ has also seen a long line of work in the mathematical statistics community in the context of estimation from pairwise comparisons, wherein the monotonicity assumption~\eqref{eq:structure-first} corresponds to strong stochastic transitivity, or SST for short~\citep[e.g.,][]{chatterjee2015matrix,ShaBalGunWai17,ChaMuk16,ShaBalWai16-2,mao2018towards}. Relatives of this model have also appeared in the context of prediction in graphon estimation~\citep{ChaAir14,AirCosCha13} and calibration in crowd-labeling~\citep{ShaBalWai16,mao2018towards}. The minimax rate (in normalized, squared Frobenius error) of estimating an $(n^{1/2} \times n^{1/2})$ SST matrix is known to be of the order $n^{-1/2}$ up to a polylogarithmic factor, but many computationally efficient algorithms~\citep{chatterjee2015matrix,ShaBalGunWai17,ChaMuk16,ShaBalWai16-2} achieved only the rate $n^{-1/4}$. Recent progress has shown more sophisticated (but still efficient) procedures with improved rates: an algorithm with rate $n^{-3/8}$ was given by~\citet{mao18breakingcolt}, and in recent work,~\citet{liu2020better} show that a rate $n^{-5/12}$ can be achieved. However, it is still not known whether the minimax rate is attainable by an efficient algorithm. The case of estimating rectangular matrices has also been studied, and the fundamental limits are known to be sensitive to the aspect ratio of the problem~\citep{mao2018towards}.  Interesting adaptation properties are also known in this case, both to parametric structure~\citep{ChaMuk16}, and to indifference sets~\citep{ShaBalWai16-2}. 

To the best of our knowledge, analogs of these results have not been explored in the multivariate setting $d \geq 3$, although a significant body of literature has studied parametric models for choice data in this case (see, e.g.,~\citet{negahban2018learning} and references therein).

\subsection{Overview of contributions}

We begin by considering the minimax risk of estimating bounded tensors satisfying assumption~\eqref{eq:structure-first}, and show in Proposition~\ref{prop:full-funlim} that when $d \geq 2$, it is dominated by the risk of estimating the underlying \emph{ordered} coordinate-wise isotonic tensor. In other words, the latent permutations do not significantly influence the statistical difficulty of the problem.
We also study the fundamental limits of estimating tensors having indifference set structure, and this allows us to assess the ability of an estimator to adapt to such structure via its \emph{adaptivity index} (to be defined precisely in equation~\eqref{eq:global-AI}). We establish two surprising phenomena in this context: First, we show in Proposition~\ref{prop:adapt-funlim} that the fundamental limits of estimating these objects preclude a parametric rate, in sharp contrast to the case without unknown permutations. Second, we prove in Theorem~\ref{thm:hardness} that the adaptivity index exhibits a statistical-computational gap under the assumption of a widely-believed conjecture in average-case complexity. In particular, we show that the adaptivity index of any polynomial time computable estimator must grow at least polynomially in $n$, assuming the hypergraph planted clique conjecture~\citep{brennan2020reducibility}. 
Our results also have interesting consequences for the isotonic regression problem without unknown permutations (see Proposition~\ref{cor:lse} and Corollary~\ref{cor:adapt-noperms}).

Having established these fundamental limits, we then turn to our main methodological contribution. We propose and analyze---in Theorem~\ref{thm:block-risk}---an estimator based on Mirsky's partitioning algorithm~\citep{mirsky1971dual} that estimates the underlying tensor (a) at the minimax rate (up to poly-logarithmic factors) for each $d \geq 3$ whenever this tensor has bounded entries, and (b) with the best possible adaptivity index for polynomial time procedures for all $d \geq 2$. The first of these findings is particularly surprising because it shows that the case $d \geq 3$ of this problem is distinctly different from the bivariate case, in that the minimax risk is achievable with a computationally efficient algorithm. This is in spite of the fact that there are more permutations to estimate as the dimension increases, which, at least in principle, ought to make the problem more difficult both statistically and computationally.

In addition to its favorable risk properties, the Mirsky partition estimator also has several other advantages: it is computable in time sub-quadratic in the size of the input, and its \emph{computational complexity} also adapts to underlying indifference set structure. In particular, when there are a fixed number of indifference sets, the estimator has almost linear computational complexity with high probability. When specialized to $d = 2$, this estimator exhibits significantly better adaptation properties to indifference set structure than known estimators that were designed specifically for this purpose; see Section~\ref{sec:adapt-bounded-MP} and Appendix~\ref{app:adapt} of 
for statements and discussions of these results.

To complement our upper bounds on the Mirsky partition estimator, we also show, somewhat surprisingly, that many other estimators proposed in the literature~\citep{ChaMuk16,ShaBalWai16-2,ShaBalGunWai17}, and natural variants thereof, suffer from an extremely large adaptivity index. In particular, they are unable to attain the polynomial time optimal adaptivity index (given by the fundamental limit established by Theorem~\ref{thm:hardness}) for any $d \geq 4$. This is in spite of the fact that some of these estimators are minimax optimal for estimation over the class of bounded tensors (see Propositions~\ref{prop:borda-worst-case} and~\ref{cor:lse}) for all $d \geq 3$. Thus, we see that simultaneously achieving good worst-case risk properties while remaining computationally efficient and adaptive to structure is a challenging requirement, thereby providing further evidence of the value of the Mirsky partitioning estimator.

\subsection{Organization}

The rest of this paper is organized as follows. In Section~\ref{sec:setup}, we introduce formally the estimation problem at hand. Section~\ref{sec:main-results} contains full statements and discussions of our main results, and Section~\ref{sec:discussion} contains some concluding remarks. Proofs of the main results are in Section~\ref{sec:proofs}. The  appendices also contain results that may be of independent interest. In particular, some of our results on adaptation in the special cases $d = 2, 3$ are postponed to Appendix~\ref{app:adapt}, and Appendix~\ref{app:iso} collects some properties of the vanilla isotonic regression estimator.

\section{Background and problem formulation} \label{sec:setup}

Let $\Pspace_k$
denote the set of all
 permutations on the set 
$[k] \defn \{1, \ldots, k\}$. We interpret $\real^{n_1 \times \cdots \times n_d}$ as the set of all
real-valued, tensors of dimension $n_1 \times \cdots \times n_d$. For a set
of positive integers $i_j \in [n_j], j \in [d]$, we use $T (i_1, \ldots, i_d)$ to index
entry $i_1, \ldots, i_d$ of a tensor $T \in \real^{n_1 \times \cdots \times
n_d}$.

The set of all real-valued, coordinate-wise isotonic functions on
the set $[0, 1]^d$ is denoted by
\begin{align*}
\Fspace_d \defn \left\{ f: [0, 1]^d \to \real: f(x_1, x_2, \ldots, x_
{d}) \leq f(x'_1, x'_2, \ldots, x'_{d}) \; \text{when} \;
x_j \leq x'_j \; \text{ for }\; j \in [d] \right\}.
\end{align*}
Let $n_j$ denote the number of observations along dimension $j$, with the total
number of observations given by $n \defn \prod_{j = 1}^{d} n_j$.
  For $n_1, \ldots, n_{d} \in \Natural$, let $\lattice_{d, n_1,
  \ldots, n_d} \defn \prod_{j = 1}^{d} [n_j]$ denote the
  $d$-dimensional lattice.
With this notation, we assume access to a tensor of observations $Y \in \Tspace^
{n_1 \times \cdots \times n_d}$, where
\begin{align*}
Y(i_1, \ldots, i_d) = f^* \left( \frac{\pistar_1 (i_1)}{n_1}, 
\frac{\pistar_2 (i_2)}{n_2}, \ldots, \frac{\pistar_d (i_d)}{n_d} \right) +
\epsilon (i_1, \ldots, i_d) \text{ for each } i_j \in [n_j], \; j \in [d].
\end{align*}
Here, the function $f^* \in \Fspace_d$ 
is unknown,
and for each $j \in [d]$, we also have an unknown permutation $\pistar_j \in
\Pspace_{n_j}$. The tensor $\epsilon \in \real^{n_1 \times \cdots \times n_d}$
represents noise in the observation process, and we assume that its entries are
given by independent standard normal random variables\footnote{We
study the canonical Gaussian setting for convenience, but our
results extend straightforwardly to sub-Gaussian noise distributions.}.
Denote the noiseless observations on the lattice by 
\begin{align*}
\thetastar(i_1, \ldots, i_d) \defn f^* \left( \frac{\pistar_1 (i_1)}{n_1}, 
\frac{\pistar_2 (i_2)}{n_2}, \ldots, \frac{\pistar_d (i_d)}{n_d} \right) \quad
\text{ for each } i_j \in [n_j], \; j \in [d];
\end{align*}
this is a generalization of\footnote{Note that unlike in equation~\eqref{eq:structure-first}, we now allow for a different
unknown permutation along each dimension for greater flexibility.} the nonparametric structure that was posited in
equation~\eqref{eq:structure-first}.

It is also convenient to define the set of tensors that can be formed by permuting evaluations of a coordinate-wise monotone function on the lattice by the permutations $(\pi_1, \ldots, \pi_d)$. Denote this set by
\begin{align*}
\Mspace (\lattice_{d, n_1, \ldots, n_d}; \pi_1, \ldots, \pi_d) &\defn \biggl\{
\theta \in \real^{n_1 \times \cdots \times n_d}: \; \exists f \in \Fspace_d
\; \text{such that } \forall i_j \in 
[n_j], \; j \in [d], \\
& \;\; \quad \qquad \qquad \qquad \theta(i_1, \ldots, i_d) = f \left( \frac{\pi_1 (i_1)}
{n_1}, \ldots, \frac{\pi_d (i_d)}{n_d} \right) \biggl\}.
\end{align*}
We use the shorthand $\Mspace (\lattice_{d, n_1, \ldots, n_d})$ to
denote this set when the permutations are all the identity.
Also define the set
\begin{align*}
\Mspace_{\perm}  (\lattice_{d, n_1, \ldots, n_d}) \defn \bigcup_{\pi_1 \in \Pspace_{n_1}} \cdots  \bigcup_{\pi_d \in \Pspace_{n_d}} \Mspace (\lattice_{d, n_1, \ldots, n_d}; \pi_1, \ldots, \pi_d)
\end{align*}
of tensors that can be formed by permuting evaluations of any coordinate-wise monotone function.

For a collection of permutations $\{ \pi_j \in \Pspace_{n_j} \}_{j = 1}^d$ and a tensor
$T \in \real^{n_1 \times \cdots \times n_d}$, we let
$T\{ \pi_1, \ldots, \pi_d\}$ denote the tensor $T$ viewed along permutation
$\pi_j$ on dimension~$j$. Specifically, we have
\[
T\{ \pi_1, \ldots, \pi_d\}(i_1, \ldots, i_d) = T(\pi_1(i_1), \ldots, \pi_d(i_d))
\quad
\text{ for each } \quad i_j \in [n_j], \;\; j \in [d].
\]
With this notation, note the
inclusion $\thetastar\{
(\pistar_1)^{-1}, \ldots,
(\pistar_d)^{-1} \} \in \Mspace (\lattice_{d, n_1, \ldots, n_d})$. However,
since we do not know the permutations $\pistar_1, \ldots, \pistar_d$ a priori,
we may only assume that \mbox{$\thetastar \in \Mspace_{\perm}  
(\lattice_{d, n_1, \ldots, n_d})$}, and our goal is to denoise our observations
and produce an
estimate of $\thetastar$.
We study the empirical $L_2$ risk of any such estimate $\thetahat \in \Tspace^{n_1 \times \cdots \times
n_d}$, given by
\begin{align*}
\Rspace_n (\thetahat, \thetastar) &\defn \EE \left[ \ell_n^2( \thetahat,
\thetastar) \right], \; \text{ where } \\
\ell_n^2 (\theta_1, \theta_2)
&\defn \frac{1}{n} \sum_{j = 1}^d \sum_{i_j = 1}^{n_j} \left( \theta_1(i_1,
\ldots, i_d) - \theta_2(i_1, \ldots, i_d) \right)^2.
\end{align*}
Note that the expectation is taken over both the noise $\epsilon$
and any randomness used to compute the estimate $\thetahat$.
In the case where $\thetahat \in \Mspace_{\perm}  (\lattice_{d, n_1, \ldots, n_d})$, we also produce a function estimate $\fhat \in \Fspace_d$ and permutation estimates
$\pihat_j \in \Pspace_{n_j}$ for $j \in [d]$, with
\begin{align*}
\thetahat(i_1, \ldots, i_d) \defn \fhat \left( \frac{\pihat_1 (i_1)}{n_1}, \frac{\pihat_2 (i_2)}{n_2}, \ldots, \frac{\pihat_d (i_d)}{n_d} \right) \quad \text{ for each } i_j \in [n_j], \; j \in [d].
\end{align*}
Note that in general, the resulting estimates $\fhat, \pihat_1, \ldots,
\pihat_d$ need not be unique, but this identifiability issue will not
concern us since we are only interested in the tensor $\thetahat$ as an estimate
of the tensor~$\thetastar$.

As alluded to in the introduction, it is
common in multiway comparisons for there to be \emph{indifference sets} of
items that all behave
identically. These sets are easiest to describe in the space of functions. 
For
each $j \in [d]$ and $s_j \in [n_j]$, let $I^j_1, \ldots, I^j_{s_j}$
denote a set of $s_j$ disjoint intervals such that $[0, 1] = \cup_{\ell = 1}^
{s_j} I^j_\ell$. Suppose that for each $\ell$, the length of the interval
$I^j_\ell$ exceeds $1/n_j$, so that we are assured that the
intersection of $I^j_\ell$ with the set $\frac{1}{n_j} \{ 1, \ldots, n_j \}$ is
non-empty. With a slight abuse of terminology, we also refer to this
intersection as an interval, and let the tuple $\kset^j = (k^j_1, \ldots,
k^j_{s_j})$ denote the cardinalities of these intervals, with $\sum_{\ell =1}^
{s_j} k^j_
{\ell} = n_j$. Let $\mathbf{K}_{s_j}$ denote the set
of all such tuples, and define $k^j_{\max} \defn \max_{\ell \in [s_j] }\kset^j_
{\ell}$. Collect $\{ \kset^j\}_{j = 1}^d$ in a tuple $\kfull = 
(\kset^1, \ldots, \kset^d)$, and the $d$ values $\{ s_j \}_{j = 1}^d$ in
a tuple $\sset = (s_1, \ldots, s_d)$. Let $\Kfull_{\sset}$ denote the set of
all such tuples~$\kfull$, and note that the possible values of $\sset$
range over the lattice $\lattice_{d, n_1, \ldots, n_d}$. Finally,
let $k^* \defn \min_{j \in [d]} k^j_{\max}$. See Figure~\ref{fig:indiff} 
for an illustration
when $d = 2$.

\begin{figure}[ht]
\centering
\includegraphics[clip, trim=11.5cm 0cm 7cm 10cm, width=0.5\linewidth]{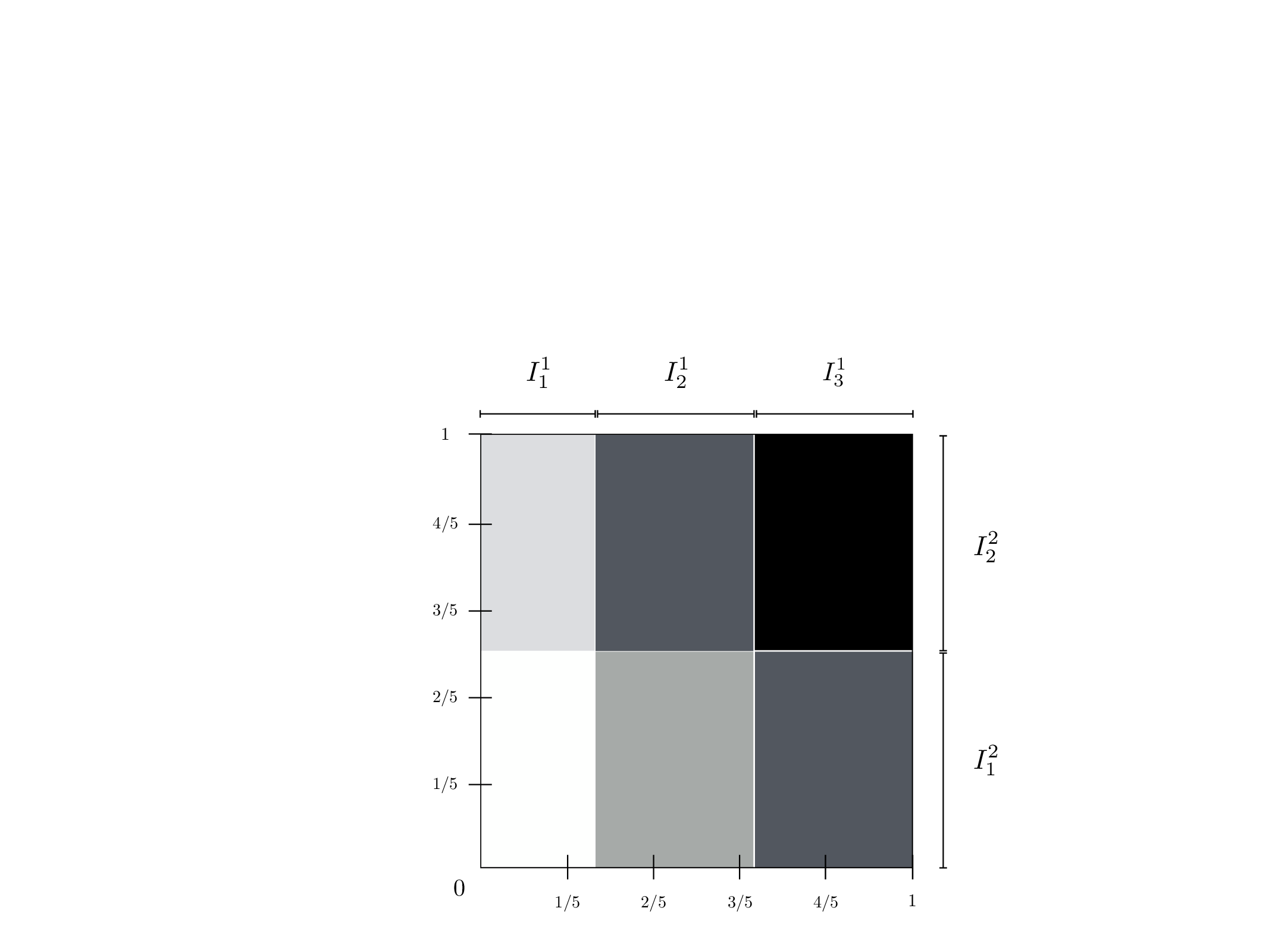}
\caption{An illustration of a block-wise constant isotonic function on $[0, 1]^2$, with lighter colors indicating smaller values. We observe this function on a $5 \times 5$, equally spaced grid (i.e., $n_1 = 5$). The number of indifference sets along the two dimensions satisfies $s_1 = 3$ and $s_2 = 2$, and their size tuples
satisfy $\kset^1 = (1, 2, 2)$ and $\kset^2 = (2, 3)$. As a result, we have $k^* = 2$.}.
\label{fig:indiff}
\end{figure}
If, for each $j \in [d]$, dimension $j$ of the domain is partitioned into the
intervals $I^j_1, \ldots, I^j_{s_j}$, then
the set $[0,1]^d$ is partitioned into $s \defn \prod_{j = 1}^d s_j$
hyper-rectangles. Here, the hyper-rectangular partition is a Cartesian 
product of univariate partitions,
i.e., each hyper-rectangle takes the form $\prod_{j=1}^d I^j_{\ell_j}$ for some sequence
of indices $\ell_j \in [s_j], j \in [d]$. We refer to the intersection of a
hyper-rectangle with the lattice $\lattice_{d, n_1, \ldots, n_d}$ also as a
hyper-rectangle, and note that $\kfull$ fully specifies such a
hyper-rectangular partition.
Denote by $\Mspace^{\kfull, \sset}
(\lattice_{d, n_1, \ldots, n_d})$ the
set of all $\theta \in \Mspace(\lattice_{d, n_1, \ldots, n_d})$ that are
piecewise constant on a hyper-rectangular partition specified by $\kfull$---we
have chosen to be explicit about the tuple $\sset$ in our notation for clarity.
Let $\Mspace^{\kfull, \sset}_
{\perm} (\lattice_{d,n_1, \ldots, n_d})$ denote the set of all
coordinate-wise permuted versions of $\theta \in \Mspace^{\kfull, \sset}
(\lattice_{d,
n_1, \ldots, n_d})$.

For the rest of this paper, we operate in the \emph{uniform, or balanced} case
$2 \leq n_1 = \cdots = n_d = n^{1/d}$, which is motivated by the comparison
setting introduced in Section~\ref{sec:intro}. We use the shorthand $\lattice_{d, n}$
to
represent the uniform lattice and $\Tspace_{d, n}$ to represent balanced tensors.
We continue to use the notation $n_j$ in some contexts since this simplifies
our exposition, and also continue to accommodate distinct permutations
$\pistar_1, \ldots, \pistar_d$ and cardinalities of indifference sets $s_1,
\ldots, s_d$ along the different dimensions for flexibility.

Let $\widehat{\Theta}$ denote the set of all estimators of $\thetastar$, i.e.~the set of all measurable functions 
(of the observation tensor $Y$) taking values in $\real_{d, n}$. Denote the
minimax risk
over the class of tensors in the set
 $\Mspace^{\kfull, \sset}_{\perm} (\lattice_
{d,n})$ by 
\begin{align*}
\minimax_{d, n}(\kfull, \sset) \defn \inf_{\thetahat \in \widehat{\Theta}} \;
\sup_
{\thetastar
\in \Mspace^{\kfull, \sset}_{\perm}(\lattice_
{d,
n})} \; \Rspace_n (\thetahat, \thetastar).
\end{align*}
Note that $\minimax_{d, n}(\kfull, \sset)$ measures the smallest possible risk
achievable with a priori 
knowledge of the inclusion $\thetastar \in \Mspace^{\kfull, \sset}_{\perm}
(\lattice_{d, n})$. On the other hand, we are interested in estimators that
\emph{adapt} to
hyper-rectangular structure without knowing of its existence in advance. One
way to measure the extent of adaptation of an estimator $\thetahat$ is in
terms of its
\emph{adaptivity index} to indifference set sizes $\kfull$, defined as
\begin{align*}
\adapt^{\kfull, \sset} ( \thetahat ) \defn \frac{\sup_{\thetastar
\in \Mspace^{\kfull, \sset}_{\perm}(\lattice_
{d,
n})} \; \Rspace_n (\thetahat, \thetastar)}{\minimax_{d, n}
(\kfull, \sset)}.
\end{align*}
A large value of this index indicates that the estimator $\thetahat$ is unable
to adapt satisfactorily to the set $\Mspace^{\kfull, \sset}_{\perm}(\lattice_
{d, n})$, since a much lower risk is achievable when the inclusion $\thetastar
\in \Mspace^{\kfull, \sset}_{\perm}(\lattice_
{d, n})$ is known in advance.
The \emph{global} adaptivity index of $\thetahat$ is then given by
\begin{align} \label{eq:global-AI}
\adapt ( \thetahat ) \defn \max_{\sset \in \lattice_{d, n}} \max_{\kfull \in 
\Kfull_{\sset}} \;
\adapt^{\kfull, \sset} (
\thetahat ).
\end{align}
We note that similar definitions of an adaptivity index or factor have
appeared in the literature; our definition most closely resembles the index
defined by~\citet{ShaBalWai16-2}, but similar concepts go back at
least to Lepski and co-authors~\citep{lepski1991problem,lepski1997optimal}.

Finally, for a tensor $X \in \Tspace_{d, n}$ and closed set $\mathcal{C}
\subseteq \Tspace_{d, n}$, it is useful to define the $L_2$-projection of $X$ onto
$\mathcal{C}$ by
\begin{align} \label{eq:lse-defn}
\thetahatlse (\mathcal{C}, X) \in \argmin_{\theta \in \mathcal{C}} \; \ell_n^2 
(X,
\theta).
\end{align}
In our exposition to follow, the set $\mathcal{C}$ will either be compact
or a finite union of closed, convex sets, and so the projection is guaranteed to
exist. When $\mathcal{C}$ is closed and convex, the projection is additionally
unique.

\paragraph*{General notation} For a (semi-)normed space $(\mathcal{F}, \|
\cdot \|)$ and positive scalar $\delta$, let $N(\delta; \mathcal{F}, \|
\cdot \|)$ denote its
$\delta$-covering
number, i.e., the minimum cardinality of any set $U \subseteq \Fspace$ such that
$\inf_{u \in U} \| x - u \| \leq \delta$ for all $x \in \Fspace$. 
Let $\infball(t)$ and $\mathbb{B}_2(t)$ denote the $\ell_\infty$ and $\ell_2$ closed
balls of 
radius $t$ in $\Tspace_{d,n}$, respectively. Let $\ind{\cdot}$ denote the
indicator function, and denote by $\mathbbm{1}_{d, n} \in \real_{d, n}$ the all-ones tensor.
For two sequences of non-negative reals $\{f_n\}_{n
\geq 1}$ and $\{g_n \}_{n \geq 1}$, we use $f_n \lesssim g_n$ to indicate that
there is a universal positive constant $C$ such that $f_n \leq C g_n$ for all
$n \geq 1$. The relation $f_n \gtrsim g_n$ indicates that $g_n \lesssim f_n$,
and we say that $f_n \asymp g_n$ if both $f_n \lesssim g_n$ and $f_n \gtrsim
g_n$ hold simultaneously. We also use standard order notation $f_n = \order
(g_n)$ to indicate that $f_n \lesssim g_n$ and $f_n = \ordertil(g_n)$ to
indicate that $f_n \lesssim
g_n \log^c n$, for a universal constant $c>0$. We say that $f_n = \Omega(g_n)$ (resp. $f_n = \widetilde{\Omega}(g_n)$) if $g_n = \order(f_n)$ (resp. $g_n = \ordertil(f_n)$). The notation $f_n = o(g_n)$ is
used when $\lim_{n \to \infty} f_n / g_n = 0$, and $f_n =
\omega(g_n)$ when $g_n = o(f_n)$. Throughout, we use $c, C$ to denote universal
positive constants, and their values may change from line to line. Finally, we
use the
symbols $c_d, C_d$ to
denote $d$-dependent constants; once again, their values will typically be
different in each instantiation. All logarithms are to the natural base unless
otherwise stated. We denote by $\NORMAL(\mu, \sigma^2)$ a normal distribution with mean $\mu$ and variance $\sigma^2$. We use $\BER(p)$ to denote a Bernoulli distribution with success probability $p$, and denote by $\BIN(n, p)$ a binomial distribution with $n$ trials and success probability $p$. We let $\Hyp(n, N, K)$ denote a hypergeometric distribution with $n$ trials, a universe of size $N$, and $K$ defectives\footnote{Recall that a hypergeometric random variable is formed as follows: Suppose that there is a universe of $N$ items containing $K$ defective items. Then $\Hyp(n, N, K)$ is the distribution of the number of defective items in a collection of $n$ items drawn uniformly at random, without replacement from the universe.}. Finally, we denote the total variation distance between two distributions $\mu$ and $\nu$ by $\TV(\mu, \nu)$.


\section{Main results} \label{sec:main-results}

We begin by characterizing the fundamental limits of
estimation and adaptation, and then turn to developing an estimator that
achieves these limits. Finally, we analyze variants of existing estimators from
this point of view.

\subsection{Fundamental limits of estimation} \label{sec:funlims}
In this subsection, our focus is on the fundamental limits of
estimation over various parameter spaces without imposing any computational
constraints on our procedures. 
We begin by characterizing the minimax risk over the
class of bounded, coordinate-wise isotonic tensors with unknown permutations. 
\begin{proposition} \label{prop:full-funlim}
There is a universal positive constant $C$ such that for each $d \geq 2$, 
\begin{align} \label{eq:minimax-lb}
c_d \cdot n^{-1/d} \leq \inf_{\thetahat \in \widehat{\Theta}} \; \sup_{\thetastar \in \Mspace_{\perm} 
(\lattice_{d, n}) \cap
\infball(1)}
\Rspace_n (\thetahat, \thetastar) \leq C \cdot n^{-1/d} \log^{2} n,
\end{align}
where $c_d > 0$ depends on $d$ alone.
\end{proposition}
The lower bound on the minimax risk in equation~\eqref{eq:minimax-lb} follows
immediately from known results on estimating bounded monotone functions on
the lattice without unknown permutations~\citep{han2019,deng2018isotonic}. These
results show that one can take $c_d \asymp (d-1)^{-(d-1)}$, but
the dependence of this constant on $d$ can likely be improved.

The upper bound is our main contribution to the proposition,
and is achieved by the bounded least squares estimator
\begin{align} \label{eq:def-blse}
\thetahatblse \defn \thetahatlse(\Mspace_{\perm}\bigl(\lattice_{d, n}) \cap \infball(1), Y\bigr).
\end{align}
In fact, the risk of $\thetahatblse$ can be expressed as a sum of two terms: 
\begin{align} \label{eq:two-terms-blse}
\Rspace_n(\thetahatblse, \thetastar) \leq C \bigl( n^{-1/d} \log^{2} n + n^{-(1 - 1/d)} \log n\bigr).
\end{align}
The first term corresponds to the error of estimating the unknown isotonic
function, and the second to the price paid for having unknown permutations.
Such a characterization was known in the case $d = 2$
\citep{ShaBalGunWai17,mao2018towards}, and our result shows that a similar
decomposition holds even for larger $d$. Note that for all $d \geq 2$, the first term of
equation~\eqref{eq:two-terms-blse} dominates the bound, and this is what leads to
Proposition~\ref{prop:full-funlim}.

Although the bounded LSE~\eqref{eq:def-blse} achieves the worst case risk
\eqref{eq:minimax-lb}, we may use its analysis as a vehicle for 
obtaining risk bounds for the vanilla least squares estimator without
imposing any boundedness constraints. This results in the following
proposition.

\begin{proposition} \label{cor:lse}
There is a universal positive constant $C$ such that for each $d \geq 2$: \\
(a) The least squares estimator over the set $\Mspace_{\perm}(\lattice_{d, n})$ has worst case risk bounded as
\begin{subequations}
\begin{align}
\sup_{\thetastar \in \Mspace_{\perm}(\lattice_{d, n}) \cap \infball(1)}
\Rspace_n\bigl(\thetahatlse(\Mspace_{\perm}(\lattice_{d, n}), Y) ,
\thetastar\bigr) \leq C n^{-1/d} \log^{5/2} n.
\end{align}
\noindent (b) The isotonic least squares estimator over $\Mspace(\lattice_{d, n})$ has worst case risk bounded as
\begin{align} \label{eq:iso-lse-sharp-bd}
\sup_{\thetastar \in \Mspace(\lattice_{d, n}) \cap \infball(1)} \Rspace_n\bigl
(\thetahatlse(\Mspace(\lattice_{d, n}), Y), \thetastar\bigr) \leq C n^{-1/d} \log^{5/2} n.
\end{align}
\end{subequations}
\end{proposition}
Part (a) of Proposition~\ref{cor:lse} deals with the LSE computed over the entire
set $\Mspace_{\perm} (\lattice_{d, n})$, and appears to be new even when $d =
2$; to the best of our knowledge, prior work~\citep{ShaBalGunWai17,mao2018towards}
 has only considered the bounded LSE~$\thetahatblse$~\eqref{eq:def-blse}. 
Part (b) of
Proposition~\ref{cor:lse},
on the other hand, provides a risk for the vanilla isotonic least squares
estimator when estimating functions in the set $\Mspace(\lattice_{d, n}) \cap
\infball(1)$. This estimator has a long history in both the statistics and
computer science communities
\citep{robertson1975consistency,dykstra1982algorithm,stout2015isotonic,kyng2015fast,ChaGunSen18,han2019};
unlike the other estimators considered so far, the isotonic LSE is
the solution to a convex optimization problem and can be computed in time
polynomial in $n$. Bounds on the worst case risk of this estimator are also
known: results for $d = 1$ are classical (see, e.g.,
\citet{brunk1955maximum,nemirovski1985convergence,Zha02}); when $d =
2$, risk
bounds were derived by~\citet{ChaGunSen18}; and the general
case $d \geq 2$ was considered by~\citet{han2019}.
Proposition~\ref{cor:lse}(b) improves the
logarithmic factor in the latter two papers from $\log^4 n$
to $\log^{5/2} n$, and is
obtained via a different proof
technique involving a truncation argument.

Two other comments are worth making.
First, it should be noted that there are other estimators for tensors in the
set $\Mspace(\lattice_{d, n}) \cap \infball(1)$ besides the isotonic LSE.
The block-isotonic estimator of~\citet{deng2018isotonic}, first
proposed by~\citet{fokianos2017integrated}, enjoys
a risk bound of the order $C_d \cdot n^{-1/d}$ for all $d \geq 2$, where $C_d >
0$ is a $d$-dependent constant. This eliminates the logarithmic factor
entirely, and matches the minimax lower bound up to a
$d$-dependent constant. In addition, the block-isotonic estimator also enjoys
significantly better adaptation properties than the isotonic LSE. On the other hand, the best known algorithm to compute
the block-isotonic estimator takes time $\order(n^3)$, while the isotonic LSE
can be computed in time $\ordertil(n^{3/2})$~\citep{kyng2015fast}. 

Second, we note that when the design is random in the setting without unknown
permutations~\citet[Theorem
3.6]{han2019global}
improves, at the expense of a $d$-dependent constant, the logarithmic
factors in the risk bounds of prior work~\citep{han2019}. His proof techniques
are based on the concentration of empirical processes on
upper and lower sets of $[0,1]^d$, and do not apply to the lattice setting
considered here. On the other hand, our proof works on
the event on which the LSE is suitably bounded, and is not immediately
applicable to the
random design setting. Both of these techniques should be viewed as particular 
ways of establishing the optimality of global empirical risk minimization
procedures even when the entropy integral for the corresponding function class
diverges; this runs contrary to previous heuristic beliefs about the
suboptimality of
these procedures (see, e.g.,~\citet{birge1993rates},~\citet[pp. 121--122]
{van2000applications},~\citet{kim2016global},
\citet{rakhlin2017empirical}, and~\citet{han2019global} for further
discussion).

Let us now turn to establishing the fundamental limits of estimation over the
class $\Mspace^{\kfull, \sset}_{\perm} (\lattice_{d, n})$. The following proposition
characterizes the minimax risk $\minimax_{d, n}(\kfull, \sset)$. Recall that \mbox{$s = \prod_{j = 1}^d s_j$} and $k^* = \min_{j \in [d]} \max_{\ell \in 
[s_j]}
k^j_{\ell}$.

\begin{proposition} \label{prop:adapt-funlim}
There is a pair of universal positive constants $(c, C)$ such
that for each $d \geq 1$, $\sset \in \lattice_{d, n}$, and
$\kfull \in \Kfull_{\sset}$, the minimax risk $\minimax_{d, n}
(\kfull, \sset)$ satisfies
\begin{align} \label{eq:set-adapt}
\frac{c}{n} \cdot \Big( s + (n_1 - k^*) \Big) \leq
\;
\inf_{\thetahat \in \widehat{\Theta}} \;
\sup_
{\thetastar
\in \Mspace^{\kfull, \sset}_{\perm}(\lattice_
{d,
n})} \; \Rspace_n (\thetahat, \thetastar)
 \; \leq \frac{C}{n} \cdot \Big( s + (n_1 - k^*) \log n
   \Big).
\end{align} 
\end{proposition}
A few comments are in order. As before, the risk can be decomposed into
two terms: the first term represents the \emph{parametric} rate of
estimating a tensor with $s$ constant pieces, and the second term is the
price paid for unknown permutations. When the parameter space is also bounded
in $\ell_\infty$-norm, such a
decomposition does not occur transparently even in the special case $d = 2$
\citep{ShaBalWai16-2}. Also note that when $s = \order(1)$ and $n_1 -
k^* = \omega(1)$, the second term of the bound~\eqref{eq:set-adapt}
dominates and the minimax risk is no longer of the parametric form~$s / n$.
This is in sharp contrast to isotonic regression without unknown permutations,
where there are estimators that achieve the parametric risk up to
poly-logarithmic factors~\citep{deng2018isotonic}. Thus, the fundamental
adaptation behavior that we expect changes significantly in the presence of
unknown permutations. 

Second, note that when $s_j = n_1$ for all $j \in [d]$, we have $\Mspace^{\kfull,
\sset}_{\perm}(\lattice_{d, n}) = \Mspace_{\perm} (\lattice_{d, n})$, in which
case the result above shows that consistent
estimation is impossible over the set of all isotonic tensors with unknown
permutations. This does \emph{not} contradict Proposition~\ref{prop:full-funlim},
since Proposition~\ref{prop:adapt-funlim} computes the minimax risk over
isotonic tensors without imposing boundedness
constraints.

Finally, we note that Proposition~\ref{prop:adapt-funlim} yields the following
corollary in the setting where we do not have unknown permutations. With a
slight abuse of notation, we let 
\begin{align*}
\Mspace^s(\lattice_{d, n}) \defn
\bigcup_{ \sset \; : \; \prod_{j = 1}^d s_j = s} \;\;  \bigcup_{\kfull
\in
\Kfull_{\sset}} \;\; \Mspace^
{\kfull, \sset}(\lattice_{d, n})
\end{align*}
denote the set of all coordinate-wise monotone tensors that are piecewise
constant on a $d$-dimensional partition having $s$ pieces.

\begin{corollary} \label{cor:adapt-noperms}
There is a pair of universal positive constants $(c, C)$ such that for each $d
\geq 1$, the following statements hold. \\
(a) For each $\sset \in \lattice_{d, n}$ and $\kfull \in \Kfull_{\sset}$, we
have
\begin{subequations}
\begin{align} \label{eq:set-adapt-noperms}
c \cdot \frac{s}{n}  \leq
\; \inf_{\thetahat \in \Thetahat} \; \sup_{\thetastar \in \Mspace^{\kfull,
\sset}
(\lattice_{d, n})}
\Rspace_n(\thetahat, \thetastar) \; \leq C \cdot \frac{s}{n}.
\end{align} 
(b) For each $s \in [n]$, we have
\begin{align} \label{eq:set-adapt-noperms-s}
c \cdot \frac{s}{n} \leq
\; \inf_{\thetahat \in \Thetahat} \; \sup_{\thetastar \in \Mspace^s (\lattice_
{d,
n})}
\Rspace_n(\thetahat, \thetastar) \; \leq C \cdot \frac{s \log n}{n}.
\end{align} 
\end{subequations}
\end{corollary}
Let us interpret this corollary in the context of known results. When $d = 1$
and there are no permutations,~\citet{bellec2015sharp} 
established minimax lower bounds of order $s / n$ and upper bounds of the
order $s \log n / n$ for estimating $s$-piece monotone
functions, and the bound~\eqref{eq:set-adapt-noperms-s} recovers this result.
The problem of estimating a univariate isotonic vector with $s$ pieces was also
considered by~\citet{gao2020estimation}, who showed a rate-optimal
characterization of the minimax risk that exhibits an
iterated logarithmic factor in the sample size whenever $s \geq 3$. When $d \geq
2$, however, the results of Corollary~\ref{cor:adapt-noperms} are new to the best
of our knowledge. 

The fundamental limits of estimation over the class $\Mspace^
{\kfull, \sset}_{\perm}(\lattice_{d, n})$ in Proposition~\ref{prop:adapt-funlim}
will allow us to assess the adaptivity
index of particular estimators. Before we do that, however, we
establish a baseline for adaptation by proving a lower bound on the adaptivity
index of polynomial time estimators.

\subsection{Lower bounds on polynomial time adaptation}

We now turn to our average-case reduction showing that any computationally efficient estimator 
cannot have a small adaptivity index.
Our primitive is the hypergraph planted clique conjecture $\HPC_D$, which is a
hypergraph extension of the planted clique conjecture.
Let us introduce the testing, or ``detection", version of this conjecture.
Denote the
set of
$D$-uniform
hypergraphs on the vertex set $[N]$ (hypergraphs in which each
hyperedge is incident on $D$ vertices) by $\HG_{D, N}$. Define, via
their
generative models, the random hypergraphs:
\begin{enumerate}
  \item $\mathcal{G}_D(N, p)$: Generate each hyperedge independently with
  probability $p$, and
  \item $\mathcal{G}_D(N, p; K)$: Choose $K \geq D$ vertices uniformly
  at random and
  form a clique, adding all~$\binom{K}{D}$ possible hyperedges between them.
  Add each remaining hyperedge independently with probability~$p$.
\end{enumerate}
Given an instantiation of a random hypergraph $G \in \HG_{D, N}$, the testing
problem is to distinguish
the
hypotheses
$H_0: G \sim \mathcal{G}_D(N, p)$ and $H_1: G \sim \mathcal{G}_D(N, p; K)$.
The error of any test $\psi_N: \HG_{D, N} \mapsto \{ 0, 1\}$ is given by
\begin{align} \label{eq:test-error}
\mathcal{E}(\psi_N) \defn \frac{1}{2} \; \EE_{H_0} \left[ \psi_N(G) \right] + \frac{1}
{2} \;
\EE_{H_1} \left[ 1 -
\psi_N(G) \right].
\end{align}

\begin{conjecture}[$\HPC_D$ conjecture] \label{conj:HPC}
Suppose that $p = 1/2$, and that $D \geq 2$ is a fixed integer. If
\begin{align*}
\limsup_{N \to \infty} \; \frac{\log K}{\log \sqrt{N}} < 1,
\end{align*}
then for any sequence of tests $\{\psi_N \}_{N \geq 1}$ such that $\psi_N$ is computable in time 
polynomial in $N^D$, we have
\begin{align*}
\liminf_{N \to \infty} \; \Espace( \psi_N) \geq 1/2.
\end{align*}
\end{conjecture}
Note that when $D = 2$, Conjecture~\ref{conj:HPC} is equivalent to the planted clique
conjecture, which is a widely believed conjecture in average-case
complexity~\citep{jerrum1992large,feige2003probable,barak2019nearly}.
The $\HPC_3$ conjecture was used by~\citet{zhang2018tensor} to show
statistical-computational gaps for third
order tensor completion; their evidence for the validity of this conjecture was
based on the threshold at which the natural spectral method for the problem
fails. In a recent paper on the general case $D \geq 3$,~\citet{luo2020tensor} showed that MCMC algorithms and
methods based on low-degree polynomials---see~\citet{hopkins2018statistical,kunisky2019notes} and the references
therein for an introduction to such methods, which comprise a large family of popular algorithms---also fail at this threshold. In concurrent work to 
that of Luo and Zhang,
\citet{brennan2020reducibility} showed that the planted clique conjecture with ``secret leakage" can
be reduced to $\HPC_D$. 
Recall
our definition of the
adaptivity index~\eqref{eq:global-AI};
the $\HPC_D$ conjecture implies the following computational lower bound.

\begin{theorem} \label{thm:hardness}
Suppose that Conjecture~\ref{conj:HPC} 
holds, and that $d \geq 2$ is a fixed integer.
Then for any sequence of estimators $\{ \thetahat_n \}_{n \geq 1}$
such that $\thetahat_n$ is computable in time polynomial in $n$, we have
\begin{align} \label{eq:adapt-lb}
\liminf_{n \to \infty} \frac{\log \adapt(\thetahat_n)}{ \log n^{ \frac{1}{2}\big( 1 - 
1/d \big)} }  \geq 1.
\end{align} 
\end{theorem}
Assuming Conjecture~\ref{conj:HPC}, Theorem~\ref{thm:hardness} thus posits that the
adaptivity index of any computationally efficient estimator must grow at least at rate $n^{\frac{1}{2}\left( 1 - 
\frac{1}{d} \right)}$, up to sub-polynomial factors in $n$.  
In particular, this precludes the existence of efficient estimators with adaptivity index bounded
poly-logarithmically in $n$. Contrast this with the case of isotonic regression
without unknown
permutations, where the block-isotonic estimator has adaptivity index\footnote{\citet{deng2018isotonic} consider the more general case
where the hyper-rectangular partition need not be consistent with the Cartesian
product of one-dimensional partitions, but the adaptivity index claimed
here can be obtained as a straightforward corollary of their results.} of the
order $\order(\log^d n )$~\citep{deng2018isotonic}. This demonstrates yet
another salient difference in adaptation behavior with and without unknown
permutations.

Finally, while Theorem~\ref{thm:hardness} is novel for all $d \geq 3$, we
note that when $d = 2$,~\citet{ShaBalWai16-2} established a computational lower bound
for the case where the noise distribution is Bernoulli and the indifference sets
are identical along all the dimensions.
On the other hand, Theorem~\ref{thm:hardness} applies in the case where the indifference
sets induced by the univariate partitions may be different along the different dimensions, and also to the case of Gaussian noise. The
latter, technical reduction is accomplished via the machinery of Gaussian rejection kernels
introduced by~\citet{brennan2018reducibility}. This device shares
many similarities with other reduction ``gadgets" used in earlier arguments
(e.g.,~\citet{berthet2013optimal,ma2015computational,wang2016statistical}).

We have thus established both the fundamental limits of estimation without
computational considerations~\eqref{eq:minimax-lb}, and a lower bound on the adaptivity index
of
polynomial time estimators~\eqref{eq:adapt-lb}. 
Next, we show that a simple, efficient
estimator simultaneously
attains both lower bounds  for all~$d \geq 3$.

\subsection{Achieving the fundamental limits in polynomial time} \label{sec:alg}

We begin with notation that will be useful in defining our estimator.
We say that a tuple $\bl = (S_1, \ldots, S_L)$ is a
\emph{one-dimensional ordered partition} of the set $[n_1]$ of size $L$
if the sets $S_1, \ldots, S_L \subseteq [n_1]$ are non-empty and pairwise disjoint, with $[n_1] = \bigcup_
{\ell =
  1}^L S_\ell$.  Note that any such one-dimensional ordered partition induces a partial order, which we denote by $\prec$, on the set $[n_1]$; to be specific, the induced partial order is such that for $a,b \in [n_1]$, we write $a \prec b$ if $a \in S_\ell$ and $b \in S_{\ell'}$ with $\ell < \ell'$.  Furthermore, each $S_{\ell}$, $\ell = 1, \ldots, L$ is an \emph{antichain} of this partial order\footnote{Recall that a subset of a partially ordered set is an antichain if no two elements in the subset are comparable with each other in the partial order.}.  As a concrete example,
suppose that $n_1 = 6$; then the one-dimensional ordered partition $\bl = (\{2, 4, 6\}, \{1,5\}, \{3\})$ induces a partial order on $[6]$ with the set of binary relations
\[
\{ 2 \prec 1, \ 2 \prec 5, \ 2 \prec 3, \ 4 \prec 1, \ 4 \prec 5, \ 4 \prec  3, \ 6\prec 1, \ 6\prec 5, \ 6\prec 3, \ 1\prec 3, \ 5\prec 3 \}.
\] 
The antichains of this partial order are indeed $\{2, 4, 6\}$, $\{1,5\}$, and $\{3\}$.

Denote the set of
all one-dimensional ordered
 partitions of size
$L$ by $\Partition_L$, and let $\Partition \defn \bigcup_{L = 1}^{n_1} 
\Partition_L$. 
Note that any one-dimensional ordered  partition of size $L$ induces a map $\sigma_
{\bl}: [n_1] \to [L]$,
where 
$\sigma_{\bl}(i)$ is the index~$\ell$ of the set $S_\ell \ni i$. In the example above, we have
$\sigma_{\bl}(1) = \sigma_{\bl}(5) = 2$ and $\sigma_{\bl}(3) = 3$.
Now given $d$ ordered partitions $\bl_1, \ldots, \bl_d \in \Partition$,
define 
\begin{align*}
\Mspace (\lattice_{d, n}; \bl_1, \ldots, \bl_d) &\defn \biggl\{
\theta \in \Tspace_{d, n}: \;\; \exists f \in \Fspace_d \; \text{ such that } \forall i_j \in [n_j], \; j \in [d], \\
& \;\; \qquad \quad \qquad \qquad \theta(i_1, \ldots, i_d) = f \left( 
\frac{\sigma_{\bl_1} (i_1)}{n_1}, \ldots, \frac{\sigma_{\bl_d} (i_d)}{n_d}
\right)  \biggr\}.
\end{align*}
In other words, the set\footnote{Note that we have abused notation slightly in defining
the sets $\Mspace (\lattice_{d, n}; \bl_1, \ldots, \bl_d)$ and $\Mspace 
(\lattice_{d, n}; \pi_1, \ldots, \pi_d)$ similarly to each other. The
reader should be able to disambiguate the two
from context, depending on whether the arguments are ordered partitions or
permutations.}
 $\Mspace (\lattice_{d, n}; \bl_1, \ldots, \bl_d)$
represents all tensors that are piecewise constant on the hyper-rectangles\footnote{Note that
$\prod_{j = 1}^d \bl_j = \{ \prod_{j = 1}^d S_j \mid S_j \in \bl_j, j \in [d] \}$.}
$\prod_{j = 1}^d \bl_j$, while also being coordinate-wise isotonic on the
partial orders specified
by $\bl_1, \ldots, \bl_d$. We refer to any
such
hyper-rectangular partition of the lattice
$\lattice_{d, n}$ that can be written in the form $\prod_{j = 1}^d \bl_j$ as
a \emph{$d$-dimensional ordered partition}.

Our estimator computes various statistics of the observation tensor $Y$, and we
require some more terminology to define these precisely. For each $j \in [d]$,
define the vector $\tauhat_j \in \real^{n_j}$
of
``scores'', whose $k$-th entry is given by
\begin{subequations} \label{eq:scores}
\begin{align}  \label{eq:est-scores}
\tauhat_j(k) \defn \sum_{i_1, \ldots, i_d \in [n_1]} Y(i_1, \ldots, i_d) \cdot
\ind{i_j = k }.
\end{align}
The score vector $\tauhat_j$ provides noisy information about the permutation
$\pistar_j$. In order to see this clearly, it is helpful to specialize to the
noiseless case $Y = \thetastar$, in which case we obtain the population scores
\begin{align} \label{eq:pop-scores}
\taustar_j(k) \defn \sum_{i_1, \ldots, i_d \in [n_1]} \thetastar(i_1, \ldots,
i_d)
\cdot
\ind{i_j = k }.
\end{align}
\end{subequations}
One can
verify that the entries of the vector $\taustar_j$ are increasing when viewed
along permutation~$\pistar_j$, i.e., that $\taustar_j(\pistar_j(1)) \leq \cdots
\leq \taustar_j(\pistar_j(n_j))$.

For each pair $k, \ell \in [n_j]$, also define the pairwise statistics
\begin{subequations} \label{eq:pair-stats}
\begin{align}
\Deltahat^{{\sf sum}}_j (k, \ell) &\defn \tauhat_j(\ell) - \tauhat_j(k) \quad 
\text{and
} \label{eq:score-diff} \\
\Deltahat^{\max}_j (k, \ell) &\defn \max_{(i_q)_{q \neq j} \in \prod_{q \neq j} [n_q]}
\left\{ Y
(i_1, \ldots, i_{j-1}, \ell, i_{j+1}, \ldots, i_d) - Y(i_1, \ldots, i_{j-1}, k, i_{j+1}, \ldots, i_d) \right\}. \label{eq:entry-diff}
\end{align}
\end{subequations}
Given that the scores provide noisy information about the unknown permutation,
the statistic $\Deltahat^{{\sf sum}}_j(k, \ell)$ provides noisy information about
the
event $\{ \pistar_j(k) < \pistar_j(\ell) \}$, i.e., a large positive value of
$\Deltahat^
{{\sf sum}}_j(k, \ell)$ provides evidence that $\pistar_j(k) < \pistar_j(\ell)$
and a
large
negative value indicates otherwise. Now clearly, the scores are not the sole
carriers of information about the unknown permutations; for instance, the
statistic
$\Deltahat^{\max}_j (k, \ell)$ measures the maximum difference between
\emph{individual} entries and a large, positive value of this statistic once
again
indicates that $\pistar_j(k) < \pistar_j(\ell)$. The
statistics~\eqref{eq:pair-stats} thus allow us to
distinguish pairs of indices, and our algorithm is based on precisely this
observation. 
Finally, recall that similarly to 
before, one may define an antichain of a directed acyclic
graph: for any pair of nodes in the antichain, there is no
directed path in the graph going from one node to the other.

Having set up the necessary notation, we are now ready 
to describe the algorithm formally.


\noindent \hrulefill

\noindent \underline{\bf Algorithm: Mirsky partition estimator}
\begin{enumerate}
\item[I.] (Partition estimation): For each $j \in [d]$, perform the following
steps:
\begin{enumerate}
  \item[a.] Create a directed graph $G'_j$ with vertex set $[n_j]$ and add the
  edge $u \to v$ if either
  \begin{subequations}
  \begin{align} \label{eq:cluster-step1}
  \Deltahat^{{\sf sum}}_j (u, v) > 8 \sqrt{\log n} \cdot n^{\frac{1}{2}(1 -
  1/d)} \quad \text{ or } \quad \Deltahat^{\max}_j (u, v) > 8 \sqrt{\log
  n}.
  \end{align}
  If $G'_j$ has cycles, then prune the graph and only keep the edges
  corresponding to the first
  condition above, i.e., 
  \begin{align} \label{eq:cluster-step2}
  u \to v \qquad \text{ if and only if } \qquad \Deltahat^{{\sf sum}}_j (u, v) > 8 
  \sqrt{\log n} \cdot n^{\frac{1}{2}(1 -
  1/d)}. 
  \end{align}
  \end{subequations}
  Let $G_j$ denote the pruned graph.
  \item[b.] Compute a one-dimensional ordered partition $\blhat_j$ as the minimal partition of the
  vertices of $G_j$ into disjoint antichains, via Mirsky's algorithm
  \citep{mirsky1971dual}. 
\end{enumerate}
\item[II.] (Piecewise constant isotonic regression): 
Project the observations on the set of isotonic functions that are consistent
with the blocking obtained in step I to obtain
\begin{align*}
\thetahatblock = \argmin_{\theta \in \Mspace (\lattice_{d, n}; \blhat_1,
\ldots, \blhat_d)} \ell_n^2 (Y, \theta).
\end{align*}
\end{enumerate}
\noindent \hrulefill
\medskip

\noindent Some discussion of the pruning step is in order. 
Note that at the end of step Ia, the graph $G_j$ is guaranteed to have
no cycles, since the pruning step is based exclusively on the score vector
$\tauhat_j$. The purpose of the pruning step is precisely to accomplish this,
and there are other heuristics that might also work.
For instance, if the graph $G'_j$ consists of disjoint cycles, then the pruning step
can instead proceed by pruning each cycle individually. As will be make clear in the proof, the probability that the graph $G'_j$ is pruned is vanishingly small, and so the exact mechanics of the pruning step are not crucial to the algorithm. 

Owing to its acyclic structure, the vertices of graph $G_j$ can always be
decomposed as the union of disjoint
antichains, since a directed acyclic graph defines a partial order on its vertices in the natural way.
The presence of an edge $u \to v$ indicates that $u \succ v$ in the partial order, and the acyclic nature of the 
graph ensures that there are no inconsistencies.

Let us now describe the intuition
behind the estimator as a whole. On
each dimension $j$, we produce
a partial order on the set $[n_j]$. We employ the statistics
\eqref{eq:pair-stats} in order to determine such a partial order, with two
indices placed in the same block if they cannot be distinguished based on
these statistics. This partitioning step serves a dual purpose: first, it
discourages us from committing to orderings over indices when our observations
on these indices look similar, and second, it serves to cluster indices that
belong to the same indifference set, since the statistics~\eqref{eq:pair-stats}
computed on pairs of indices lying in the same indifference set are likely to
have small magnitudes.
Once we
have determined the partial order via Mirsky's algorithm, we project our
observations onto isotonic tensors that are piecewise constant on the
$d$-dimensional partition specified by the individual partial orders.
We note that the Mirsky partition estimator presented here derives some inspiration from 
existing estimators. For instance, the idea of associating a
partial order
with the indices has appeared before~\citep{pananjady2020worst,mao2018towards},
and variants of the pairwise statistics~\eqref{eq:pair-stats} have been used in
prior work on permutation estimation~\citep{FlaMaoRig16,mao2018towards}.
However, to the best of our knowledge, no existing estimator computes a
partition of the indices into antichains: a natural idea that significantly
simplifies both the algorithm---speeding it up considerably when there are a
small number of indifference sets (see the following paragraph for a
discussion)---and its analysis.

We now turn to a discussion of the computational complexity of this
estimator. Suppose that we compute the score vectors $\tauhat_j, j \in [d]$ 
first, which takes $\order(dn)$ operations. Now for each $j \in [d]$,
step I of the algorithm can be computed in time $\order(n_j^2)$, since it takes
$\order(n_j^2)$ operations to form the graph $G_j$, and Mirsky's algorithm
\citep{mirsky1971dual} for the
computation of
a ``dual Dilworth" decomposition into antichains runs in time $\order(n_j^2)$.
Thus, the total
computational
complexity of step I is given by $\order(d \cdot n_1^2)$.
Step II of the algorithm involves an isotonic projection onto a partially
ordered set. As we establish in Lemma~\ref{lem:composition},
such a projection can be computed by first averaging the entries of $Y$ on the
hyper-rectangular blocks formed by the $d$-dimensional ordered partition
$\prod_
{j = 1}^d \blhat_j$, and then
performing multivariate isotonic regression on the result. The first operation
takes linear time $\order(n)$, and the second operation is a
weighted isotonic regression problem that can be computed in time $\ordertil
(B^{3/2})$ if there are $B$ blocks in the $d$-dimensional ordered 
partition~\citep{kyng2015fast}. Now clearly, $B \leq n$, so that step II of
the Mirsky partition estimator has
worst-case complexity $\ordertil(n^{3/2})$. Thus, the overall estimator (from
start to finish) has
worst-case complexity $\ordertil(n^{3/2})$.
Furthermore, we show in Lemma
\ref{lem:block-structure} that if $\thetastar \in \Mspace^
{\kfull, \sset}_
{\perm}(\lattice_{d, n})$, then $B \leq s$ with high probability, and on this
event, step
II only takes time $\order(n) + \ordertil(s^{3/2})$. When $s$ is small, the
overall complexity of the Mirsky partition procedure is therefore
dominated by that of computing the scores, and given by $\order(d n)$ with
high probability.
Thus, the computational complexity also adapts
to underlying structure.

Having discussed its algorithmic properties, let us now turn to the risk bounds
enjoyed by the Mirsky partition estimator. Recall, once again, the notation $k^*
= \min_{j \in [d]} k^j_{\max}$.

\begin{theorem} \label{thm:block-risk}
There is a universal positive constant $C$ such that for all $d
\geq 2$: \\
(a) We have the worst-case risk bound
\begin{align} \label{eq:block-worst-risk}
\sup_{\thetastar \in \Mspace_{\perm}  (\lattice_{d, n}) \cap \infball(1)} \;
\Rspace_n (\thetahatblock, \thetastar) \; \leq C \left\{ n^{-1/d} \log^{5/2} n +
d^2
n^{-\frac{1}{2}(1 - 1/d)} \log n \right\}.
\end{align}

\noindent (b) 
We have the adaptive bound
\begin{subequations}
\begin{align} 
\sup_
{\thetastar
\in \Mspace^{\kfull, \sset}_{\perm}(\lattice_
{d,
n})} \; \Rspace_n (\thetahatblock, \thetastar) \leq \frac{C}{n} \left
\{s + d^2
(n_1 - k^*) \cdot n^{\frac{1}{2}\left( 1 - \frac{1}{d}
\right)} \right\} \log n. \label{eq:set-adapt-ub}
\end{align}
Consequently, the estimator $\thetahatblock$ has adaptivity index bounded as
\begin{align} \label{eq:set-AI-ub}
\adapt ( \thetahatblock ) \leq C d^2 \cdot n^{\frac{1}{2}\left( 1 - 
\frac{1}{d} \right)} \log n.
\end{align}
\end{subequations}
\end{theorem}
When taken together, the two parts of Theorem~\ref{thm:block-risk} characterize
both the risk and adaptation behaviors of the Mirsky partition estimator
$\thetahatblock$. Let us discuss some particular consequences of these results,
starting with
part (a) of the theorem. When $d = 2$, we see that the second term of equation
\eqref{eq:block-worst-risk} dominates the bound, leading to a risk of order $n^
{-1/4}$. Comparing with the minimax lower bound~\eqref{eq:minimax-lb}, we
see that this is sub-optimal by a factor $n^{1/4}$. There are other
estimators that attain strictly better rates~\citep{mao2018towards,liu2020better}, but to the
best of our knowledge, it is not yet known whether the minimax lower
bound~\eqref{eq:minimax-lb} can be attained by an estimator that is computable
in polynomial time. On the other hand, for $d \geq 3$, the first
term of equation~\eqref{eq:block-worst-risk} dominates, and we achieve the 
lower bound on the minimax risk~\eqref{eq:minimax-lb} up to a
poly-logarithmic factor. Thus, the case $d \geq 3$ of this problem is distinctly
different from the bivariate case: The minimax risk is achievable with a
computationally efficient algorithm in spite of the fact that there are more
permutations to estimate in higher dimensions. This
surprising behavior can be reconciled with prevailing intuition by two
high-level
observations. First, as $d$ grows, the isotonic function becomes much harder to
estimate, so we are able to tolerate more sub-optimality in estimating the permutations. Second, in higher dimensional problems, a single permutation perturbs large blocks of the tensor, and this allows us to obtain more information about it than when $d = 2$.
Both of these observations are made quantitative and precise in the proof.

As a side note, we believe that the logarithmic factor in the bound~\eqref{eq:block-worst-risk} 
can be improved; one way to do so is to use other isotonic
regression estimators (like the bounded LSE) in step II of our algorithm. But
since our notion of adaptation requires an estimator that
performs well even when the signal is unbounded, we have used the vanilla
isotonic LSE in step II. 
%
%

Turning our attention now to part (b) of the theorem, notice that we achieve
the lower bound~\eqref{eq:adapt-lb} on the adaptivity index of polynomial time
procedures up to a sub-polynomial factor in $n$. Such a result was not known, to
the best of our knowledge, for any $d \geq 3$. Even when $d = 2$, the
Count-Randomize-Least-Squares (CRL) estimator of~\citet{ShaBalWai16-2} was shown to have adaptivity index bounded by
$\ordertil(n^{1/4})$ over a sub-class 
of \emph{bounded} bivariate isotonic
matrices with unknown permutations
that are also piecewise constant on two-dimensional ordered  partitions
$\Mspace^{\kfull, \sset}_{\perm} 
(\lattice_{2, n}) \cap \infball(1)$. As we show in Proposition
\ref{prop:AI-MP-bounded} presented in Appendix
\ref{app:adapt}, the Mirsky partition estimator is also adaptive in
this case, and attains an adaptivity index that significantly improves upon
the best bound known for the CRL estimator in terms of the logarithmic factor. In particular,
the adaptivity index $n^{1/4} \log^8 n$ for the CRL estimator\footnote{To be
clear, this is the best known upper bound on the adaptivity index of the CRL
estimator due to~\citet{ShaBalWai16-2}. These results, in turn,
rely on the adaptation properties of the bivariate isotonic least squares
estimator~\citep{ChaGunSen18}, and it is not clear if they
can be improved substantially.} is improved to $n^{1/4} \log^{5/4} n$ for the
Mirsky partition estimator $\thetahatblock$ and further to $n^{1/4} \log n$ for
a bounded variant (see Remark~\ref{rem:bounded-MP}). Appendix~\ref{app:adapt} also establishes some other adaptation properties for a variant
of the CRL estimator in low dimensions.
An even starker difference between the adaptation properties of
the CRL and Mirsky partition estimators is evident in higher dimensions. We show in
Theorem~\ref{thm:adapt-lb} to follow that for
higher dimensional problems with $d \geq
4$, the CRL estimator has strictly sub-optimal
adaptivity index. Thus, in an overall sense, the Mirsky partition estimator
is better equipped to adapt to indifference set structure than the CRL
estimator.

Let us also briefly comment on the proof of part (b) of the theorem, which has
several components that are novel to the best of our knowledge. We begin by
employing a decomposition of the error of the estimator in terms of the sum of
estimation and approximation errors; while there are also compelling aspects to
our bound on the estimation error, let us showcase some interesting
components involved in bounding the approximation error. The first key
insight is a structural result (given as Lemma
\ref{lem:composition} in Appendix~\ref{app:iso}) that allows us to write step
II of the algorithm as a composition of two simpler steps. Besides having
algorithmic consequences (alluded to in our discussion of the
running time of the Mirsky partition estimator), Lemma~\ref{lem:composition} allows us
to write the approximation error as a sum of two terms corresponding to
the two simpler steps of this composition. In bounding these terms,
we make repeated use of a second key
component: Mirsky's algorithm groups the indices into clusters of disjoint
antichains, so our bound on the approximation error incurred on any single
block of the partition makes critical use of the condition
\eqref{eq:cluster-step1} used to accomplish this clustering.
Our final key component, which is absent from proofs in the literature to the
best of our knowledge, is to handle the approximation error on unbounded mean
tensors $\thetastar$, which is crucial to establishing that the bound
\eqref{eq:set-adapt-ub} holds in expectation---this is, in turn, necessary to
provide a bound on the adaptivity index. This component requires us to
leverage the pruning condition~\eqref{eq:cluster-step2} of the algorithm in
conjunction with some careful conditioning arguments.

Taking both parts of
Theorem~\ref{thm:block-risk} together, then, we have produced a
computationally efficient estimator that is both worst-case optimal when $d \geq
3$ and optimally
adaptive among the class of computationally efficient estimators. Let us now
turn to other
natural estimators
for this problem, and assess their worst-case risk, computation, and adaptation
properties.

\subsection{Adaptation properties of existing estimators}

\sloppy 
Arguably, the most natural estimator for this problem is the global least
squares estimator
$\thetahatlse(\Mspace_{\perm}(\lattice_{d, n}), Y)$, which corresponds to the
maximum likelihood estimator in our setting with Gaussian errors.
The worst-case risk behavior of the LSE over the set~\mbox{$\Mspace_{\perm}
(\lattice_{n, d}) \cap \infball(1)$} was already discussed in Proposition~\ref{cor:lse}(a):
It attains the minimax lower bound~\eqref{eq:minimax-lb} up to a
poly-logarithmic
factor.
However, computing such an
estimator is NP-hard in the worst-case even when $d = 2$, since the notoriously
difficult
max-clique instance can be straightforwardly reduced to the
corresponding quadratic assignment optimization problem (see, e.g.,
\citet{Pitsoulis2001} for reductions of this type).

Another class of procedures consists of two-step estimators
that first estimate the unknown permutations defining the model, and then the
underlying isotonic function. Estimators of this form abound in prior
work~\citep{ChaMuk16,ShaBalWai16-2,pananjady2020worst,mao2018towards,liu2020better}. 
We unify such estimators under Definition
\ref{def:pp} to follow, but first, let us consider a particular instance of such
an estimator in which the permutation-estimation step is given by
a multidimensional extension of the Borda or Copeland count. A close relative
of
such an estimator has been analyzed when $d = 2$~\citep{ChaMuk16}.

\medskip
\noindent \hrulefill

\noindent \underline{\bf Algorithm: Borda count estimator} 
\begin{enumerate}
\item[I.] (Permutation estimation): Recall the score vectors $\tauhat_1, \ldots,
\tauhat_d$ from~\eqref{eq:est-scores}.
Let $\pihatcount_j$ be any permutation along which the entries of $\tauhat_j$ are non-decreasing; i.e., 
\begin{align*}
\tauhat_j \bigl(\pihatcount_j (k)\bigr) \leq \tauhat_j \bigl(\pihatcount_j 
(\ell)\bigr) \text{ for all } 1 \leq k \leq \ell \leq n_j.
\end{align*}
\item[II.] (Isotonic regression): Project the observations onto the class of
isotonic tensors that are consistent with the permutations obtained in step I to
obtain
\begin{align*}
\thetahatplug \defn \argmin_{\theta \in \Mspace(\lattice_{d, n}; \pihatcount_1,
\ldots, \pihatcount_d)} \; \ell_n^2 (Y, \theta).
\end{align*}
\end{enumerate}
\noindent \hrulefill
\bigskip

The rationale behind the estimator is simple: If we were given the true
permutations $(\pistar_1, \ldots, \pistar_d)$, then performing isotonic
regression on the permuted observations $Y\{ (\pistar_1)^{-1}, \ldots,
(\pistar_d)^{-1} \}$ would be the most natural thing to do. Thus, a natural idea
is to \emph{plug-in} permutation estimates $(\pihatcount_1,
\ldots, \pihatcount_d)$ of the true permutations. The computational complexity of
this estimator is dominated by the isotonic regression step, and is thus given
by $\ordertil(n^{3/2})$~\citep{kyng2015fast}.
The following proposition
provides an upper bound on the worst-case risk of this estimator over bounded
tensors in the set $\Mspace_{\perm}(\lattice_{d, n})$.

\begin{proposition} \label{prop:borda-worst-case}
There is a universal positive constant $C$ such that for each $d \geq 2$, 
we
have
\begin{align} \label{eq:borda-worst-case-risk}
\sup_{\thetastar \in \Mspace_{\perm}  (\lattice_{d, n}) \cap \infball(1)} \;
\Rspace_n (\thetahatplug, \thetastar) \; \leq C \cdot \left( n^{-1/d} \log^{5/2}
n
+ d^2 n^{-\frac{1}{2}(1 - 1/d)} \right).
\end{align}
\end{proposition}
A few comments are in order. First, note that a variant of this estimator has
been analyzed previously in the case $d = 2$, but with the bounded isotonic LSE
in step II instead of the (unbounded) isotonic LSE~\citep{ChaMuk16}. When $d =
2$, the second term of equation~\eqref{eq:borda-worst-case-risk}
dominates the bound and Proposition~\ref{prop:borda-worst-case}
establishes the rate $n^{-1/4}$, without the logarithmic factor present in~\citet{ChaMuk16}. 

Second, note that when $d \geq 3$, the first term of equation
\eqref{eq:borda-worst-case-risk} dominates the bound, and comparing this bound
with the minimax lower bound~\eqref{eq:minimax-lb}, we see that the Borda count
estimator is minimax optimal up to a poly-logarithmic factor for all
$d \geq 3$. In this respect, it resembles both the full least squares
estimator~$\thetahatlse(\Mspace_
{\perm}(\lattice_{d, n}), Y)$ and the Mirsky partition estimator
$\thetahatblock$. 

Unlike the Mirsky partition estimator, however, both the global LSE and the
Borda count estimator are unable to
adapt optimally to indifference sets. This is a consequence of a more general
result that we state after the following definition.

\begin{definition}[Permutation-projection based estimator] \label{def:pp}
We say that an estimator $\thetahat$ is permutation-projection based if it can
be written as either
\begin{align*}
\thetahat = \argmin_{\theta \in \Mspace(\lattice_{d, n}; \pihat_1, \ldots,
\pihat_d)} \ell_n^2(Y, \theta) \quad \quad \quad \text{ or } \quad \quad \quad
\thetahat = \argmin_{\theta \in \Mspace(\lattice_{d, n}; \pihat_1, \ldots,
\pihat_d) \cap \infball(1)} \ell_n^2(Y, \theta)
\end{align*}
for a tuple of permutations $(\pihat_1, \ldots, \pihat_d)$. These permutations
may be chosen in a data-dependent fashion.
\end{definition}
The bounded LSE~\eqref{eq:def-blse}, the global LSE, and the Borda
count estimator are permutation-projection
based, as
is the CRL estimator of~\citet{ShaBalWai16-2}. The Mirsky partition
estimator, on the other hand, is not. The following theorem proves a lower
bound on the adaptivity index of any permutation-projection based estimator.

\begin{theorem} \label{thm:adapt-lb}
For each $d \geq 4$, there is a pair of constants $
(c_d, C_d)$ that depend
only on the dimension $d$ such that for each $n \geq C_d$ and any
permutation-projection based
estimator $\thetahat$, we have
\begin{align*}
\adapt(\thetahat) \geq c_d \cdot n^
{1 - 2/d}.
\end{align*}
\end{theorem}
For each $d \geq 4$, we have $n^{1 - 2/d} \gg n^{\frac{1}{2}(1 - 1/d)}$,
so by comparing Theorem~\ref{thm:adapt-lb} with Theorem
\ref{thm:hardness}, we see that no permutation-projection based estimator
can attain the smallest adaptivity index possible for polynomial time
algorithms. In fact, even the global LSE, which is not computable in polynomial
time
to the best of our knowledge,
falls short of the polynomial time benchmark of Theorem~\ref{thm:hardness}.

On the other hand, when $d = 2$, we note once again that~\citet{ShaBalWai16-2}
leveraged the favorable adaptation properties of the bivariate 
isotonic LSE~\citep{ChaGunSen18} to show that their CRL estimator has
the optimal adaptivity index for polynomial time algorithms over the class
$\Mspace^{\kfull, \sset}_{\perm}(\lattice_{2, n}) \cap \infball(1)$. They also
showed that the bounded LSE~\eqref{eq:def-blse} does not adapt optimally in this
case. In higher dimensions, however, even the
isotonic LSE---which must be employed within any permutation-projection based
estimator---has poor adaptation properties~\citep{han2019}, and this leads to our
lower bound in Theorem~\ref{thm:adapt-lb}.

The case $d = 3$ represents a transition between these two extremes, where the
isotonic LSE adapts sub-optimally, but a good enough adaptivity index is still
achievable owing to the lower bound of Theorem~\ref{thm:hardness}. Indeed, we
show in Proposition~\ref{prop:AI-CRL-unbounded} in Appendix~\ref{app:adapt} that
a variant of the CRL
estimator also attains the polynomial time optimal adaptivity index for this
case.
Consequently, a result as strong as Theorem~\ref{thm:adapt-lb}---valid
for all permutation-projection based estimators---cannot hold when $d = 3$.


\subsection{Adaptation of the Mirsky partition estimator in the bounded case} \label{sec:adapt-bounded-MP}

The careful reader would have noticed that our results on adaptation hold for
\emph{unbounded} signals, and as such, do not recover our minimax results in the bounded case
(see the discussion following Proposition~\ref{prop:adapt-funlim}). This raises the natural
question of whether one can show adaptation results for signals that are piecewise
constant on hyper-rectangles but also uniformly bounded.

In order to answer this question, let us first define the adaptivity index over a hierarchy of bounded
sets $\Mspace^{\kfull, \sset}_{\perm}(\lattice_{d, n}) \cap \infball(1)$. 
Begin by defining the minimax risk
\[
\overline{\minimax}_{d,n}(\kfull, \sset) \defn \inf_{\thetahat \in \Thetahat} \;
\sup_
{\thetastar \in \Mspace^{\kfull, \sset}_{\perm}(\lattice_{d, n}) \cap \infball
(1)} \Rspace_n (\thetahat, \thetastar).
\]
Now for an
estimator $\thetahat \in \Thetahat$, let
\begin{subequations} \label{eq:AI-unbounded}
\begin{align}
\overline{\adapt}^{\kfull, \sset}(\thetahat) &\defn \frac{ \sup_{\thetastar \in
\Mspace^{\kfull, \sset}_{\perm}(\lattice_{d, n}) \cap \infball(1)} \Rspace_n
(\thetahat, \thetastar) }{ \overline{\minimax}_{d,n}(\kfull, \sset) } \quad 
\text{
and } \\
\overline{\adapt}(\thetahat) &\defn \max_{\sset \in \lattice_{d, n}} \max_
{\kfull \in
\Kfull_{\sset}} \; \overline{\adapt}^{\kfull, \sset}(\thetahat).
\end{align}
\end{subequations}
With these definitions set up, we are now ready to state our main result of this
section: an adaptation result for the Mirsky partition estimator for bounded, two-dimensional
signals.

\begin{proposition} \label{prop:AI-MP-bounded}
Let $d = 2$. 
There is a universal positive constant
$C$ such that the Mirsky
partition estimator satisfies
\begin{align*}
\sup_{\thetastar \in \Mspace^{\kfull, \sset}_{\perm}(\lattice_{d, n}) \cap \infball(1) }
\Rspace_n
(\thetahatblock, \thetastar) \leq \frac{C}{n} \cdot (n_1 -
k^* +
1) \cdot n^
{1/4}
\log^{5/4} n.
\end{align*}
Consequently\footnote{The reason for this consequence is an existing minimax lower bound~\citep{ShaBalWai16-2}, and is made clear in the proof.}, we have
\begin{align*}
\overline{\adapt}(\thetahatblock) \leq C \cdot n^{1/4} \log^{5/4} n.
\end{align*}
\end{proposition}
Let us begin by comparing\footnote{When making this comparison, note the
differences
between our notation and theirs: we consider $n_1 \times n_1$ matrices with $n =
n_1^2$, while~\citet{ShaBalWai16-2} work with $n \times n$ matrices.}
Proposition~\ref{prop:AI-MP-bounded} to the results of~\citet{ShaBalWai16-2}.
Assuming the planted clique conjecture,~\citet[Theorem 3]{ShaBalWai16-2} show a lower bound on the (bounded) adaptivity index of any polynomial time procedure.
Proposition~\ref{prop:AI-MP-bounded} shows that the Mirsky partition estimator 
matches this bound up to a poly-logarithmic factor, thereby achieving the smallest adaptivity index achievable for any polynomial time procedure. Comparing Proposition
\ref{prop:AI-MP-bounded} with~\citet[Theorem 2]{ShaBalWai16-2}, we also see that in the bounded case,
the Mirsky partition estimator
significantly improves the logarithmic factor in the best-known upper bound, from
$\log^8 n$ (for their CRL estimator), to $\log^{5/4} n$. In fact, the following
remark shows that an even smaller adaptivity index can be achieved.

\begin{remark} \label{rem:bounded-MP}
If step II of the Mirsky partition estimator is changed to a projection onto the
bounded set $\Mspace(\lattice_{d, n}; \blhat_1, \ldots, \blhat_d) \cap \infball
(1)$,
then it can be shown by repeating the steps of our proof of Proposition~\ref{prop:AI-MP-bounded}
and using metric entropy bounds from the proof of
Proposition~\ref{prop:full-funlim} that the
resulting estimator
$\thetahatblock^{{\sf
bd}}$ satisfies
\begin{align*}
\sup_{\thetastar \in \Mspace^{\kfull, \sset}_{\perm}(\lattice_{d, n}) \cap \infball(1) }
\Rspace_n
(\thetahatblock^{{\sf bd}}, \thetastar) \leq \frac{C}{n} \cdot (n_1 - k^* + 1)
\cdot n^
{1/4}
\log n,
\end{align*}
leading to the adaptivity index
\begin{align*}
\overline{\adapt}(\thetahatblock^{{\sf bd}}) \leq C \cdot n^{1/4} \log n.
\end{align*}
\end{remark}

Having established results when $d = 2$, let us now turn to a discussion of the general case $d \geq 3$.
By straightforward modifications to our arguments used to prove Proposition~\ref{prop:AI-MP-bounded}, 
it is possible to prove a general upper bound of the form
\begin{align} \label{eq:bdd-upper-bound}
\sup_{\thetastar 
\in \Mspace^{\kfull, \sset}_{\perm}(\lattice_{d, n}) \cap \infball(1)} \;
\Rspace_n (\thetahatblock, \thetastar)  \leq \frac{C}{n} \left
\{ \min \bigl(s, n^{1-1/d} \log^{3/2} n \bigr) +  d^2
(n_1 - k^*) \cdot n^{\frac{1}{2}\left( 1 - \frac{1}{d}
\right)} \right\} \log n.
\end{align}
However, in order to turn the guarantee~\eqref{eq:bdd-upper-bound} into a bound on the adaptivity index $\overline{\adapt}(\thetahatblock)$, we would require a corresponding minimax lower bound over the set 
$\Mspace^{\kfull, \sset}_{\perm}(\lattice_{d, n}) \cap \infball(1)$. Now such a lower bound appears to be particularly
challenging to obtain even in the case \emph{without unknown permutations}, and is likely to exhibit an intricate dependence
on the pair $(\kfull, \sset)$ over and above just the number of pieces $s$. Obtaining a sharp guarantee on this minimax
risk---and subsequently, providing a sharper analysis of the Mirsky partition estimator to attempt to match this guarantee---are both interesting open problems that we discuss in more detail in Section~\ref{sec:discussion}.


\section{Proofs of main results} \label{sec:proofs}

We now turn to proofs of our main results, beginning with some quick notes for
the reader. Throughout, the values of universal constants $c, C, c_1, \ldots$
may change from line to line. 
We also require a bit of additional notation. For
a tensor $T \in \Tspace_{d, n}$, we write $\| T \|_2 = \sqrt{ \sum_{x \in
\lattice_{d, n}} T^2_x }$, so that $\ell_n^2(\theta_1, \theta_2) = \| \theta_1
- \theta_2 \|_2^2$. For a pair of binary vectors
$(v_1, v_2)$ of equal dimension, we let $\dH(v_1, v_2)$ denote the Hamming
distance between them. We use the abbreviation ``wlog'' for ``without loss of
generality". Finally, since we assume throughout that $n_1 \geq 2$ and $d \geq 2$,
we will use the fact that $n \geq 4$ repeatedly and without explicit mention.

We employ two elementary facts about $\ell_2$ projections,
which are stated below for convenience. Recall our notation for the least
squares
estimator~\eqref{eq:lse-defn} as the $\ell_2$ projection onto a closed set
$\mathcal{C} \subseteq \Tspace_{d, n}$, and assume that the projection exists.
For tensors $T_1 \in \Cspace$ and $T_2
\in \Tspace_{d,n}$, we have 
\begin{subequations}
\begin{align} \label{near-contract-closed}
\| \thetahatlse( \Cspace, T_1 + T_2 ) - T_1 \|_2 \leq 2 \| T_2 \|_2.
\end{align}
The proof of this statement is straightforward; by the triangle inequality, we
have
\begin{align*}
\| \thetahatlse( \Cspace, T_1 + T_2 ) - T_1 \|_2 &\leq \| \thetahatlse(
\Cspace, T_1 + T_2 ) - (T_1 + T_2) \|_2 + \| (T_1 + T_2 ) - T_1 \|_2 \\
&\leq 2 \| (T_1 + T_2 ) - T_1 \|_2 \\
&= 2 \| T_2 \|_2,
\end{align*}
where the second inequality follows since $\thetahatlse( \Cspace,
T_1 + T_2 )$ is the closest point in $\Cspace$ to $T_1 + T_2$.
If moreover, the set $\mathcal{C}$ is convex, then the projection is unique, and
non-expansive:
\begin{align} \label{contract-convex}
\| \thetahatlse( \Cspace, T_1 + T_2 ) - T_1 \|_2 \leq \| T_2 \|_2.
\end{align}
\end{subequations}
With this setup in hand, we are now ready to proceed to the proofs of the main
results.

\subsection{Proof of Proposition~\ref{prop:full-funlim}}

It suffices to prove the upper bound, since the minimax lower bound was already
shown
by~\citet{han2019} for isotonic regression without unknown
permutations. We also focus on the case $d \geq 3$ since the result is already
available for $d = 2$~\citep{ShaBalGunWai17,mao2018towards}.
Our proof proceeds in two parts. First, we show that the bounded least squares
estimator over isotonic tensors (without unknown permutations) enjoys the
claimed risk bound. We then use our proof of this result to prove the upper
bound~\eqref{eq:two-terms-blse}. The proof of the first result is also useful in
establishing part (b) of Proposition~\ref{cor:lse}.

\paragraph*{Bounded LSE over isotonic tensors}
For a tensor $A \in \real_{d, n}$ and $x_1, \ldots, x_{d - 2} \in [n_1]$, let 
$A_{x_1, \ldots, x_{d-2}}$ denote the matrix formed by fixing the first $d - 2$
dimensions (variables) of $A$ to $x_1, \ldots, x_{d - 2}$, i.e., entry $(i, j)$
of this matrix is given by
$A (x_1, \ldots, x_{d-2}, i, j)$.
Recall that $\Mspace(\lattice_{2, n_1, n_1})$ denotes the set of all bivariate isotonic
$n_1 \times n_1$ matrices.

For convenience, let $\Mspace(\lattice_{d,n} \mid r)$ and
$\Mspace(\lattice_{2, n_1, n_1} \mid r)$ denote the intersection of the respective
sets with the $\ell_\infty$ ball of radius $r$. Letting $A - B$ denote the
Minkowski difference between the sets $A$ and $B$, define
\begin{align*}
\Mspace^{\diff}(\lattice_{2, n_1, n_1} \mid r) &\defn \Mspace(\lattice_{2, n_1, n_1} \mid r) - \Mspace(\lattice_{2, n_1, n_1} \mid r) \quad \text{ and } \\
\Mspace^{\full} (r) &\defn \left\{T \in \real_{d, n} : T_{x_1, \ldots, x_{d-2}} \in \Mspace(\lattice_{2, n_1, n_1} \mid r ) \text{ for all } x_1, \ldots, x_{d-2} \in [n_1]\right\}.
\end{align*}
We observe that there is a bijection between $\Mspace^{\full} (r)$ and the $n_1^{d-2}$-fold Cartesian product $\prod_{x_1, \ldots, x_{d - 2}}  \Mspace(\lattice_{2, n_1, n_1} \mid r)$ of $\Mspace(\lattice_{2, n_1, n_1} \mid r)$.  
Note that through this notation, we are indexing the components of this Cartesian product by the elements of $\mathbb{L}_{d-2,n_1,\ldots,n_1}$, so that a generic element $A$ of this set has $(x_1,\ldots,x_{d-2})$-th component $A_{x_1,\ldots,x_{d-2}} \in \Mspace(\lattice_{2, n_1, n_1} \mid r)$. The bijection of interest then maps $T \in \Mspace^{\full} (r)$ to the element of \mbox{$\prod_{x_1, \ldots, x_{d - 2}}  \Mspace(\lattice_{2, n_1, n_1} \mid r)$}
whose $(x_1,\ldots,x_{d-2})$-th component is $T_{x_1,\ldots,x_{d-2}}$.

%
Note that by construction, we have ensured, for each $r \geq 0$, the inclusions
\begin{subequations} \label{eq:inclusions}
\begin{align} 
\Mspace(\lattice_{d,
n} \mid r) &\subseteq \Mspace^{\full}(r) \quad \text{ and } \\
\Mspace(\lattice_{d, n} \mid r) -  \Mspace
(\lattice_{d, n} \mid r) &\subseteq
\prod_{x_1, \ldots, x_{d - 2}}  \Mspace^{\diff}(\lattice_{2, n_1, n_1} \mid r) =
\Mspace^{\full}(r) - \Mspace^{\full} (r). 
\end{align}
\end{subequations}

With this notation at hand, let us now proceed to bound the risk of the bounded
LSE. By definition, this estimator can be written as the projection of $Y$ onto
the set $\Mspace(\lattice_{d, n} \mid 1)$, so we have
\begin{align} \label{eq:obj-blse}
\thetahatblse = \argmin_{\theta \in \Mspace(\lattice_{d, n} \mid 1)} \| Y -
\theta \|_2^2.
\end{align}
Letting $\Deltahat \defn \thetahatblse - \thetastar$, the optimality of
$\thetahatblse$ and feasibility of $\thetastar$ in the objective
\eqref{eq:obj-blse} yield the basic inequality
\[
\| Y - \thetahatblse \|_2^2 \leq \| Y - \thetastar \|_2^2,
\]
and writing $Y = \thetastar + \epsilon$, we obtain $\| \Deltahat - \epsilon \|_2^2
\leq \| \epsilon \|_2^2$. Expanding the square yields $\| \Deltahat \|_2^2 + \| \epsilon \|_2^2 - 2 \inprod{\Deltahat}{\epsilon} \leq \| \epsilon \|_2^2$, and rearranging this inequality, we obtain
\begin{align*}
\frac{1}{2} \| \Deltahat \|_2^2 \leq \inprod{\epsilon}{\Deltahat} \leq \sup_{ 
\substack{\theta \in \Mspace(\lattice_{d, n} \mid 1) \\ \| \theta - \thetastar \|_2
\leq \| \Deltahat \|_2} } \inprod{\epsilon}{\theta - \thetastar} \leq \sup_{ \substack{\Delta \in \Mspace^{\full}(1) - \Mspace^{\full}(1) \\ \| \Delta \|_2 \leq \| \Deltahat \|_2} } \inprod{\epsilon}{\Delta}.
\end{align*}
For convenience, define for each $t \geq 0$ the random variable
\begin{align*}
\xi(t) \defn \sup_{ \substack{\Delta \in \Mspace^{\full}(1) - \Mspace^{\full}(1) \\
\| \Delta \|_2 \leq t} } \inprod{\epsilon}{\Delta}.
\end{align*}
Also note that the set $\Mspace^{\full}(1) - \Mspace^{\full}(1)$ is star-shaped
and non-degenerate (see Definition~\ref{def:star-shaped} in Appendix
\ref{app:lse}).
Now applying Lemma~\ref{lem:star-shaped} from the appendix---which is, in turn,
based on~\citet[Theorem 13.5]{wainwright2019high}---we see that
\begin{align} \label{eq:risk-from-ci}
\EE[ \| \Deltahat \|_2^2 ] \leq C (t_n^2 + 1),
\end{align}
where $t_n$ is the smallest (strictly) positive solution to the critical
inequality
\begin{align} \label{eq:crit-thm}
\EE [\xi(t)] \leq \frac{t^2}{2}.
\end{align}
Thus, it suffices to produce a bound on $\EE [\xi(t)]$, and in order to do so,
we use Dudley's entropy integral along with a bound on the
$\ell_2$ metric entropy of the set $\left( \Mspace^{\full}(1) - \Mspace^{\full}
(1)
\right) \cap \mathbb{B}_2(t)$. Owing to
the inclusions~\eqref{eq:inclusions}, we see that in order to cover the set
$\Mspace^{\full}(1) - \Mspace^{\full}(1)$ in $\ell_2$-norm at radius $\delta$, 
it suffices to produce a cover of the set 
$\Mspace(\lattice_{2, n_1, n_1} \mid 1)$ in $\ell_2$-norm at radius $\delta' = 
n_1^{-\frac{d - 2}{2}} \cdot \frac{\delta}{\sqrt{2}}$. This is because we must cover
the $n_1^{d-2}$-fold Cartesian product of $\Mspace^{\diff}(\lattice_{2, n_1, n_1} \mid 1)$, and a $\delta$-covering
of the set $\Mspace^{\diff}(\lattice_{2, n_1, n_1} \mid 1)$ can be accomplished using
$\delta/\sqrt{2}$ coverings of the two copies of $\Mspace(\lattice_{2, n_1, n_1} \mid 1)$ 
that are involved in the Minkowski difference. Thus, recalling our
notation for the covering number of a set from the end of Section~\ref{sec:setup} of
the main text, we have
\begin{align} \label{eq:me-calc}
N (\delta; \Mspace^{\full}(1) - \Mspace^{\full}(1), \| \cdot \|_2 ) \leq
\left\{ N( \delta'; \Mspace(\lattice_{2, n_1, n_1} \mid 1),
\| \cdot \|_2 ) \right\}^{2 n_1^{d-2}}.
\end{align}
Furthermore, by~\citet[Theorem 1.1]{gao2007entropy} (see also
\citet[equation (29)]{ShaBalGunWai17}),
we have\footnote{Note that we did not explicitly introduce the covering itself, since bounds on the covering number suffice; the construction of the cover for bounded isotonic functions can be found in~\citet{gao2007entropy}.}
\begin{align} \label{eq:gao-wellner}
\log N(\tau; \Mspace(\lattice_{2, n_1, n_1} \mid 1), \| \cdot \|_2 ) \lesssim 
\frac{n_1^2}{\tau^2} \log^2  \left(\frac{n_1}{\tau} \right) \;\; \text{ for
each } \tau > 0.
\end{align}
Recall that the notation $\lesssim$ hides a universal constant independent of the dimension.
Using inequalities~\eqref{eq:me-calc} and~\eqref{eq:gao-wellner} in conjunction, we obtain
\begin{align*}
\log N (\delta; \Mspace^{\full}(1) - \Mspace^{\full}(1), \| \cdot \|_2 )
&\lesssim n_1^{d - 2} \cdot \log N( \delta'; \Mspace(\lattice_{2, n_1, n_1} \mid
1), \| \cdot \|_2 ) \\
&\stackrel{\1}{\lesssim} n_1^{d - 2} \cdot \frac{n}{\delta^2} 
\log^2
\left( \frac{n}{\delta} \right),
\end{align*}
where in step $\1$, we have substituted the value of $\delta'$ and noted
that $n_1^d = n$.
Now the truncated form of Dudley's entropy integral (see, e.g.,~\citet[Theorem 5.22]{wainwright2019high}) 
yields, for each $t_0 \in [0, t]$, the bound
\begin{align*}
\EE[\xi(t)] &\lesssim t_0 \cdot \sqrt{n} + \int_{t_0}^{t} \sqrt{ \log N 
(\delta; \Mspace^{\full}(1) - \Mspace^{\full}(1) \cap \mathbb{B}_2(t), \|
\cdot \|_2 ) } d\delta \\
&\leq t_0 \cdot \sqrt{n} + \int_{t_0}^{t} \sqrt{ \log N 
(\delta; \Mspace^{\full}(1) - \Mspace^{\full}(1), \| \cdot \|_2 ) } d\delta
\end{align*}
Choose $t_0 = n^{-11/2}$, apply inequality~\eqref{eq:me-calc}, and note that
$\log \frac{n}
{\delta}
\lesssim \log n$ for all $\delta \geq n^{-11/2}$ to obtain
\begin{align*} 
\EE[\xi(t)] &\lesssim n^{-5} + \int_{n^{-11/2}}^t \sqrt{n_1^{d - 2}} \sqrt{n}
\cdot (\log n) \cdot  \delta^{-1} d \delta \\
&\lesssim n^{1 - 1/d} \cdot (\log n) \cdot (\log nt).
\end{align*}
Some algebraic manipulation then yields that the solution $t_n$ to the critical
inequality~\eqref{eq:crit-thm} must satisfy $t_n^2 \leq C n^{1 - 1/d} \cdot
\log^2 n$. Substituting into inequality~\eqref{eq:risk-from-ci} yields the risk bound
\begin{align} \label{eq:hpb-blse-noperm}
\EE \left[ \| \Deltahat \|_2^2 \right] \leq C n^{1 - 1/d}
\cdot \log^2 n,
\end{align}
as desired.

\paragraph*{Bounded least squares with unknown permutations} The proof for this
case proceeds very similarly to before; the only additional effort is to bound
the empirical process over a \emph{union} of a large
number of difference-of-monotone cones.
Similarly to before, define $\Mspace_{\perm}(\lattice_{d, n} \mid
r) \defn \Mspace_{\perm}(\lattice_{d, n}) \cap \infball(r)$ and the sets
$\Mspace(\lattice_{d, n}; \pi_1, \ldots, \pi_d \mid r)$ analogously. Let
$\thetahatblse$ now denote the bounded LSE with permutations~\eqref{eq:def-blse}.
Proceeding similarly to before with $\Deltahat = \thetahatblse -
\thetastar$ yields 
\begin{align}
\frac{1}{2} \| \Deltahat \|_2^2 &\leq \sup_{ \substack{ \theta_1 \in \Mspace_
{\perm}(\lattice_{d, n} \mid 1) \\ \theta_2 \in \Mspace_
{\perm}(\lattice_{d, n} \mid 1) \\ \| \theta_1 - \theta_2 \|_2 \leq \|
\Deltahat \|_2 } } \inprod{\epsilon}{\theta_1 - \theta_2} \notag \\
&= \max_{\pi_1, \ldots,
\pi_d \in \Pspace_{n_1}} \; \max_{\pi'_1, \ldots, \pi'_d \in \Pspace_{n_1} }\;
\sup_{\substack{\theta_1 \in \Mspace(\lattice_{d, n}; \pi_1, \ldots, \pi_d \mid
1) \\ \theta_2 \in \Mspace(\lattice_{d, n}; \pi'_1, \ldots, \pi'_d \mid 1) \\ \| \theta_1 - \theta_2 \|_2 \leq \|
\Deltahat \|_2}}
\inprod{\epsilon}
{\theta_1 - \theta_2}. \label{eq:union-decomp}
\end{align}
For convenience, denote the supremum of the empirical process localized at
radius $t > 0$ by
\begin{align*}
\xi(t) \defn \sup_{ \substack{ \Delta \in \Mspace_
{\perm}(\lattice_{d, n} \mid 1) -  \Mspace_
{\perm}(\lattice_{d, n} \mid 1) \\ \| \Delta \|_2 \leq t}} \inprod{\epsilon}
{\Delta}.
\end{align*}
Since the set $\Mspace_
{\perm}(\lattice_{d, n} \mid 1) - \Mspace_
{\perm}(\lattice_{d, n} \mid 1)$ is star-shaped and non-degenerate (see
Definition~\ref{def:star-shaped}), applying Lemma
\ref{lem:star-shaped} as before yields the risk bound
$\EE [\| \Deltahat \|_2^2] \leq C(t_n^2 + 1)$, where
$t_n$ is the smallest positive solution to the critical inequality $\EE [\xi(t)] \leq \frac{t^2}{2}$.
Using the form of
the empirical process in equation~\eqref{eq:union-decomp}, notice that $\xi(t)$
is the supremum of a Gaussian process over the union of \mbox{$K = (n_1!)^{2d}$}
sets,
each of which contains the origin and is contained in an $\ell_2$ ball of
radius $t$. We also have $\log K \leq 2d n_1 \log n_1 = 2 n_1 \log n$.
Applying
Lemma~\ref{lem:sup-emp} from the appendix, we thus obtain
\begin{align} \label{eq:union-step1}
\EE [\xi(t)] \leq  \max_{\pi_1, \ldots,
\pi_d \in \Pspace_{n_1}} \; \max_{\pi'_1, \ldots, \pi'_d \in \Pspace_{n_1} }\;
\EE \; \sup_{\substack{\theta_1 \in \Mspace(\lattice_{d, n}; \pi_1, \ldots,
\pi_d \mid 1) \\ \theta_2 \in \Mspace(\lattice_{d, n}; \pi'_1, \ldots, \pi'_d
\mid 1) \\ \| \theta_1 - \theta_2 \|_2 \leq t}} \inprod{\epsilon}{\theta_1 -
\theta_2} + C t \sqrt{n_1 \log n}.
\end{align}
Let us now fix permutations $\pi_1, \ldots, \pi_d \in \Pspace_{n_1}$ and $\pi'_1, \ldots,
\pi'_d \in \Pspace_{n_1}$ and bound the expectation of the supremum in equation~\eqref{eq:union-step1}. 
For convenience, let 
\begin{align*}
\mathcal{D}(\pi_1, \ldots, \pi_d; \pi'_1,
\ldots, \pi'_d) \defn \Mspace(\lattice_{d, n}; \pi_1, \ldots, \pi_d \mid 1) -
\Mspace(\lattice_{d, n}; \pi'_1, \ldots, \pi'_d \mid 1),
\end{align*} 
and note the sequence of
covering number bounds
\begin{align*}
N (\delta; \mathcal{D}(\pi_1, \ldots, \pi_d; \pi'_1,
\ldots, \pi'_d) \cap \mathbb{B}_2(t), \| \cdot \|_2 ) &\leq
N (\delta; \mathcal{D}(\pi_1, \ldots, \pi_d; \pi'_1,
\ldots, \pi'_d), \| \cdot \|_2 ) \\
&\stackrel{\2}{\leq} \left[ N (\delta / \sqrt{2}; \Mspace(\lattice_{d, n} \mid
1), \|
\cdot \|_2 ) \right]^2,
\end{align*}
where step $\2$ follows since it suffices to cover the sets $\Mspace(\lattice_{d, n}; \pi_1, \ldots, \pi_d
\mid 1)$ and \mbox{$\Mspace(\lattice_{d, n}; \pi'_1, \ldots, \pi'_d \mid 1)$} at
radius $\delta / \sqrt{2}$, and each
of these has covering number equal to that of $\Mspace(\lattice_{d, n} \mid
1)$. Now proceeding exactly as in the previous calculation and performing the
entropy integral, we have
\begin{align} \label{eq:integral-step2}
\EE \; \sup_{\substack{\theta_1 \in \Mspace(\lattice_{d, n}; \pi_1, \ldots,
\pi_d \mid 1) \\ \theta_2 \in \Mspace(\lattice_{d, n}; \pi'_1, \ldots, \pi'_d
\mid 1) \\ \| \theta_1 - \theta_2 \|_2 \leq t}} \inprod{\epsilon}{\theta_1 -
\theta_2} \lesssim n^{1 - 1/d} \cdot (\log n) \cdot (\log nt ).
\end{align}
Combining inequalities~\eqref{eq:union-step1} and~\eqref{eq:integral-step2} yields the bound
\begin{align*}
\EE [\xi(t)] \lesssim n^{1 - 1/d} \cdot (\log n) \cdot (\log nt ) + t \sqrt{n_1 \log n}.
\end{align*}
Consequently, the smallest positive solution $t_n$ to the critical inequality $\EE [\xi(t)] \leq \frac{t^2}{2}$ 
satisfies $t_n \lesssim \sqrt{n^{1 - 1/d}} \cdot (\log n) + \sqrt{n_1 \log n}$, and performing some algebra
completes the proof.
\qed

\subsection{Proof of Proposition~\ref{cor:lse}}

This proof utilizes Proposition~\ref{prop:full-funlim} in conjunction with a
truncation argument. We provide a full proof of part (a); the
proof of part
(b) is very similar and we sketch the differences. Recall throughout that by
assumption, we have
$\thetastar \in \infball(1)$.

Recalling our notation~\eqref{eq:lse-defn} for least squares estimators, note
that the global least squares estimator
$\thetahatlse(\Mspace_{\perm}(\lattice_{d, n}), Y)$ belongs to the set
\[
\left\{ \thetahatlse( \Mspace(\lattice_{d, n}; \pi_1, \ldots, \pi_d) , Y)
\mid \pi_1, \ldots, \pi_d \in \Pspace_{n_1} \right\}.
\]
Applying Lemma~\ref{lem:contraction} from the appendix, we see that
for each tuple of permutations $(\pi_1, \ldots, \pi_d)$, the projection onto the
set $\Mspace(\lattice_{d, n}; \pi_1, \ldots, \pi_d)$ is
$\ell_\infty$-contractive, so that
\[
\|\thetahatlse( \Mspace(\lattice_{d, n}; \pi_1, \ldots, \pi_d) , Y)\|_
{\infty} \leq \| Y \|_{\infty}.
\]
Consequently, we have $\| \thetahatlse(\Mspace_
{\perm}(\lattice_{d, n}), Y)\|_{\infty} \leq \| Y\|_{\infty} \leq 1 + \|
\epsilon \|_{\infty}$. By a union bound,
\begin{align*}
\Pr \{ \| \epsilon \|_{\infty} \geq 4 \sqrt{\log n} \} \leq n \Pr \{ |\epsilon_1 | \geq 4 \sqrt{\log n} \} \leq n^{-7},
\end{align*}
where in the final step, we have used the standard Gaussian tail bound $\Pr\{ Z \geq t \} \leq \frac{1}{2} e^{-t^2/2}$.

Let $\psi_n \defn 4 \sqrt{\log n} + 1$ for convenience. On the
event $\Espace \defn \{ \| \epsilon \|_{\infty} \leq 4 \sqrt{\log n}\}$, we
thus have $\| \thetahatlse(\Mspace_{\perm}(\lattice_{d, n}), Y)\|_{\infty} \leq
\psi_n$.
Therefore, on this event, we have an equivalence between the vanilla
LSE and the bounded~LSE:
\begin{align} \label{eq:first-bound}
\thetahatlse(\Mspace_
{\perm}(\lattice_{d, n}), Y) = \thetahatlse(\Mspace_
{\perm}(\lattice_{d, n}) \cap \infball(\psi_n), Y).
\end{align}
Now replicating the proof of Proposition~\ref{prop:full-funlim} for the bounded LSE
with $\ell_\infty$-radius\footnote{In more detail, note that by a rescaling
argument, it suffices to replace $\tau$ in equation~\eqref{eq:gao-wellner} with
$\tau/r$. Since $r \leq n$, note that $\log (rn) \lesssim \log n$.} $r \in 
(0,
n]$ yields the risk bound
\begin{align} \label{eq:second-bound}
\EE \left[ \| \thetahatlse( \Mspace_{\perm}(\lattice_{d, n}) \cap \infball(r),
Y) - \thetastar \|_2^2 \right] \leq c (r n^{1 - 1/d} \log^2 n + n^{1/d}
\log n ).
\end{align}
Finally, since the least squares
estimator is a projection onto a union of convex sets, inequality~\eqref{near-contract-closed}
yields the bound
\begin{align} \label{eq:third-bound}
\| \thetahatlse(\Mspace_{\perm}(\lattice_{d, n}), Y) - \thetastar \|^2_2 \leq
4 \|
\epsilon \|^2_2.
\end{align} 
Using the bounds~\eqref{eq:first-bound},~\eqref{eq:second-bound} with $r =
\psi_n$,
and~\eqref{eq:third-bound} in conjunction with Lemma~\ref{lem:RV} from the
appendix yields 
\begin{align*}
\Rspace_n( \thetahatlse(\Mspace_
{\perm}(\lattice_{d, n}), Y), \thetastar) &\leq C \left\{ \psi_n \cdot n^
{1 - 1/d}
\log^2 n + n^{1/d}
\log n  + n^{-7/2} \cdot \sqrt{ \EE [ \| \epsilon \|_2^4 ] } \right\} \\
&\stackrel{\1}{\leq} C \left\{  \psi_n \cdot n^{1 - 1/d}
\log^2 n + n^{1/d}
\log n \right\},
\end{align*}
where step $\1$ follows since
\begin{align} \label{eq:fourth-moment-noise}
\EE [ \| \epsilon \|_2^4 ] = n^2 + 2n \leq (n +
1)^2.
\end{align}

In order to prove part (b) of the proposition, all steps of the previous argument
can be reproduced
verbatim with the set $\Mspace(\lattice_{d, n})$ replacing $\Mspace_{\perm}(
\lattice_{d, n})$. The risk bound for the estimator $\thetahatlse(
\Mspace(\lattice_{d, n}) \cap \infball(r), Y)$ can be obtained from the first
part of the proof of Proposition~\ref{prop:full-funlim} (see equation
\eqref{eq:hpb-blse-noperm}) and takes the
form
\begin{align} \label{eq:new-second-bd}
\EE \left[ \| \thetahatlse( \Mspace(\lattice_{d, n}) \cap \infball(r),
Y) - \thetastar \|_2^2 \right] \leq c r n^{1 - 1/d} \log^2 n \quad \text{ for
each }  0 < r \leq n.
\end{align}
Replacing equation~\eqref{eq:second-bound} with~\eqref{eq:new-second-bd},
setting $r = \psi_n$,
and putting together the
pieces as before proves the
claimed result.
\qed

\subsection{Proof of Proposition~\ref{prop:adapt-funlim}} \label{sec:pf-funlim}

We prove the upper and lower bounds separately. Recall our notation $\Kset_q$
for the set of all tuples of positive integers $\kset = (k_1, \ldots, k_{q})$
with $\sum_{\ell =1}^{q} k_{\ell} = n_1$. In this proof, we make use of
notation that was defined in Section~\ref{sec:alg}. Recall from that section
our definition of a one-dimensional ordered partition, the set $\Mspace
(\lattice_{d, n}; \bl_1, \ldots, \bl_d)$, and that $\Partition_L$
denotes the set of all one-dimensional ordered partitions of size $L$.
Also, let $\Partition^{\max}_{k}$ denote all
one-dimensional partitions of $[n_1]$ in which the largest block has size at
least $k$.

\subsubsection{Proof of upper bound}

For each tuple $\kset \in \Kset_{q}$, let $\beta(\kset) \subseteq \Partition_q$
denote the set of all one-dimensional ordered partitions that are consistent with the set
sizes $\kset$. For example, when $n_1 = 3$ and $\kset = (1, 2)$, the set $\beta(\kset)$ 
consists of the elements 
\begin{align*}
\bl_1 &= (\{1,2\}, \{3\}),  &\bl_2 &= (\{3\}, \{1, 2\}), \\
\bl_3 &= (\{1,3\}, \{2\}), &\bl_4 &= (\{2\}, \{1, 3\}), \\
\bl_5 &= (\{2, 3\}, \{1\}), \qquad \text{ and } &\bl_6 &= (\{1\}, \{2, 3\}).
\end{align*}
Note the equivalence
\begin{align*}
\Mspace^{\kfull, \sset}_{\perm} (\lattice_{d, n}) = \bigcup_{\bl_1 \in \beta
(\kset^1) } \cdots
\bigcup_
{\bl_d \in \beta(\kset^d) } \Mspace(\lattice_{d, n}; \bl_1, \ldots, \bl_d),
\end{align*}
and also that $|\beta(\kset^j)| \leq | \Partition^{\max}_{k^j_{\max}}| \leq 
e^{3(n_1 - k^j_{\max}) \log n_1}$; here, the final inequality follows from
Lemma~\ref{lem:num-partitions}(b) in the appendix.
Recall that we have $Y = \thetastar + \epsilon$ for some tensor $\thetastar \in
\Mspace(\lattice_{d, n}; \blstar_1, \ldots, \blstar_d)$, where $\blstar_j \in
\beta(\kset^j)$ for each $j \in [d]$.
The estimator that we analyze for the upper bound is the least squares
estimator $\thetahatlse( \Mspace^{\kfull, \sset}_{\perm}(\lattice_{d, n}), Y)$,
which we denote for convenience by $\thetahat$ 
for this proof. Since we are analyzing a least squares estimator, our strategy
for this proof will be to set up the appropriate empirical process and apply the
variational inequality in Lemma~\ref{lem:variational} in order to bound the
 error.

Specifically, for each $t \geq 0$, define the random variable
\begin{align*}
\xi(t) \defn \sup_{ \substack{\theta \in \Mspace^{\kfull, \sset}_{\perm}
(\lattice_
{d, n}) \\ \| \theta - \thetastar \|_2^2 \leq t } } \inprod{\epsilon}{\theta -
\thetastar} = \max_{ \substack{ \bl_1, \ldots, \bl_d \\ \bl_j \in \beta
(\kset^j)}}  \; \sup_
{ \substack{\theta \in
\Mspace(\lattice_{d, n}; \bl_1,
\ldots, \bl_d) \\
\| \theta - \thetastar \|_2 \leq t } } \inprod{\epsilon}{\theta - \thetastar},
\end{align*}
which is the pointwise maximum of 
$K = \prod_{j = 1}^d |\beta
(\kset^j)|$ random variables. Note that we
have 
\[
\log K = \sum_{j = 1}^d \log | \beta(\kset^j)| \leq \sum_{j =
1}^d 3 (n_1 - k^j_{\max}) \log n_1 \leq 3 (n_1 - k^*) \log n.
\]
Applying Lemma~\ref{lem:sup-emp}(a) from the appendix, we have that for each $u
\geq 0$,
\begin{align} \label{eq:partition-hbp}
\Pr \left\{ \xi(t) \geq \max_{ \substack{ \bl_1, \ldots, \bl_d \\ \bl_j \in
\beta(\kset^j)} } \; \EE \! \sup_
{
\substack{\theta \in
\Mspace(\lattice_{d, n}; \bl_1,
\ldots, \bl_d)  \\
\| \theta - \thetastar \|_2 \leq t } } \inprod{\epsilon}{\theta - \thetastar}
 + C t (\sqrt{(n_1 \! - \! k^*) \log n} + \sqrt{u} ) \right\} \leq e^{-u}.
\end{align}
for some universal constant $C > 0$. Lemma~\ref{lem:fixed-partition-ep},
which is stated and proved at the end of this subsection, controls the expected 
supremum of the empirical process for a fixed choice of the partitions
$\bl_1, \ldots, \bl_d$. 
Combining Lemma~\ref{lem:fixed-partition-ep} with the high probability
bound~\eqref{eq:partition-hbp}, we obtain, for each pair of non-negative scalars
$(t, u)$, the bound
\begin{align} \label{eq:hpb-var}
\Pr \left\{ \xi(t) \geq  C t \left( \sqrt{s} + \sqrt{(n_1 - k^*) \log n}
+
\sqrt{u} \right) \right\} \leq e^{-u}.
\end{align}
Now define the function $f_{\thetastar}(t) \defn \xi(t) - \frac{t^2}{2}$; our
goal---driven by Lemma~\ref{lem:variational}---is to compute a value $t_{*}$ 
such that with high probability, $f_{\thetastar}(t) < 0$ for all $t \geq t_*$. 

For a sufficiently large constant $C > 0$, define the scalar
\[
t_u \defn C ( \sqrt{s} + \sqrt{(n_1 - k^*) \log n} + \sqrt{u} ) \;\; 
\text{
for
each } u \geq 0.
\]
We claim that on an event $\Espace$ occurring with probability at least $1 -
C n^{-10}$, the choice $u^* = C \log n$ ensures that
\begin{align} \label{eq:suff-claim}
f_{\thetastar} (t) < 0 \qquad \text{ simultaneously for all } t \geq t_{u^*}.
\end{align}
Taking this claim as given for the moment, the proof of the upper bound of the
proposition follows
straightforwardly: Applying Lemma~\ref{lem:variational} and substituting the
value $t^* = t_{u^*}$ yields the bound
\begin{align*}
\| \thetahat - \thetastar \|_2^2 \leq C \left( s + (n_1 - k^*)
\log n \right)
\end{align*}
with probability at least $1 - C n^{-10}$. In order to produce a bound that
holds in expectation, note that since $\thetahat$ is obtained via a projection
onto a union of convex sets, inequality~\eqref{near-contract-closed} yields the
pointwise bound $\| \thetahat - \thetastar \|_2^2 \leq 4 \| \epsilon
\|_2^2$. Applying Lemma~\ref{lem:RV} and combining the
pieces yields
\begin{align*}
\EE [ \| \thetahat - \thetastar \|_2^2 ] &\leq C \left( s + (n_1 - k^*)
\log n \right) + C' \sqrt{\EE [ \| \epsilon \|_2^4 ] } \cdot \sqrt{n^{-10}} \\
&\leq C \left( s + (n_1 - k^*)
\log n \right),
\end{align*}
where the final inequality is a consequence of the bound
\eqref{eq:fourth-moment-noise}.
 
It remains to establish claim~\eqref{eq:suff-claim}. First, inequality
\eqref{eq:hpb-var}
ensures
that $\xi(t) < t^2
/ 8$ for each fixed $t \geq t_u$ with probability at least $1 - e^{-u}$,
thereby guaranteeing that $f_{\thetastar}(t) < 0$ for each fixed $t \geq t_u$.
Moreover, the Cauchy--Schwarz inequality yields the pointwise bound
$\xi(t) \leq t \| \epsilon \|_2$, so that applying the chi-square tail bound
from~\citet[Lemma 1]{laurent2000adaptive} yields
\[
\Pr \left\{ \xi(t) 
\leq t (\sqrt{n}+ \sqrt{2u'}) \right\} \geq 1 - e^{-u'} \quad \text{ for
each }
u'\geq 0.
\]

Set $u' = u^*$, and note that on this event, we have $f_{\thetastar}
(t) < 0$ \emph{simultaneously} for all $t \ge t^\#_{u^*} \defn C (\sqrt{n} + 
\sqrt{\log n})$. It remains to handle the values of $t$ between $t_{u^*}$
and
$t^{\#}_{u^*}$. We suppose that $t^{\#}_{u^*} \geq t_{u^*}$ without loss of
generality---there is nothing to prove otherwise---and
employ a discretization argument.
Let $T = \{ t^1, \dots, t^L \}$ be a  discretization of the interval $[t_{u^*},
t^\#_{u^*}]$ such that
$t_{u^*} = t^1 < \cdots < t^L = t^{\#}_{u^*}$
and $2t^i \geq t^{i+1} $. 
Note that $T$ can be chosen so that 
\[
L = |T| \le \log_2 \frac{t^\#_{u^*}}{t_{u^*}} + 1
\leq c\log n.
\]
Using the high probability bound $\xi(t) < t^2/8$ for each individual $t \ge
t_{u^*}$
and a union bound over $T$, we obtain that with probability at least $1 - c \log
n \cdot e^{-u^*}$, we have
\[
\max_{t \in T} \left\{ \xi(t) - t^2/8 \right\} < 0. 
\]
On this event, 
we use the fact that $\xi(t)$ is (pointwise) non-decreasing and that $t^i \ge t^
{i+1}/2$ to
conclude that for each $i \in [L-1]$ and $t \in [t^i, t^{i+1}]$, we
have
\[
f_{\thetastar}(t) = \xi(t) - t^2/2 \le \xi(t^{i+1}) - (t^i)^2/2 \le \xi(t^{i+1})
-
(t^
{i+1})^2/8 \le \max_{t \in T} \left\{ \xi(t) - t^2/8 \right\} < 0.
\]
Thus, we have shown that $f_{\thetastar}(t) < 0$ simultaneously for all
$t \ge t_{u^*}$ with probability at least
$1 - e^{-u^*} - c \log n \cdot e^{-u^*} \geq 1 - C n^{-10}$. The final
inequality is ensured by adjusting the constants appropriately.
\qed

\begin{lemma} \label{lem:fixed-partition-ep}
Let $\epsilon$ be the standard Gaussian tensor in $\Tspace_{d, n}$. Suppose that
$\sset = (s_1, \ldots, s_d)$ satisfies $\prod_{j = 1}^d s_j = s$,
and that $\kset^j \in \Kset_{s_j}$ for each $j \in [d]$. Then for any
tensor $\thetastar \in \Tspace_{d, n}$ and any sequence of ordered partitions
$\bl_1, \ldots, \bl_d$ with $\bl_j \in \beta(\kset^j)$ for all $j \in [d]$, we have
\begin{align*}
\EE \; \sup_{
\substack{\theta \in
\Mspace(\lattice_{d, n}; \bl_1,
\ldots, \bl_d)  \\
\| \theta - \thetastar \|_2 \leq t } } \inprod{\epsilon}{\theta - \thetastar}
\leq t \sqrt{s} \quad \text{ for each } t \geq 0.
\end{align*}
\end{lemma}
\begin{proof}
First, let $\thetabar = \thetahatlse(\Mspace(\lattice_{d,n}; \bl_1, \ldots,
\bl_d), \thetastar)$ denote the projection of $\thetastar$ onto the set $\Mspace(\lattice_{d,n}; \bl_1, \ldots,
\bl_d)$. Let $\mathbb{B}_2(\theta, t)$ denote the
$\ell_2$ ball of radius $t$ centered at~$\theta$. 
Since the
$\ell_2$ projection onto a
convex set is non-expansive~\eqref{contract-convex}, each $\theta \in \Mspace
(\lattice_{d,n}; \bl_1, \ldots,
\bl_d)$ satisfies
$
\| \theta - \thetabar \|_2 \leq \| \theta - \thetastar \|_2,
$
and so we have the inclusion $\Mspace
(\lattice_{d,n}; \bl_1, \ldots,
\bl_d) \cap \mathbb{B}_2(\thetastar, t) \subseteq \Mspace(\lattice_{d,n}; \bl_1, \ldots,
\bl_d) \cap \mathbb{B}_2(\thetabar, t)$ for each $t \geq 0$.
Consequently, we obtain
\begin{align*}
\sup_{
\substack{\theta \in
\Mspace(\lattice_{d, n}; \bl_1,
\ldots, \bl_d)  \\
\| \theta - \thetastar \|_2 \leq t } } \inprod{\epsilon}{\theta - \thetastar} &= 
\inprod{\epsilon}{\thetabar - \thetastar} + \sup_{
\substack{\theta \in
\Mspace(\lattice_{d, n}; \bl_1,
\ldots, \bl_d)  \\
\| \theta - \thetastar \|_2 \leq t } }  \inprod{\epsilon}{\theta - \thetabar} \\
&\leq \inprod{\epsilon}{\thetabar - \thetastar} + \sup_{
\substack{\theta \in
\Mspace(\lattice_{d, n}; \bl_1,
\ldots, \bl_d)  \\
\| \theta - \thetabar \|_2 \leq t } }  \inprod{\epsilon}{\theta - \thetabar}.
\end{align*}
Since $\thetabar$ is non-random, the term $\inprod{\epsilon}{\thetabar -
\thetastar}$ has expectation zero.
Thus, taking expectations and applying Lemma~\ref{lem:const-blocks-ep} from the
appendix yields
\begin{align*}
\EE \sup_{
\substack{\theta \in
\Mspace(\lattice_{d, n}; \bl_1,
\ldots, \bl_d)  \\
\| \theta - \thetastar \|_2 \leq t } } \inprod{\epsilon}{\theta - \thetastar}
\leq \EE \sup_{
\substack{\theta \in
\Mspace(\lattice_{d, n}; \bl_1,
\ldots, \bl_d)  \\
\| \theta - \thetabar \|_2 \leq t } }  \inprod{\epsilon}{\theta - \thetabar}
\leq t \sqrt{s},
\end{align*}
thereby completing the proof.
\end{proof}

\subsubsection{Proof of lower bound}

Our proof proceeds in two parts: We separately establish the inequalities
\begin{subequations} \label{eq:equivalent}
\begin{align}
\minimax_{d, n}(\kfull, \sset) &\geq c_1 \cdot \frac{s}{n}  
\label{eq:unbounded-iso-risk}
\\
\minimax_{d, n}(\kfull, \sset) &\geq c_2 \cdot \frac{n_1 - k^*}{n} 
\label{eq:perm-risk}
\end{align}
for a pair of universal positive constants $(c_1, c_2)$ and each $\sset \in
\lattice_{d, n}$ and $\kfull \in \Kfull_{\sset}$.
Combining the bounds~\eqref{eq:equivalent}
yields
the
claimed lower
bound on the minimax risk.
\end{subequations}

\paragraph*{Proof of claim~\eqref{eq:unbounded-iso-risk}} We show this lower
 bound over just the set $\Mspace^{\kfull, \sset}(\lattice_{d, n})$, without the
 unknown
 permutations. In order to simplify notation, we let $\phi_\kfull:
 \real_{d, n} \to \real_{d, s_1, \ldots, s_d}$ be a map that
 that collapses each hyper-rectangular block, defined by the
 tuple~$\kfull$, of the input into a
 scalar that is equal to the average of the entries within that 
 block. By construction, for an input tensor $\theta
 \in \Mspace^{\kfull, \sset}(\lattice_{d, n})$, we have the
 inclusion $\phi_\kfull(\theta) \in
 \Mspace(\lattice_{d, s_1, \ldots, s_d})$. Let $\phi^{-1}_{\kfull}: \real_{d,
 s_1, \ldots, s_d} \to
 \real_{d, n}$ denote the inverse (``lifting") map obtained by
 populating
 each block of the output with identical entries. Furthermore, for each $x \in
 \lattice_{d, s_1, \ldots, s_d}$, let $b(x)$ denote the cardinality of block $x$
 specified by the
 tuple~$\kfull$. We now split the proof into two cases, depending on the value of~$s$.

\noindent \underline{Case $s < 32$:} The proof for this case follows from
considering
the case $s = 1$. In particular, let $\kfull_1$ denote the tuple $((n_1),
\ldots, (n_d))$ corresponding to a single indifference set along all dimensions,
with $\sset^{(1)} \defn (1, \ldots, 1)$ denoting the corresponding tuple of
indifference set cardinalities along the $d$ dimensions.
Note that $\Mspace^{\kfull_1, \sset^{(1)}}(\lattice_{d, n})$ consists of all
constant tensors on the lattice. Clearly, we have the inclusion $\Mspace^
{\kfull_1, \sset^{(1)}}(\lattice_{d, n}) \subseteq \Mspace^
{\kfull, s} (\lattice_{d, n})$,
and the estimation problem over the class of tensors $\Mspace^{\kfull_1, \sset^{(1)}}
(\lattice_{d, n})$
is equivalent to estimating a single scalar parameter from $n$ i.i.d.
observations
with standard Gaussian noise. The minimax lower bound of order $1 / n$ is
classical, and adjusting the constant factor completes the proof for this case.

\noindent \underline{Case $s \geq 32$:}
In this case, we construct a packing of the set $\Mspace
 (\lattice_{d, s_1, \ldots, s_d})$ and lift this packing into the space of
 interest.
 First,
let $\alpha_0 \in \Mspace(\lattice_{d, s_1, \ldots, s_d})$ denote a base tensor
having
entries
\begin{align*}
\alpha_0(i_1, \ldots, i_d) =  \sum_{j = 1}^d (i_j - 1) \quad \text{
for each } i_j \in [s_j], \; j \in d.
\end{align*} 
The
 Gilbert--Varshamov bound~\citep{Gil52,Var57} guarantees the existence a set of
 binary tensors on the lattice\footnote{For a set $\mathbb{X}$, we use $\mathbb{X}^{\lattice_{d, s_1, \ldots, s_d}}$ to denote the collection of vectors defined on the lattice $\lattice_{d, s_1, \ldots, s_d}$, with each entry of the vector belonging to the set $\mathbb{X}$.}
 $\Omega \subseteq \{ 0, 1\}^{\lattice_{d, s_1, \ldots, s_d}}$ such that the Hamming distance between each pair of distinct vectors $\omega, \omega'
 \in \Omega$ is lower bounded as $\dH(\omega, \omega') \geq \sbar/4$
 and
 \[
 |\Omega | \geq \frac{2^s}{\sum_{i = 0}^{s/4} \binom{s}{i}} \geq e^{s / 8}.
 \]
In deriving the final inequality, we have used Hoeffding's inequality on the
lower tail of the distribution $\BIN(s, 1/2)$ to  deduce that $\sum_{i = 0}^{(s
- \alpha)/2} \binom{s}
 {i} \leq 2^s \exp(- \alpha^2 / 2 s)$ for each $\alpha \geq 0$; see,
 also,~\citet[Lemma 4.7]{massart2007concentration}.

We now use the set $\Omega$ to construct a packing over $\Mspace
 (\lattice_{d, s_1, \ldots, s_d})$: For a scalar $\delta \in (0, 1]$ to be
 chosen shortly,
define
 \begin{align*}
 \alpha^\omega (x) \defn \alpha_0 (x) + \omega(x) \cdot \frac{\delta}{\sqrt{b
 (x)}}
 \quad  \text{ for each } \;\;
 x \in \lattice_{d, s_1, \ldots, s_d}. 
 \end{align*}
By construction, the
 inclusion $\alpha^\omega \in \Mspace(\lattice_{d, s_1, \ldots, s_d})$ holds for each
 $\omega \in \Omega$. Finally, define the tensors $\theta^{\omega} \defn \phi^
 {-1}_
 {\kfull}(\alpha^{\omega})$ for each $\omega \in \Omega$. Note that $\theta^
 {\omega} \in \Mspace^{\kfull, \sset}(\lattice_{d, n})$ for each $\omega \in
 \Omega$, and
 also that
 \begin{align*}
\| \theta^{\omega} - \theta^{\omega'} \|_2^2 = \delta^2 \cdot \dH(\omega,
\omega') \text{ for each distinct pair } \omega, \omega' \in \Omega.
 \end{align*}
Thus, we see that we have constructed a local packing
 $\{\theta^{\omega} \}_{\omega \in \Omega}$ with \mbox{$\log (|\Omega|) \geq 
\sbar/8$} such that
 \begin{align*}
\frac{\sbar}{4} \delta^2 \leq \| \theta^{\omega} - \theta^{\omega'} \|_2^2 \leq
\sbar \delta^2 \text{ for each distinct pair } \omega, \omega' \in \Omega.
 \end{align*}
Employing Fano's method (see, e.g.,~\citet[Proposition 15.12 and equation
(15.34)]{wainwright2019high}) then yields, for a universal positive constant $c
$, the minimax risk lower bound
 \begin{align*}
\inf_{\thetahat \in \Thetahat} \sup_{\thetastar \in \Mspace^{\kfull, \sset}
(\lattice_{d, n})}
\Rspace_n (\thetahat, \thetastar) \geq c \cdot \delta^2 \sbar \left( 1 -
\frac{\delta^2 \sbar + \log 2}{\log (|\Omega|)} \right) \geq c \cdot 
\delta^2 \sbar \left( 3/4 -
8\delta^2\right),
 \end{align*}
 where we have used the fact that $s \geq 32$ in order to write
 $\frac{\log 2}{s/8} \leq 1/4$.
Choosing $\delta = 1/4$ completes the proof.
 \qed

\paragraph*{Proof of claim~\eqref{eq:perm-risk}} The proof of this claim uses
the unknown permutations defining the model in order to construct a packing. Related
ideas have appeared in the special case $d = 2$~\citep{ShaBalWai16-2}.
Let us begin by defining some notation. 
For any tuple $\kset \in \cup_{i = 1}^{n_1} \Kset_i$, let 
\begin{align*}
\kfull_{j}(\kset) &= ((n_1),
\ldots, (n_{j - 1}),
\kset, (n_{j + 1}), \ldots, (n_d) ) \text{ and } \\
\sset_{j}(\kset) &= (\; 1, \quad \ldots \quad , 1 \;, \;\;\; |\kset| \;, \; 1,
\quad \ldots \quad , 1 \; ),
\end{align*}
respectively. In words, these denote the size tuple and cardinality tuple
corresponding to a single indifference set along all dimensions except the
$j$-th, along which we have $|\kset|$ indifference sets with cardinalities
given by the tuple $\kset$.

Turning now to the problem at hand,
consider $\kfull \in \Kfull_{\sset}$, and let $j^* \in
\argmin_{j \in [d]} k^j_{\max}$ be any index that satisfies $k^{j^*}_{\max} =
 k^*$. 
Let $\widetilde{\kfull} = \kfull_{j^*}(\kset^{j^*})$, and $\widetilde{\sset}
= \sset_{j^*}(\kset^{j^*})$. Finally, define the special tuple
$\sset^{(2)} \in \Natural^d$ by specifying, for
each $j \in [d]$, its $j$-th entry as
\begin{align*}
[\sset^{(2)}]_j \defn
\begin{cases}
2 \quad &\text{ if } j = j^* \\
1 &\text{ otherwise}.
\end{cases}
\end{align*}
By definition, we
have 
\begin{align} \label{eq:inclusion}
\Mspace_ {\perm}^{\widetilde{\kfull}, \widetilde{\sset}}(\lattice_{d, n})
\subseteq \Mspace_
{\perm}^{\kfull, \sset}(\lattice_{d, n}).
\end{align} 
We require Lemma~\ref{lem:two-sets}, which is stated and proved at the end of
this subsection, and split the proof into two cases depending on a 
property of the size tuple $\kfull$. 

\noindent \underline{Case $k^* > n_1/3$:} In this case, set $\kset = (k^*, n_1
- k^*)$
and note the inclusion
\[
\Mspace^{\kfull_{j^*}(\kset), \sset^{(2)}}_{\perm}(\lattice_
{d, n}) \subseteq \Mspace^{\widetilde{\kfull}, \widetilde{\sset}}_{\perm} 
(\lattice_{d,
n}).
\] 
Applying Lemma~\ref{lem:two-sets} in conjunction with the further
inclusion~\eqref{eq:inclusion} then yields the bound
\[
\minimax_{d, n}(\kfull, \sset) \geq \frac{c}{n} \min\{ k^*, n_1 - k^*\} \geq 
\frac{c}{2n} (n_1 - k^*),
\]
where in the final inequality, we have used the bound $k^* > n_1 / 3$.

\noindent \underline{Case $k^* \leq n_1/3$:} In this case, note that the
largest indifference set defined by the size tuple $\kset^{j^*}$ is at most
$n_1/3$. Consequently, these indifference sets can be combined to form two
indifference sets of sizes $(\ktil, n_1 - \ktil)$ for some $n_1 / 3 \leq \ktil
\leq
2n_1 / 3$.
Now letting $\widetilde{\kset} \defn (\ktil, n_1 - \ktil)$, we have
\[
\Mspace^{\kfull_{j^*}(\widetilde{\kset}), \sset^{(2)}}_{\perm}(\lattice_
{d, n}) \subseteq \Mspace^{\widetilde{\kfull}, \widetilde{\sset}}_{\perm} 
(\lattice_{d,
n}),
\]
and
proceeding as before completes the proof for this case.
\qed


\begin{lemma} \label{lem:two-sets}
\sloppy
Suppose $\kset = (k_1, k_2)$, with $\kbar \defn \max_{\ell = 1, 2} k_\ell$, and
let
$
\mathcal{K}_{\perm} \defn \Mspace^{\kfull_{j}(\kset), \sset_j(\kset)}_{\perm}
(\lattice_
{d, n})$ for convenience.
Then, for each $j \in [d]$,
we have
\begin{align*}
\inf_{\thetahat \in \Thetahat} \sup_{\thetastar \in \mathcal{K}_{\perm}}
\Rspace_n (\thetahat, \thetastar) \geq c \cdot \frac{n_1 - \kbar}{n},
\end{align*}
where $c$ is a universal positive constant.
\end{lemma}
\begin{proof} 
Suppose wlog that $k_2 \leq
k_1$, so
that it suffices
to prove the bound
\begin{align*}
\inf_{\thetahat \in \Thetahat} \sup_{\thetastar \in \mathcal{K}_{\perm}}
\Rspace_n (\thetahat, \thetastar) \geq c \cdot \frac{k_2}{n}.
\end{align*}
Also note that by symmetry, it suffices to prove the bound for $j = 1$.

\sloppy
\noindent \underline{Case $k_2 < 32$:} From the proof of
claim~\eqref{eq:unbounded-iso-risk}, recall the set $\Mspace^{\kfull_1, \sset^{(1)}}(\lattice_{d, n})$, noting
the inclusion $\Mspace^{\kfull_1, \sset^{(1)}}(\lattice_{d, n}) \subseteq \mathcal{K}_
{\perm}$. From the same proof, we thus have the bound
$\inf_{\thetahat \in \Thetahat} \sup_{\thetastar \in \mathcal{K}_{\perm}}
\Rspace_n (\thetahat, \thetastar) \geq c \cdot \frac{1}{n}$, which suffices
since the constant factors can be adjusted appropriately.

\noindent \underline{Case $k_2 \geq 32$:}
Define the set $\Kspace \defn \Mspace^{\kfull_{j}(\kset),
\sset_j(\kset)}(\lattice_{d, n})$ for convenience.
As before, we use the Gilbert--Varshamov bound~\citep{Gil52,Var57} to claim that
there must exist a set
of
binary vectors $\Omega \subseteq \{0, 1\}^{k_2}$ such that $\log (|\Omega|)
\geq k_2/8$ and $\dH(\omega,\omega') \geq k_2/4$ for each distinct $\omega,
\omega' \in \Omega$. For a positive scalar $\delta$ to be specified
shortly, construct a base tensor $\theta_0 \in \Kspace$ by specifying its entries as
\begin{align*}
\theta_0 (i_1, \ldots, i_d) \defn
\begin{cases}
\delta \quad &\text{ if } i_1 \leq k_2 \\
0 \quad &\text{ otherwise.}
\end{cases}
\end{align*}
Now for each $\omega \in \Omega$, define the tensor $\theta^\omega \in \real_
{d, n}$ via
\begin{align*}
\theta^{\omega}(i_1, \ldots, i_d) \defn 
\begin{cases}
\delta \cdot \omega_{i_1} \quad &\text{ if } i_1 \leq k_2 \\
\delta \cdot (1 - \omega_{n_1 - i_1 + 1}) \quad &\text{ if } n_1 - k_2 + 1
\leq
i_1 \leq n_1 \\
0 &\text{ otherwise.}
\end{cases}
\end{align*}
Since $k_2 \leq k_1$, we have $n_1 - k_2 \geq k_2$. Thus, an equivalent way
to construct the tensor $\theta^
{\omega}$ is to specify a permutation using the vector $\omega$ (which flips
particular entries depending of the value of $\omega$ on that entry), and then
apply this permutation along the first dimension of $\theta_0$. Consequently, we have $\theta^{\omega} \in
\Kspace_{\perm}$ for each $\omega
\in \Omega$.
Also, by construction, we have $\| \theta^{\omega} - \theta^{\omega'} \|_2^2 =
\delta^2 \dH(\omega, \omega')$, so that 
the packing over the Hamming cube ensures that
\[
\frac{k_2}{4} \cdot \delta^2 \leq \| \theta^{\omega} - \theta^{\omega'} \|_2^2
\leq k_2
 \delta^2  \text{ for
all distinct pairs } \omega, \omega' \in \Omega. 
\]
Thus, applying Fano's method as in the proof of the claim~\eqref{eq:unbounded-iso-risk} yields,
for a small enough universal
constant $c > 0$, the bound
\begin{align*}
\inf_{\thetahat \in \Thetahat} \sup_{\thetastar \in \Kspace_{\perm}}
\Rspace_n (\thetahat, \thetastar) \geq c \cdot \delta^2 \frac{k_2}{n} \left( 1 -
\frac{\delta^2 k_2 + \log 2}{k_2/8} \right);
\end{align*}
choosing $\delta$ to be a small enough constant and noting that $k_2 \geq
32$ completes the proof.
\end{proof}

\subsection{Proof of Corollary~\ref{cor:adapt-noperms}}

We establish the two parts of the corollary separately.

\subsubsection{Proof of part (a)}
The lower bound follows immediately from our
proof of claim~\eqref{eq:unbounded-iso-risk}. 
Let us prove the upper bound. First, note that the set $\Mspace^{\kfull, \sset}
(\lattice_{d, n}) - \Mspace^{\kfull, \sset}
(\lattice_{d, n})$ is star-shaped and non-degenerate (see Definition~\ref{def:star-shaped}).
%
Thus, it suffices, as in the proof of Proposition~\ref{prop:full-funlim}, to bound the expectation of
the random variable
\begin{align*}
\xi(t) = \sup_{\substack{\theta_1, \theta_2 \in
\Mspace^{\kfull, \sset}(\lattice_{d, n}) \\ 
\| \theta_1 - \theta_2 \|_2 \leq t } } \inprod{\epsilon}{\theta_1 - \theta_2}.
\end{align*}
Applying Lemma~\ref{lem:const-blocks-ep} from the appendix yields
\begin{align*}
\EE[\xi(t)] \leq t \sqrt{s},
\end{align*}
and so the smallest positive solution $t_n$ to the critical inequality $\EE [\xi(t_n)] \leq \frac{t^2_n}{2}$ satisfies 
\mbox{$t^2_n \leq 4s$}. This completes the proof.
\qed

\subsubsection{Proof of part (b)}

The lower bound follows directly from the corresponding lower bound in part (a)
of the corollary. In order to
establish  the upper bound, we
use an argument that is very similar to the proof of the
corresponding upper bound in Proposition~\ref{prop:adapt-funlim}, and so we
sketch
the differences.

First, note that $\Mspace^s(\lattice_{d, n})$ can be written as the
union of convex sets. For convenience, let $\phi_s \defn \{\sset \in \lattice_
{d, n}: \prod_{j = 1}^d s_j = s\}$.
Proceeding as in the proof of Proposition~\ref{prop:adapt-funlim}, we see that
it suffices
to control the
expectation
of the random variable
\begin{align*}
\xi(t) = \max_{\sset \in \phi_s} \;\; \max_
{ \kfull \in
\Kfull_{\sset} } \; \sup_{
\substack{\theta \in
\Mspace^{\kfull, \sset}(\lattice_{d, n}) \\
\| \theta - \thetastar \|_2 \leq t } } \inprod{\epsilon}{\theta - \thetastar}.
\end{align*}
Now note that $|\Kfull_{\sset}|$ can be bounded by counting, for each $j \in [d]$,
the number of $s_j$-tuples of positive integers whose sum is $n_1$. A
stars-and-bars argument thus yields
$|\Kfull_{\sset}| = \prod_{j = 1}^d \binom{n_1 - 1}{s_j - 1} \leq n_1^s$.
Simultaneously, we also
have $|\phi_s| \leq s^d$. As a consequence, we see that $\xi(t)$ is
the maximum of at most $K \defn n_1^s \cdot s^d$ random variables. Note that
$\log K = s \log n_1 + d \log s \lesssim s \log n$, where we have used the fact
that $d \log s \lesssim d s \log n_1 = s \log n$.

Applying Lemma~\ref{lem:sup-emp}(a) from the appendix in conjunction with
Lemma~\ref{lem:fixed-partition-ep} yields, for each $u \geq
0$, the tail bound
\begin{align*}
\Pr \left\{ \xi(t) \geq t \sqrt{s} +
C t \left( \sqrt{s \log n} + \sqrt{u} \right) \right\} \leq e^{-u}.
\end{align*}
The rest of the proof is identical to the proof of Proposition
\ref{prop:adapt-funlim}, and putting it all together yields the claim. 
\qed

\subsection{Proof of Theorem~\ref{thm:hardness}} \label{sec:pf-hardness}
Let us begin with a high-level sketch of the proof, which proceeds by contradiction. We assume
that we are given access to a polynomial-time estimator $\thetahat$ having small
adaptivity index, and to a random hypergraph $G$ generated from either the null hypothesis $H_0$ or the alternative hypothesis $H_1$ of Conjecture~\ref{conj:HPC}. We apply a polynomial-time algorithm to $G$ in order to generate an ensemble of tensor detection problems, and then apply the estimator $\thetahat$ to this ensemble in order to distinguish 
the hypotheses $H_0$ and $H_1$ even when the size of the planted clique $K$ is small. 
Our polynomial-time reduction proceeds roughly as follows: Given access to $G$, we transform this data into a random, (near-)Gaussian tensor.
In particular, this transformation \emph{does not} require knowledge of the underlying hypothesis from
which the hypergraph is generated, but produces a zero mean tensor under hypothesis $H_0$, and a tensor with suitably structured mean under hypothesis $H_1$. The
crucial aspect of this structure that we will leverage is that the mean tensor under $H_1$ always belongs to the set 
$\Mspace_{\perm}(\lattice_{d, n})$, and in addition, has a small number of indifference sets along each dimension.
Applying the estimator $\thetahat$ to our noisy tensor then yields an effective \emph{denoising} procedure, since the estimator $\thetahat$ has a small adaptivity index and can therefore take advantage of such structure in the mean. Finally, we apply a simple threshold test on the denoised tensor to distinguish the two hypotheses.

Let $E(G)$ represent the edge set of hypergraph $G$.
Recall our notation $\mathbb{H}_{D, N}$ for the set of all $D$-regular hypergraphs on the vertex set $[N]$. Finally, we require the notion of a Gaussian rejection kernel, due to~\citet{brennan2018reducibility}, which is used to transform Bernoulli random variables into near-Gaussian random variables. This device is provided in Appendix~\ref{app:hardness} for convenience, along with the associated guarantee on the total variation distance between the output and suitably defined Gaussians.

We are now ready to describe our algorithm that produces a polynomial-time test to distinguish the hypotheses $H_0$ and $H_1$.

\medskip

\noindent \hrulefill

\noindent \underline{\bf Algorithm: Detection via adaptation (DA)} \\

\smallskip
\noindent {\bf Input:} $G \in \mathbb{H}_{d, N}$ with $N = n_1 \cdot d$, clique size $K$ in the hypothesis $H_1$. Polynomial time estimator $\thetahat_n: \Tspace_{d, n} \to \Tspace_{d, n}$; scalar $\rho$ defining the elevation level for the rejection kernel.
\begin{itemize}
\item {\bf Step I:} Form the tensor $\widetilde{Y} \in \Tspace_{d, n}$ via the following procedure. Set
\begin{align*}
\widetilde{Y}(i_1, \ldots, i_d) = \ind{  (i_1, i_2 + n_1, \ldots, i_d + (d - 1) n_1) \in E(G) } \quad \text{ for each } i_1, \ldots, i_d \in [n_1].
\end{align*}

\item {\bf Step II:} Apply the Gaussian rejection kernel to each entry of $\widetilde{Y}$
in order to obtain the tensor $Y \in \Tspace_{d, n}$ (see Lemma~\ref{lem:GRK-tensor} in Appendix~\ref{app:hardness}).

\item {\bf Step III:} Apply the estimator $\thetahat_n$ to the tensor $Y$ in order to obtain the estimate $\thetahat = \thetahat_n(Y)$. 

\item {\bf Step IV:} Let $\widetilde{K} = K / (2d)$ and compute the test $\psi' = \ind{ \| \thetahat \|_2^2 \geq \rho^2 \widetilde{K}^d  / 4 }$.
\end{itemize}
{\bf Output:} Hypothesis test $\psi^{\mathsf{DA}}_n(G) = \psi'$.

\noindent \hrulefill

Step I of this algorithm is by now standard in the literature on average-case reductions, and allows the breaking of symmetry by ``zooming in" to a portion of the graph on which there are no self-edges; see, e.g.,~\citet{berthet2013optimal,ma2015computational,luo2020tensor}. In step II, we transform our Bernoulli data into Gaussians via the rejection kernel technique~\citep{brennan2018reducibility}. Finally, steps III and IV use the estimator $\thetahat_n$ to
denoise the tensor $Y$ and produce a threshold test, respectively. Note that all four steps of the algorithm can be computed in time polynomial in $n$. Theorem~\ref{thm:hardness} is then an immediate consequence of the following proposition and Conjecture~\ref{conj:HPC}; recall our notation for the error $\Espace(\psi)$ of a test $\psi$ from equation~\eqref{eq:test-error}.

\begin{proposition} \label{prop:equiv-AI}
Suppose the adaptivity index of the sequence of estimators $\{\thetahat_n\}_{n \geq 1}$ satisfies 
\begin{align*}
\limsup_{n \to \infty} \frac{\log \adapt(\thetahat_n)}{\log n^{\frac{1}{2} (1 - 1/d)}} \leq 1 - \tau \quad \text{ for some } 0 < \tau \leq 1,
\end{align*}
the clique sizes satisfy $\liminf_{N \to \infty} \frac{\log K}{\log \sqrt{N}} \geq 1 - \tau/2$, and $\rho = \frac{\log 2}{2 \sqrt{6 (d + 1) \log n_1 + 2\log 2}}$. Then, we have $\limsup_{n \to \infty} \Espace (\psi^{\mathsf{DA}}_n (G) ) \leq 1/3$.
\end{proposition}

We dedicate the rest of this section to a proof of Proposition~\ref{prop:equiv-AI}, beginning with some notational setup. We use the shorthand $\binom{[n]}{s} \defn \{S \subseteq [n] \; \mid \; |S| = s \}$ to represent all subsets of $[n]$ of size $s$. 
For a set $S$ of non-negative integers and a positive integer $z$, let $S \ominus z \defn \{s - z \; \mid \; s \in S \}$. Let $\Vspace \subset [N]$ represent the (random) set of vertices chosen to form the clique under hypothesis $H_1$; in particular, under $H_1$, we have $\Vspace \sim \mathsf{Unif}  \binom{[N]}{K}$. Note that the random set of vertices $\Vspace$ induces, for each $j \in [d]$, the sets
\begin{align*}
\widetilde{\Vspace}_j \defn \Vspace \cap \{ n_1(j - 1) + 1, \ldots, n_1 j \} \quad \text{ and } \quad \Vspace_j \defn \widetilde{\Vspace}_j \ominus n_1(j - 1).
\end{align*}
In words, the set $\widetilde{\Vspace}_j$ consists of the (random) vertices within the set $\{ n_1(j - 1) + 1, \ldots, n_1 j \}$ that are also in the clique, and $\Vspace_j \subseteq [n_1]$ ``recenters'' the labels of these vertices. 

By construction of the tensor $\widetilde{Y}$ in the DA algorithm, the sets $\Vspace_1, \ldots, \Vspace_d$ provide locations of the signal in $\widetilde{Y}$.
Under $H_0$, the tensor $\widetilde{Y}$ has i.i.d. entries drawn from the distribution $\BER(1/2)$. Under $H_1$, we have
$\widetilde{Y}(i_1, \ldots, i_d) = 1$  if $(i_1, \ldots, i_d) \in \Vspace_1 \times \cdots \times \Vspace_d$ and $\widetilde{Y}(i_1, \ldots, i_d) \sim \BER(1/2)$ otherwise. In step II of the DA algorithm, we use the Gaussian rejection kernel to then transform this Bernoulli tensor into a near-Gaussian tensor, where the mean of a particular entry of $Y$ is elevated (i.e., large) if the corresponding entry of $\widetilde{Y}$ was equal to $1$. See Lemma~\ref{lem:GRK-tensor} and the surrounding discussion in Appendix~\ref{app:hardness} for details. In particular, the mean of any such entry is given by  the scalar 
\begin{align} \label{eq:rho-def-2}
\rho \defn \frac{\log 2}{2 \sqrt{6 (d + 1) \log n_1 + 2\log 2}}. 
\end{align}

Now for each tuple $V_1, \ldots, V_d \subseteq [n_1]$, define the tensor $\mu_{V_1, \ldots, V_d} \in \Tspace_{d, n}$ via
\begin{align*}
\mu_{V_1, \ldots, V_d} (i_1, \ldots, i_d) = 
\begin{cases}
\rho \qquad&\text{ if } (i_1, \ldots, i_d) \in V_1 \times \cdots \times V_d \\
0 &\text{ otherwise.}
\end{cases}
\end{align*}
When $\Vspace_j = V_j$ for all $j \in [d]$, the rejection kernel ensures that the mean of $Y$ is roughly equal to $\mu_{V_1, \ldots, V_d}$; once again, see Lemma~\ref{lem:GRK-tensor} in the appendix.
With these definitions in hand, we are now ready to state the main technical lemma
used in the proof of Proposition~\ref{prop:equiv-AI}.

\begin{lemma} \label{lem:adapt-suff}
Let $\epsilon \in \Tspace_{d, n}$ denote a standard Gaussian tensor, and let $V_1, \ldots, V_d$ denote subsets of $[n_1]$ such that $|V_j| \leq \kappa$ for all $j \in [d]$. There is a universal positive constant $C$ such that the following statements hold true:

(a) Let $\thetahat = \thetahat_n (\mu_{V_1, \ldots, V_d} + \epsilon)$. Then
\begin{align*}
\| \thetahat - \mu_{V_1, \ldots, V_d} \|_2^2 \leq C \cdot \adapt(\thetahat_n) \cdot \big( 2^d + \kappa \big) \cdot \log n
\end{align*}
with probability at least $2 / 3$.

(b) Let $\thetahat = \thetahat_n(\epsilon)$. Then
\begin{align*}
\| \thetahat  \|_2^2 \leq C \cdot \adapt(\thetahat_n) \cdot \log n
\end{align*}
with probability at least $2 / 3$.
\end{lemma}
The proof of this lemma is postponed to the end of this section. For now, we take it as given and proceed to the proof of Proposition~\ref{prop:equiv-AI}.

\paragraph*{Proof of Proposition~\ref{prop:equiv-AI}} Recall that $\widetilde{K} = K / (2d)$. Throughout this proof, we suppose that the adaptivity index of $\thetahat_n$ and the normalized size of the planted clique $\widetilde{K}$ satisfy, for some $0 < \tau \leq 1$, the bounds
\begin{align} \label{eq:scale-bounds} 
\limsup_{n \to \infty} \frac{\log \adapt(\thetahat_n)}{\log n^{\frac{1}{2} (1 - 1/d)}} \leq 1 - \tau \quad \quad \text{ and } \quad \quad \liminf_{n \to \infty} \frac{\log \widetilde{K}}{\log \sqrt{n_1 d}} \geq 1 - \tau/2,
\end{align}
respectively. Now, for each $\tau \in (0, 1]$, we have
\begin{align*}
\frac{\adapt(\thetahat_n) \cdot \log n}{ \rho^2 \widetilde{K}^{d-1}} = \frac{\adapt(\thetahat_n)}{n^{\frac{1}{2}(1 - 1/d) (1 - \tau)}} \cdot \frac{n^{\frac{1}{2}(1 - 1/d) (1 - \tau)}}{(n_1 d)^{\frac{1}{2}(d-1)(1 - \tau/2)}} \cdot \left( \frac{(n_1 d)^{\frac{1}{2}{(1 - \tau/2)}}}{\widetilde{K}}\right)^{d-1} \cdot \frac{\log n}{\rho^2}.
\end{align*}
Note that by our choice of $\rho$ in equation~\eqref{eq:rho-def-2}, we have $\frac{\log n}{\rho^2} \leq C$. Thus, in the 
limit $n \to \infty$, the assumed relations~\eqref{eq:scale-bounds} yield
\begin{align} \label{eq:key-cond-LB}
\lim_{n \to \infty} \frac{\adapt(\thetahat_n) \cdot \log n}{ \rho^2 \widetilde{K}^{d-1}} = 0.
\end{align}
Our goal is to show that condition~\eqref{eq:key-cond-LB} ensures that the test $\psi_n^{\mathsf{DA}}(G)$ has small error. We split the rest of the proof into the null case and the alternative case.

\smallskip 

\noindent \underline{Null case:} Under hypothesis $H_0$, Lemma~\ref{lem:GRK-tensor} in the appendix yields the total variation bound\footnote{We use $\mathcal{L}(X)$ to denote the law of $X$.}
$
\TV( \Lspace(Y), \Lspace(\epsilon)) \leq C/ n_1,
$
and so combining with Lemma~\ref{lem:adapt-suff}(b), we have
\begin{align} \label{eq:rando}
\| \thetahat  \|_2^2 = \| \thetahat_n(Y) \|_2^2 \leq C  \adapt(\thetahat_n) \cdot \log n
\end{align}
with probability at least $2/3 - C/n_1$. Note that this bound holds for finite $n$.
It remains to show that such a bound suffices for step IV of the algorithm to produce a good test. Owing to the conditions~\eqref{eq:key-cond-LB}, 
for some large enough $n$ depending on the value of the tuple $(d, \tau)$, we must have $C \adapt(\thetahat_n) \cdot \log n < \rho^2 \widetilde{K}^d / 4$. Consequently, when inequality~\eqref{eq:rando} holds and for large enough $n$, we have $\psi' = 0$.
Thus, we have the bound
$
\limsup_{n \to \infty} \EE_{H_0} [ \psi_n^{\mathsf{DA}} (G) ] \leq 1/3.
$

\smallskip

\noindent \underline{Alternative case:} The proof for this case is slightly more involved.
As a first step, we show that the random variables $|\Vspace_1|, \ldots, |\Vspace_d|$ concentrate around $2\widetilde{K}$. Note that $|\Vspace_1|$ is drawn from the hypergeometric distribution $\Hyp(N/ d, N, K)$; applying Lemma~\ref{lem:HG} from the appendix in conjunction with a union bound yields that the event
\begin{align} \label{eq:vertex-sets-large}
\Omega_n^{(1)} \defn \left\{ \widetilde{K} \leq |\Vspace_j| \leq 3 \widetilde{K} \text{ for all } j \in [d] \right\},
\end{align}
occurs with probability at least $1 - 2d e^{-c \widetilde{K}}$ for some universal positive constant $c$. 

Recall that the tensor $Y$ was produced in step II of the DA algorithm. By Lemma~\ref{lem:adapt-suff}(a) and Lemma~\ref{lem:GRK-tensor}, we have that the event
\begin{align} \label{eq:two-lemmas}
\Omega_n^{(2)} \defn \left\{ \| \thetahat(Y) - \mu_{\Vspace_1, \ldots, \Vspace_d} \|_2^2 \leq C \adapt(\thetahat_n) \cdot \bigl( 2^d + \max_{j \in [d]} |\Vspace_j| \bigr) \log n \right\}
\end{align}
occurs with probability exceeding $2/3 - C/n_1$. 
Under hypothesis $H_1$ and on the event $\Omega_n^{(1)} \cap \Omega_n^{(2)}$, we may write
\begin{align*}
\| \thetahat(Y) \|_2^2 &\stackrel{\1}{\geq} \frac{1}{2} \| \mu_{\Vspace_1, \ldots, \Vspace_d} \|_2^2 - \| \thetahat(Y) - \mu_{\Vspace_1, \ldots, \Vspace_d} \|_2^2 \\
&\stackrel{\2}{\geq} \frac{1}{2} \rho^2 \prod_{j = 1}^d |\Vspace_j| - C \adapt(\thetahat_n) \cdot \bigl( 2^d + \max_{j \in [d]} |\Vspace_j| \bigr) \log n \\
&\stackrel{\3}{\geq} \frac{1}{2} \rho^2 \widetilde{K}^d - C \adapt(\thetahat_n) \cdot \bigl(2^d + 3 \widetilde{K} \bigr) \log n \\
&=: \frac{1}{4} \rho^2 \widetilde{K}^d + \Delta_n.
\end{align*}
Here, step $\1$ follows from Young's inequality $\| u + v \|_2^2 \leq 2 \| u \|_2^2 + 2 \|v \|_2^2$. Step $\2$ holds by definition of the tensor $\mu_{\Vspace_1, \ldots, \Vspace_d}$ and equation~\eqref{eq:two-lemmas}. Step $\3$ follows from the relation~\eqref{eq:vertex-sets-large}. It is now straightforward to verify that owing to condition~\eqref{eq:key-cond-LB}, we have $\liminf_{n \to \infty} \Delta_n > 0$. Thus, on the event $\Omega_n^{(1)} \cap \Omega_n^{(2)}$, we reject the null in step IV of the DA algorithm.

Consequently, the reverse Fatou lemma yields
\begin{align*}
&\limsup_{n \to \infty} \; \EE_{H_1} [1 - \psi_n^{\mathsf{DA}} (G)] \\
&\leq \limsup_{n \to \infty} \; \Pr \left\{ (\Omega_n^{(1)} \cap \Omega_n^{(2)})^c \right\} + \limsup_{n \to \infty} \; \EE_{H_1} \left[ (1 - \psi_n^{\mathsf{DA}} (G) ) \cdot \ind{ \Omega_n^{(1)} \cap \Omega_n^{(2)} } \right] \\
&\leq   \limsup_{n \to \infty} \; \big(1/3 + C/n_1 + 2d e^{-c\widetilde{K}} \big) \\
&= 1/3. 
\end{align*}

\smallskip

\noindent Finally, combining the null and alternative cases, we have $\limsup_{n \to \infty} \Espace(\psi_n^{\mathsf{DA}} (G) ) \leq 1/3$, and this completes the proof.
\qed

\paragraph*{Proof of Lemma~\ref{lem:adapt-suff}} The proof of this lemma follows immediately from the definition of the adaptivity index, and from noting certain properties of the mean tensor $\mu_{V_1, \ldots, V_d}$. Both parts of the proof make use
of the following inequality, which is an application of Markov's inequality to the definition of the adaptivity index: If $\mu \in \Mspace_{\perm}^{\kfull, \sset}(\lattice_{d, n})$, then
\begin{align*}
\Pr \left\{ \| \thetahat_n(\mu + \epsilon) - \mu \|_2^2 \geq 3 n \cdot \adapt(\thetahat_n) \cdot \minimax_{d, n}(\kfull, \sset) \right\} \leq 1/3.
\end{align*}
Applying the upper bound on the minimax risk~$\minimax_{d, n}(\kfull, \sset)$ in Proposition~\ref{prop:adapt-funlim}, we then obtain
\begin{align} \label{eq:Markov-AI}
\Pr \left\{ \| \thetahat_n(\mu + \epsilon) - \mu \|_2^2 \geq C \log n \cdot \adapt(\thetahat_n) \cdot \left\{ s + (n_1 - k^*) \right\} \right\} \leq 1/3.
\end{align}
In order to prove part (a) of the lemma, note that each $\mu_{V_1, \ldots, V_d} \in \Mspace_{\perm}(\lattice_{d, n})$, and furthermore, that the number of indifference sets of this tensor along each dimension $j \in [d]$ satisfies $s_j \leq 2$. Now note that by definition, we have 
\begin{align*}
n_1 - k^* = \max_{j \in [d]} \min(n_1 - |V_j|, |V_j|) \leq \max_{j \in [d]} |V_j| \leq \kappa,
\end{align*}
where the final inequality holds by assumption.
Substituting into the bound~\eqref{eq:Markov-AI} then completes the proof. 
Part (b) of the lemma follows immediately from equation~\eqref{eq:Markov-AI},
since in this case, we have $\mu = 0$, and the all zero tensor has $s = 1$ and $k^* = n_1$. \qed

\subsection{Proof of Theorem~\ref{thm:block-risk}} \label{sec:pf-block-risk}

We first establish a certain error decomposition that results from our
algorithm, and then proceed to proofs of the two parts of the theorem.
Recall that the algorithm computes, from the score vectors, a set of estimated
ordered partitions $\blhat_1, \ldots,
\blhat_d$, and projects the observations onto the set $\Mspace(\lattice_{d, n};
\blhat_1, \ldots, \blhat_d)$.
Denote
by $\blop(\theta; \bl_1, \ldots, \bl_d)$ the tensor obtained by projecting
$\theta \in
\Tspace_{d,n}$ onto the set $\Mspace(\lattice_{d, n}; \bl_1, \ldots, \bl_d)$;
since the set is closed and convex, the projection is unique and given by
\begin{align} \label{eq:blop-def}
\blop(\theta; \bl_1, \ldots, \bl_d) = \argmin_{\thetatil \in \Mspace(\lattice_
{d, n}; \bl_1, \ldots, \bl_d)} \| \theta - \thetatil \|_2^2.
\end{align}
 Additionally, for notational convenience, let
$\Bhat \defn \blhat_1, \ldots, \blhat_d$. Given $\blhat_1, \ldots, \blhat_d$, note that since $\thetahatblock$ 
is simply a projection onto the set $\Mspace(\lattice_
{d, n}; \blhat_1, \ldots, \blhat_d)$, and so we may write $\thetahatblock = \blop
(\thetastar + \epsilon;
\Bhat)$. Applying the triangle inequality yields
\begin{align}
&\| \thetahatblock - \thetastar \|_2 \notag \\
&\leq \| \thetahatblock - \blop( \blop
(\thetastar; \Bhat) \! + \! \epsilon; \Bhat ) \|_2 + \| \blop( \blop
(\thetastar; \Bhat) + \epsilon; \Bhat ) - \blop(\thetastar; \Bhat) \|_2 + \|
\blop(\thetastar; \Bhat) - \thetastar \|_2 \notag \\
&= \| \blop( \blop
(\thetastar; \Bhat) + \epsilon; \Bhat ) - \blop(\thetastar + \epsilon;
\Bhat) \|_2 + \| \blop( \blop
(\thetastar; \Bhat) + \epsilon; \Bhat ) - \blop(\thetastar; \Bhat) \|_2 \notag \\
&\qquad \qquad \qquad \qquad \qquad \qquad \qquad \qquad \qquad \qquad \qquad \qquad \qquad + \|
\blop
(\thetastar; \Bhat) - \thetastar \|_2 \notag \\
&\leq \underbrace{\| \blop( \blop (\thetastar; \Bhat) + \epsilon; \Bhat ) -
\blop
(\thetastar;
\Bhat) \|_2}_{\text{estimation error}} + \underbrace{2 \| \blop
(\thetastar; \Bhat) - \thetastar \|_2}_{\text{approximation error}}. 
\label{eq:decomposition}
\end{align}
Here, inequality~\eqref{eq:decomposition} follows from the non-expansiveness of the $\blop$ operator, since it is an
$\ell_2$-projection onto a convex set~\eqref{contract-convex}. In particular, we have used the fact that
\begin{align*}
\| \blop( \blop
(\thetastar; \Bhat) + \epsilon; \Bhat ) - \blop(\thetastar + \epsilon;
\Bhat) \|_2 \leq \| (\blop
(\thetastar; \Bhat) + \epsilon) - (\thetastar + \epsilon) \|_2 = \| \blop
(\thetastar; \Bhat) - \thetastar \|_2.
\end{align*}
We now state three lemmas that lead to the desired bounds in the various cases.
For an ordered partition $\bl$, denote by $\card(\bl)$ the number of blocks in
the partition, and let $\kappa^*(\bl)$ denote the size of the largest block in
the partition.  Our first lemma captures some key structural
properties of the estimated ordered partitions $\blhat_1, \ldots, \blhat_d$.

\begin{lemma} \label{lem:block-structure}
Suppose that $\thetastar \in \Mspace_{\perm}^{\kfull, \sset}(\lattice_{d, n})$. Then
with probability at least $1 - 2n^{-7}$, the partition $\Bhat = \blhat_1,
\ldots, \blhat_d$ satisfies
\begin{align} \label{eq:block-structure}
\card(\blhat_j) \leq s_j \quad \text{ and } \quad \kappa^*(\blhat_j) \geq k^j_{\max}
\quad \text{ simultaneously for all } \; j \in [d].
\end{align}
\end{lemma}

Our next lemma bounds the estimation error term in two different ways.

\begin{lemma} \label{lem:est-error}
There is a universal positive constant $C$ such that for all $u \geq 0$,
each of the following statements holds with probability at least $1 - e^{-u}$: \\

(a) For any set of one-dimensional ordered partitions $\bl_1, \ldots, \bl_d$
satisfying $\card(\bl_j) \leq \widetilde{s}_j$ for all $j \in [d]$,
and any tensor $\theta \in
\Mspace(\lattice_{d, n}; \bl_1, \ldots, \bl_d)$, we have
\begin{align*}
\| \blop(\theta + \epsilon; \bl_1, \ldots, \bl_d) - \theta \|^2_2 \leq C
\left( \widetilde{s} + u \right),
\end{align*}
where
$\widetilde{s} = \prod_{j = 1}^d \widetilde{s}_j$.

(b) For any set of one-dimensional ordered partitions $\bl_1, \ldots, \bl_d$
and any tensor \linebreak
\mbox{$\theta
\in \Mspace
(\lattice_{d, n}; \bl_1, \ldots, \bl_d) \cap \infball(1)$}, we have
\begin{align*}
\| \blop(\theta + \epsilon; \bl_1, \ldots, \bl_d) - \theta \|^2_2 \leq C
(n^{1 - 1/d} \log^{5/2} n + u).
\end{align*}
\end{lemma}

Our final lemma handles the approximation error term.

\begin{lemma} \label{lem:approx-error}
There is a universal positive constant $C$ such that for all $\thetastar \in
\Mspace^{\kfull, \sset}_{\perm}(\lattice_{d, n})$, we
have
\begin{subequations}
\begin{align} 
\Pr\{ \| \blop(\thetastar; \Bhat) - \thetastar \|^2_2  \geq Cd^2 
(n_1 - k^*) \cdot n^{\frac{1}{2}(1 - 1/d)} \log n \} \leq 4n^{-7}, \text{ and }
\\
\EE \left[ \| \blop(\thetastar; \Bhat) - \thetastar \|^2_2 \right] \leq Cd^2 
(n_1 - k^*) \cdot n^{\frac{1}{2}(1 - 1/d)} \log n. \label{eq:exp-bd}
\end{align}
\end{subequations}
\end{lemma}

We prove these lemmas in the subsections to follow. For now, let us use them to
prove the two parts of Theorem~\ref{thm:block-risk}.

\subsubsection{Proof of Theorem~\ref{thm:block-risk}, part (a)}
Consider any tensor $\theta_0 \in \infball(1)$. First, by applying Lemma
\ref{lem:contraction} in the appendix, we deduce that $\blop
(\theta_0; \bl_1, \ldots, \bl_d ) \in \Mspace(\lattice_{d, n}; \bl_1, \ldots,
\bl_d) \cap \infball(1)$, since the operator $\blop$ is 
$\ell_\infty$-contractive.
Also, note from Lemma~\ref{lem:num-partitions}(a) in the appendix that the total
number of one-dimensional partitions
satisfies
$
|\Partition| = (n_1)^{n_1}.
$
Thus, we may apply Lemma~\ref{lem:est-error}(b) with the substitution $\theta =
\blop
(\theta_0; \bl_1, \ldots, \bl_d )$ and $u = n_1 \log n + u'$. In conjunction
with a union
bound
over at most $|\Partition|^d$ possible choices of one-dimensional ordered 
partitions $\bl_1, \ldots, \bl_d \in \Partition$, this yields the bound
\begin{align*}
&\max_{\bl_1, \ldots, \bl_d \in \Partition} \| \blop( \blop 
(\theta_0; \bl_1,
\ldots, \bl_d) +
\epsilon; \bl_1, \ldots, \bl_d ) -
\blop(\theta_0; \bl_1, \ldots, \bl_d) \|^2_2 \\
&\qquad \qquad \qquad \qquad \qquad \qquad \qquad \qquad \qquad \qquad \leq C
\cdot (n^{1 - 1/d} \log^{5/2} n + n_1 \log n + u')
\end{align*}
with probability at least
\[
1 - |\Partition|^d \exp \left( - n_1 \log n - u' \right) \geq
1 - (n_1)^{n_1 \cdot d} \cdot (n_1)^{-n_1 \cdot d} \cdot e^
{-u'} = 1 - e^{-u'},
\]
where we have used the fact that $\log (n_1)^{n_1 \cdot d} = n_1 \log n$ for
each $n_1 \geq 2$. Integrating
this tail bound and noting that $\blhat_1, \ldots, \blhat_d \in \Partition$, we
obtain
\begin{align*}
\EE \left[ \| \blop( \blop (\theta_0; \Bhat) + \epsilon; \Bhat ) -
\blop
(\theta_0;
\Bhat) \|^2_2 \right] \leq C n^{1 - 1/d} \log^{5/2} n
\end{align*}
for any $\theta_0 \in \infball(1)$. Choosing $\theta_0 = \thetastar$
and combining this with equation~\eqref{eq:decomposition} and Lemma~
\ref{lem:approx-error} completes the
proof.
\qed

\subsubsection{Proof of Theorem~\ref{thm:block-risk}, part (b)}

We split the proof into two cases depending on the value of $s$.

\paragraph*{Case 1} Let us first handle the case $s = 1$, in which case $n_1 -
k^* = 0$. By Lemma~\ref{lem:block-structure}, there is an event 
occurring with probability at least $1 - 2n^{-7}$ such that on this
event, our estimated
blocks satisfy $\card(\blhat_j) = 1$. On this event, the projection is a constant
tensor, and each entry of the error tensor $\blop( \blop 
(\thetastar; \Bhat) +
\epsilon; \Bhat ) -
\blop (\thetastar; \Bhat)$ is equal to $\overline{\epsilon}
\defn n^{-1} \sum_{x
\in \lattice_{d, n}} \epsilon_x$. But $\| \overline{\epsilon} \cdot \mathbbm{1}_
{d, n}
\|_2^2 \sim \chi^2_1$, and so a tail bound for the standard Gaussian yields
\begin{align*}
\Pr\{ \| \overline{\epsilon} \cdot \mathbbm{1}_
{d, n}
\|_2 \geq t \} \leq e^{-t^2 / 2}.
\end{align*}
Employing a union
bound then yields
%
$\| \blop( \blop 
(\thetastar; \Bhat) +
\epsilon; \Bhat ) -
\blop (\thetastar; \Bhat) \|^2_2 \leq 8 \log n$ with probability at
least $1 - 2n^{-7} - n^{-4}$. In order
to bound the error in expectation, first
note that since the projection onto a convex set is non-expansive~\eqref{contract-convex}, we
have the
pointwise bound
\[
\| \blop( \blop (\thetastar; \Bhat) + \epsilon; \Bhat ) -
\blop
(\thetastar;
\Bhat) \|^2_2 \leq \| \epsilon \|_2^2.
\]
Putting together the two bounds and applying Lemma~\ref{lem:RV} from the
appendix then yields the bound
\begin{align*}
\EE \left[ \| \blop( \blop (\thetastar; \Bhat) + \epsilon; \Bhat ) -
\blop(\thetastar; \Bhat) \|^2_2 \right] \leq 8 \log n + \sqrt{2 n^{-7} + n^{-4}}
\cdot
\sqrt{\EE [\| \epsilon \|_2^4]} \leq C s \log n.
\end{align*}
Combining with equation~\eqref{eq:decomposition} and Lemma
\ref{lem:approx-error} completes the
proof of the claim in expectation. The proof for this case is thus complete.

\paragraph*{Case 2} In this case, $s \geq 2$, which ensures that
$n_1 - k^* \geq 1$.
Recall the set $\Partition^{\max}_{k^*}$ defined in the proof of Proposition~\ref{prop:adapt-funlim}, and note that applying Lemma~\ref{lem:num-partitions}(b) from the appendix yields the bound
$|\Partition^{\max}_{k^*}| \leq e^{3(n_1 - k^*) \log n_1}$. With this
calculation in
hand, the proof proceeds very similarly to that of Theorem~\ref{thm:block-risk}(a). 

We first apply Lemma~\ref{lem:est-error}(a) with the substitution $u = C (n_1 -
k^*) \log n$ and take a union bound over at most
$|\Partition^{\max}_{k^*}|^d$ possible choices of ordered partitions $\bl_1, \ldots, \bl_d
\in \Partition^{\max}_{k^*}$ satisfying $\card(\bl_j) \leq s_j$ for all $j \in [d]$. This yields, 
for each $\thetastar \in \Tspace_{d, n}$, the bound
\begin{align*}
&\max_{ \substack{\bl_1, \ldots, \bl_d \in \Partition^{\max}_{k^*}: \\ \card(\bl_j) \leq s_j, \; j \in [d] }} \| \blop( \blop (\thetastar; \bl_1,
\ldots, \bl_d) +
\epsilon; \bl_1, \ldots, \bl_d ) -
\blop(\thetastar; \bl_1, \ldots, \bl_d) \|^2_2 \\
&\qquad \qquad \qquad \qquad \qquad \qquad \qquad \qquad \qquad \qquad \leq C
\cdot \left(s + (n_1 - k^*) \log n \right)
\end{align*}
with probability exceeding 
\[
1 - |\Partition^{\max}_{k^*}|^d \cdot \exp( -  C (n_1 - k^*) \log n ) \geq 1 - \exp( -  (C -
3) \cdot (n_1 - k^*) \log n ) \geq  1 - n^{-7}.
\]
Here, the last inequality can be ensured by choosing a large enough constant
$C$, since $n_1 - k^* \geq 1$.

Furthermore, Lemma~\ref{lem:block-structure} guarantees that with probability at least $1 - 2n^{-7}$,
we have \mbox{$\blhat_1,
\ldots, \blhat_d \in \Partition^{\max}_{k^*}$} and $\card(\blhat_j) \leq s_j$ for all $j \in [d]$. Consequently, by applying a union
bound, we obtain, for any $\thetastar \in
\Tspace_{d, n}$, the high probability bound
\begin{align} \label{eq:hp-error}
\Pr \left\{ \| \blop( \blop (\thetastar; \Bhat) + \epsilon; \Bhat ) -
\blop
(\thetastar;
\Bhat) \|^2_2 \geq C \cdot \left( s + (n_1 - k^*)  \log n \right) \right\} \leq
3n^{-7}.
\end{align}
Thus, we have succeeded in bounding the error
with high probability. In order to bound the error in expectation, note once
again that the projection onto a convex set is non-expansive~\eqref{contract-convex}, 
and so we
have the
pointwise bound
\[
\| \blop( \blop (\thetastar; \Bhat) + \epsilon; \Bhat ) -
\blop
(\thetastar;
\Bhat) \|^2_2 \leq \| \epsilon \|_2^2.
\]
Putting it all together and applying Lemma~\ref{lem:RV} from the
appendix then yields
\begin{align*}
\EE \left[ \| \blop( \blop (\thetastar; \Bhat) + \epsilon; \Bhat ) -
\blop(\thetastar; \Bhat) \|^2_2 \right] 
&\leq C \cdot \left\{ s + (n_1 - k^*) \log n \right\} + \sqrt{3 n^{-7}}
\cdot
\sqrt{\EE [\| \epsilon \|_2^4]} \\
&\leq C \cdot \left\{ s + (n_1 - k^*)  \log n \right\}.
\end{align*}
Combining with equation~\eqref{eq:decomposition} and Lemma
\ref{lem:approx-error} completes the
proof in this case.

\smallskip

Combining the two cases completes the proof of the theorem.
\qed

\begin{remark} \label{rem:stronger-bound}
Note that in establishing the two parts of Theorem~\ref{thm:block-risk}, we have
actually derived the high
probability bound
\begin{align*}
&\Pr \left\{ \| \blop( \blop (\thetastar; \Bhat) \! + \! \epsilon; \Bhat ) \! - \! 
\blop
(\thetastar;
\Bhat) \|^2_2 \geq C \, \min \left(
s \! + \! (n_1 \! - \! k^* \! + \! 1) \log n, n^{1- 1/d} \log^{5/2} n \right) \right\} \\
&\qquad \qquad \qquad \qquad \qquad \qquad \qquad \qquad \qquad \qquad \qquad \qquad \qquad \qquad \qquad \qquad \leq 4n^
{-4},
\end{align*}
which we use in Proposition~\ref{prop:AI-MP-bounded} in the main text.
\end{remark}

\smallskip

\noindent It remains to prove the three technical lemmas. Before we do so, we
state and
prove a claim that will be used in multiple proofs.

\subsubsection{A preliminary result} \label{sec:claim-section}
Define two events
\begin{subequations} \label{eq:hp-events}
\begin{align}
\Espace_1 &\defn \left\{ \| Y - \thetastar \|_{\infty} \leq 4 \sqrt{\log n }
\right\} \text{
and } \\
\Espace_2 &\defn \left\{ \max_{1 \leq j \leq d} \| \tauhat_j - \taustar_j \|_
{\infty}
\leq 4 \sqrt{\log n} \cdot n^{\frac{1}{2}(1 - 1/d)} \right\}.
\end{align}
\end{subequations}
Note that by applying a union bound in conjunction with Gaussian tail bounds, we have $\Pr\{ \Espace_1
\cap \Espace_2 \} \geq 1 - 2n^{-7}$. 
Recall the graphs $G'_j$ and $G_j$
obtained over the
course of running the algorithm. We use the following fact,
guaranteed by
step Ia of the algorithm.
\begin{claim} \label{clm:graph}
On the event $\Espace_1 \cap \Espace_2$, the following statements hold simultaneously
for all $j \in [d]$: \\

(a) The graph $G'_j$ is a directed acyclic graph, and consequently $G_j = G'_j$.

(b) For each $\ell \in [s_j]$ and all pairs of indices $u, v \in
I^j_\ell$, the edges $u \to v$ and $v \to u$ do not exist in the graph $G_j$.
\end{claim}

\begin{proof}
Recall our pairwise statistics $\Deltahat^{{\sf sum}}_j(u, v)$ and 
$\Deltahat^{\max}_j(u, v)$, and let $\Delta^{{\sf sum}}_j(u, v)$ and $\Delta^
{\max}_j(u, v)$ denote their population versions, that is, with $\thetastar$ replacing
$Y$ in the definition~\eqref{eq:pair-stats}. Note that by applying the triangle
inequality, we obtain
\begin{subequations} \label{eq:det-conc}
\begin{align}
|\Deltahat^{{\sf sum}}_j(u, v) - \Delta^{{\sf sum}}_j(u, v)| &\leq |\tauhat_j(u) -
\taustar_j(u)| + |\tauhat_j(v) -
\taustar_j(v)| \text{ and } \\
|\Deltahat^{\max}_j(u, v) - \Delta^{\max}_j(u, v)| &\leq 2 \| Y - \thetastar \|_
{\infty}. 
\end{align}
\end{subequations}
We now prove each part of the claim separately.

\paragraph*{Proof of part (a)} Working on the
event
$\Espace_1 \cap \Espace_2$ and using equations~\eqref{eq:hp-events} and
\eqref{eq:det-conc}, we see that
if
\begin{align*}
\Deltahat^{{\sf sum}}_j(u, v) > 8 \sqrt{\log n} \cdot n^{\frac{1}{2}(1 - 1/d)}
\quad \text{or } \quad \Deltahat^{\max}_j(u, v) > 8 \sqrt{\log n}
\end{align*}
then
\begin{align*}
\Delta^{{\sf sum}}_j(u, v) &= \taustar_j(v) - \taustar_j(u) > 0 \quad \text{ or
} \quad \Delta^{\max}_j(u, v) > 0.
\end{align*}
In particular, the relation $\Delta^{\max}_j(u, v) > 0$ implies that for some $i_\ell \in [n_\ell], \; \ell \in [d]
\setminus \{j\}$,
\[
\thetastar(i_1, \ldots, i_{j - 1}, v,
i_{j + 1},\ldots, i_d) - \thetastar(i_1, \ldots, i_{j - 1}, u,
i_{j + 1},\ldots, i_d) > 0. 
\]
In either case, we have $\pistar_j(u) < \pistar_j(v)$ by
the monotonicity
property of $\thetastar$. Thus, every edge $u \to v$ in the graph is consistent
with the permutation $\pistar_j$, and so the graph $G'_j$ is acyclic.

\paragraph*{Proof of part (b)} Note that if $u, v \in
I^j_\ell$ for some $\ell \in [s_j]$, then 
\begin{align*}
\Delta^{{\sf sum}}_j(u, v) = \Delta^{\max}_j(u, v) = 0.
\end{align*}
Therefore, on the event $\Espace_1 \cap \Espace_2$ and owing to
the inequalities~\eqref{eq:hp-events} and~\eqref{eq:det-conc}, we have
\begin{align*}
|\Deltahat^{{\sf sum}}_j(u, v)| \leq 8 \sqrt{\log n} \cdot n^{\frac{1}{2} (1 -
1/d)} \quad \text{and } \quad |\Deltahat^{\max}_j(u, v)| \leq 8 \sqrt{\log
n}.
\end{align*}
Consequently, neither of the edges $u \to v$ or $v \to u$ exists in the graph
$G_j$.
\end{proof}

We are now ready to establish the individual lemmas.

\subsubsection{Proof of Lemma~\ref{lem:block-structure}}

Suppose wlog that $\pistar_1 = \cdots =\pistar_d = \id$, so that $\thetastar
\in \Mspace^{\kfull, \sset}(\lattice_{d, n})$. This implies that the true ordered
partition $\bl^*_j$ consists of $s_j$ intervals $(I^j_1, \ldots, I^j_{s_j})$ of
sizes $(k^j_1, \ldots, k^j_{s_j})$, respectively. The largest interval, which we denote by
$I^j_{\max}$, has size $k^j_{\max}$.


By part (a) of Claim~\ref{clm:graph}, the graph $G_j$ is a directed acyclic graph, and so $\card
(\blhat_j)$ is equal to the size of the minimal partition of the graph
into disjoint antichains. But Claim~\ref{clm:graph}(b) ensures that each of the
sets $I^j_\ell, \ell \in [s_j]$ forms an antichain of graph $G_j$, and
furthermore,
these sets are disjoint and form a partition of $[n_j]$. Hence, $\card(\blhat_j)
\leq s_j$.

In order to show that $\kappa^*(\blhat_j) \geq k^j_{\max}$, note that by Claim
\ref{clm:graph}(b), the set $I^j_{\max}$ is an
antichain of $G_j$ of size~$k^j_{\max}$.
\qed

\subsubsection{Proof of Lemma~\ref{lem:est-error}}

Recall that the estimator
$\blop(\theta +
\epsilon; \bl_1, \ldots, \bl_d)$
is a projection onto a closed, convex set. Using this fact, let us prove the two parts of the lemma separately.

\paragraph*{Proof of part (a)}  
Recall that Corollary~\ref{cor:adapt-noperms}(a) already provides a bound on the
error of interest in expectation. Multiplying both sides of that guarantee by $n$ and combining it with Lemma~\ref{lem:vdG-W} 
then yields the claimed high probability bound.
\qed

\paragraph*{Proof of part (b)} Let us begin with a simple definition. Note that
any one-dimensional ordered partition $\bl$
specifies a partial ordering over the set $[n_1]$. In particular, recall that any one-dimensional ordered partition $\bl = (S_1, \ldots, S_L)$ induces a map $\sigma_
{\bl}: [n_1] \to [L]$, where  $\sigma_{\bl}(i)$ is the index~$\ell$ of the set $S_\ell \ni i$.
We say that a permutation
$\pi$ is \emph{faithful} to the ordered partition $\bl$ if it is consistent
with this partial ordering, and denote by $\Fspace(\bl)$ the set of all
permutations that are faithful to~$\bl$. Specifically,
\begin{align} \label{eq:faithful-perm}
\Fspace(\bl) \defn \left\{ \pi \in \Pspace_{n_1} \mid   \text{ for all } i, j \in [n_1] \text{ with } \sigma_{\bl}(i) < \sigma_{\bl}(j), \text{ we have } \pi(i) < \pi(j) \right\}.
\end{align}

By definition, we have the inclusion
$\Mspace
(\lattice_{d, n}; \bl_1, \ldots, \bl_d) \subseteq \Mspace(\lattice_{d, n};
\pi_1, \ldots,
\pi_d)$ for any tuple $(\pi_1, \ldots, \pi_d)$ satisfying $\pi_j \in \Fspace
(\bl_j)$ for all $j \in [d]$. 
We now turn to the proof of the lemma. Denote the error tensor by $\Deltahat \defn
\blop(\theta + \epsilon; \bl_1, \ldots, \bl_d) - \theta$; our key claim is that
\begin{align} \label{eq:expected-error-bounded}
\EE [ \| \Deltahat \|_2^2 ] \lesssim n^{1 - 1/ d} \log^{5/2} n.
\end{align}
Indeed, with this claim in hand, the proof of the lemma follows by applying
Lemma~\ref{lem:vdG-W} from the appendix.

We dedicate the rest of the proof to establishing claim
\eqref{eq:expected-error-bounded}. Our strategy is almost identical to the
proof of Proposition~\ref{cor:lse}; we first control the error of the bounded
least squares estimator in this setting and then obtain claim
\eqref{eq:expected-error-bounded} via a truncation argument. 

Denote the bounded LSE for this setting by
\[
\thetahatblse(r) \defn \argmin_{\thetatil \in \Mspace(\lattice_{d, n}; \bl_1,
\ldots,
\bl_d) \cap \infball(r)} \| \theta + \epsilon - \thetatil \|_2^2,
\]
and let $\Deltahatblse(r) \defn \thetahatblse(r) - \theta$. Rearranging the
basic
inequality similarly to the argument below equation~\eqref{eq:obj-blse} 
(see also~\citet[Chapter 13]{wainwright2019high}) yields the bound
\begin{align*}
\frac{1}{2} \| \Deltahatblse(r) \|_2^2 \leq \sup_{ \substack{\thetatil \in
\Mspace
(\lattice_{d, n};
\bl_1, \ldots, \bl_d) \cap \infball(r) \\ \| \thetatil - \theta \|_2 \leq \|
\Deltahatblse(r) \|_2}
}
\inprod{\epsilon}{\thetatil - \theta} \leq \sup_{ \substack{\thetatil \in
\Mspace(\lattice_{d, n}; \pi_1, \ldots, \pi_d) \cap \infball(r) \\ \| \thetatil - \theta \|_2
\leq \| \Deltahatblse(r) \|_2} }
\inprod{\epsilon}{\thetatil - \theta}, 
\end{align*}
for permutations $(\pi_1, \ldots, \pi_d)$ satisfying $\pi_j \in \Fspace
(\bl_j)$ for all $j \in [d]$.
For convenience, let $\theta_{\pi^{-1}} \defn \theta\{\pi^{-1}_1,
\ldots,
\pi^{-1}_d \}$. Proceeding from the previous bound, we have
\begin{align*}
\sup_{ \substack{\thetatil \in
\Mspace(\lattice_{d, n}; \pi_1, \ldots, \pi_d) \cap \infball(r) \\ \| \thetatil
- \theta \|_2
\leq \| \Deltahatblse(r) \|_2} }
\inprod{\epsilon}{\thetatil - \theta} &= \sup_{ \substack{\thetatil \in
\Mspace(\lattice_{d, n}; \pi_1, \ldots, \pi_d) \cap \infball(r) \\ \| \thetatil
- \theta \|_2
\leq \| \Deltahatblse(r) \|_2} }
\inprod{\epsilon \{ \pi^{-1}_1, \ldots, \pi^{-1}_d\}}{(\thetatil - \theta)\{\pi^{-1}_1,
\ldots,
\pi^{-1}_d \} }  \\
&= \sup_{ \substack{\thetatil \in
\Mspace(\lattice_{d, n})\cap \infball(r) \\ \| \thetatil - \theta_{\pi^{-1}} \|_2
\leq \| \Deltahatblse(r) \|_2} }
\inprod{\epsilon\{ \pi^{-1}_1, \ldots, \pi^{-1}_d \} }{\thetatil - \theta_{\pi^{-1}}}  \\
&\stackrel{d}{=} \sup_{ \substack{\thetatil \in
\Mspace(\lattice_{d, n}) \cap \infball(r) \\ \| \thetatil - \theta_{\pi^{-1}} \|_2
\leq \| \Deltahatblse(r) \|_2} }
\inprod{\epsilon}{\thetatil - \theta_{\pi^{-1}}},
\end{align*}
where the equality in distribution follows from the exchangeability of the components of the noise
$\epsilon$. Recall the notation $\Mspace^{\full}(r)$ from the proof of
Proposition~\ref{prop:full-funlim}. Since $\theta_{\pi^{-1}} \in \Mspace(\lattice_{d, n})
\cap \infball(1)$, we have $\thetatil - \theta_{\pi^{-1}} \in \Mspace^{\full}(r) - \Mspace^{\full}(r)$ provided $r \geq 1$.
Let
\begin{align*}
\xi(t) \defn \sup_{ \substack{\thetatil \in
\Mspace(\lattice_{d, n}) \cap \infball(r) \\ \| \thetatil - \theta_{\pi^{-1}} \|_2
\leq t} }
\inprod{\epsilon}{\thetatil - \theta_{\pi^{-1}}},
\end{align*}
and note from the proof of Propositions~\ref{prop:full-funlim} and~\ref{cor:lse} 
that we have $\EE[\xi(t)] \lesssim r n^{1 - 1/d} \log n \cdot \log (nt)$ for each $r \in (0, n]$. Since the set $\Mspace^{\full}(r) - \Mspace^{\full}(r)$ is star-shaped and non-degenerate,
applying Lemma~\ref{lem:star-shaped} then yields the bound
\begin{align*}
\EE [\| \Deltahatblse(r) \|^2_2] \lesssim r n^{1 - 1/d} \log^2 n \text{ for each
}
r \in [1, n].
\end{align*}
We now employ the truncation argument from the proof of Proposition~\ref{cor:lse} 
to bound $\EE [\| \Deltahat
\|_2^2]$. Lemma~\ref{lem:contraction}
from the appendix guarantees the existence of an event $\Espace$ occurring with
probability at least $1 - n^{-7}$, on which $\| \blop(\theta + \epsilon;
\bl_1, \ldots, \bl_d) \|_{\infty} \leq 4 \sqrt{\log n} + 1 =: \psi_n$. On
this event, we therefore have 
\[
\blop(\theta + \epsilon;
\bl_1, \ldots, \bl_d) = \thetahatblse( \psi_n ) \quad \text{ and } \quad \|
\Deltahat \|_2^2 = \| \Deltahatblse(\psi_n) \|_2^2.
\]
Finally, since $\blop(\theta + \epsilon;
\bl_1, \ldots, \bl_d)$ is obtained via an $\ell_2$-projection onto a convex
set, we may apply inequality~\eqref{contract-convex} to obtain $\| \Deltahat
\|^2_2 \leq \| \epsilon
\|_2^2$ pointwise. Finishing the argument as in the proof of Proposition~\ref{cor:lse}, we obtain claim~\eqref{eq:expected-error-bounded}.
\qed

\subsubsection{Proof of Lemma~\ref{lem:approx-error}}

Suppose wlog that $\pistar_1 = \cdots =
\pistar_d = \id$, so that $\thetastar \in \Mspace^{\kfull, \sset}(\lattice_{d, n})$. 
Recall from equation~\eqref{eq:faithful-perm} the set $\Fspace(\bl)$ consisting of permutations that are faithful to the ordered partition $\bl$.

\sloppy
We will now apply Lemma~\ref{lem:composition} from the appendix.
We refer the reader to equation~\eqref{eq:ind-projections}, in which we define two projection operators $\Pscript$ and $\Aspace$ onto convex sets. These correspond to
projections onto the (permuted) isotonic cone and block-wise constant cone, respectively.
Lemma~\ref{lem:composition} shows that the $L_2$-projection
onto the set
$\Mspace(\lattice_{d, n}; \blhat_1, \ldots, \blhat_d)$ can be written as two
successive projections: in other words 
\begin{align*}
\Bspace(\;\cdot\; ; \blhat_1, \ldots, \blhat_d) =
\Pscript
(\; \Aspace(\;\cdot\;; \blhat_1, \ldots, \blhat_d)\; ; \pihat_1, \ldots,
\pihat_d),
\end{align*}
where $\pihat_1, \ldots, \pihat_d$ is any set of permutations such that
$\pihat_j \in \Fspace(\blhat_j)$.
Thus, from successive applications of the triangle inequality, we obtain
\begin{align*}
&\| \Bspace(\thetastar; \blhat_1, \ldots, \blhat_d) - \thetastar \|_2\\
&\quad \leq \|
\Bspace(\thetastar; \blhat_1, \ldots, \blhat_d) - \Pscript(\thetastar; \pihat_1,
\ldots, \pihat_d) \|_2 + \| \Pscript(\thetastar; \pihat_1, \ldots, \pihat_d) -
\thetastar\{ \pihat_1, \ldots, \pihat_d\} \|_2 \\
&\qquad \qquad \qquad \qquad \qquad \qquad \qquad \qquad \qquad \qquad \qquad \qquad + \| \thetastar\{ \pihat_1,
\ldots, \pihat_d \} - \thetastar \|_2 \\
&\quad \stackrel{\1}{\leq} \| \Aspace(\thetastar; \blhat_1, \ldots, \blhat_d) -
\thetastar \|_2
+ 2 \| \thetastar\{ \pihat_1,
\ldots, \pihat_d \} - \thetastar \|_2,
\end{align*}
where in step $\1$, we have twice used the fact that the projection onto a
convex set is non-expansive~\eqref{contract-convex}. We now bound these two terms separately,
starting
with the second term, but first, for each $j \in [d]$, define the random
variables
\begin{align} \label{eq:TU-RV}
T_j \defn  2 \| \tauhat_j - \taustar_j \|_\infty + 8 
\sqrt{\log n} \cdot n^{\frac{1}{2}(1 - 1/d)} \quad \text{ and }  \quad U \defn 2
\| Y - \thetastar \|_\infty + 8 \sqrt{\log n}.
\end{align}
Also recall the events $\Espace_1$ and $\Espace_2$ defined in equation
\eqref{eq:hp-events}.

\paragraph*{Bound on permutation error}

Our proof proceeds by bounding this term in two different ways; let us now
sketch it. 
We will establish the claims
\begin{align} \label{eq:clm-error-perm}
\| \thetastar\{ \pihat_1, \ldots, \pihat_d \} - \thetastar \|_2^2 \leq 2d \cdot 
(n_1 - k^*) \cdot
\begin{cases}
 U \cdot \sum_{j = 1}^d T_j \quad &\text{on the event }
\Espace_1 \cap
\Espace_2, \\
\sum_{j = 1}^d T^2_j \quad &\text{pointwise}.
\end{cases}
\end{align}
Let us take equation~\eqref{eq:clm-error-perm} as given for the moment and
establish a bound on the permutation error. First, note that on
$\Espace_1 \cap \Espace_2$, we have $\max_{j \in [d]} \; T_j \lesssim  
\sqrt{\log n} \cdot n^{\frac{1}{2}(1 - 1/d)}$ and $U \lesssim 
\sqrt{\log n}$. Also note that $\Pr\{ (\Espace_1 \cap \Espace_2)^c
\}  \leq 2n^{-7}$; this establishes the high probability bound
\begin{align*}
\Pr \left\{ \| \thetastar\{ \pihat_1, \ldots, \pihat_d \} - \thetastar
\|_2^2 \leq C d^2 (n_1 - k^*) \cdot n^{\frac{1}{2}(1 - 1/d)} \cdot \log n
\right\} \geq 1 - 2n^{-7}.
\end{align*}
On the other hand, we have
\begin{align*}
\EE \biggl[ \big(\sum_{j = 1}^d T_j^2 \big)^2 \biggr] \lesssim d \cdot \sum_{j = 1}^d \bigg\{ \EE\left[ \| \tauhat_j - \taustar_j
\|_\infty^4 \right] + (\log n)^2 \cdot n^{2(1 - 1/d)} \bigg\} \lesssim \left( d \log n \cdot n^{1 - 1/d} \right)^2,
\end{align*}
where the second inequality follows since $\tauhat_j \sim \NORMAL(\taustar_j,
n^{1 - 1/d} \cdot I)$ and so $\| \tauhat_j - \taustar_j \|_\infty$ is the maximum
absolute deviation of $n_1$ i.i.d. Gaussian random variables with mean zero and
variance $n^{1 - 1/d}$. 

Putting together the bounds in expectation and high probability and applying Lemma~\ref{lem:RV} from the appendix,
we obtain
\begin{align*}
\frac{\EE \left[ \| \thetastar\{ \pihat_1, \ldots, \pihat_d \} - \thetastar
\|_2^2
\right]}{d^2 (n_1 - k^*)} &\lesssim \left( \log n \right) \cdot n^{\frac{1}{2}(1 - 1/d)} + d (\log n) \cdot n^{1 -
1/d} \cdot n^{-7/2} \\
&\lesssim (\log n) \cdot n^{\frac{1}{2}(1 - 1/d)},
\end{align*}
which is of the same order as the bound claimed by Lemma
\ref{lem:approx-error}. It remains to establish
claim~\eqref{eq:clm-error-perm}.

In order to do so, we employ an inductive argument by
handling the approximation error along one dimension at a time. As a first step,
we have
\begin{align}
&\| \thetastar\{ \pihat_1, \ldots, \pihat_d \} - \thetastar \|_2  \notag \\
&\qquad \qquad \qquad \leq \| \thetastar
\{
\pihat_1, \ldots, \pihat_d \} - \thetastar\{ \pistar_1, \pihat_2, \ldots,
\pihat_d \} \|_2 + \| \thetastar\{ \pistar_1, \pihat_2, \ldots,
\pihat_d \} - \thetastar \|_2 \notag \\
&\qquad \qquad \qquad  \stackrel{\2}{=} \| \thetastar\{
\pihat_1, \id, \ldots, \id \} - \thetastar\{ \pistar_1, \id, \ldots,
\id \} \|_2 + \| \thetastar\{ \pistar_1, \pihat_2, \ldots,
\pihat_d \} - \thetastar \|_2, \label{eq:two-terms}
\end{align}
where step $\2$ follows by the unitary invariance of the $\ell_2$-norm. If we
write $P_j$ for the squared error along the $j$-th dimension with $P_1
\defn \| \thetastar\{
\pihat_1, \id, \ldots, \id \} - \thetastar\{ \pistar_1, \id, \ldots,
\id \} \|^2_2$, then bounding the error along the remaining dimensions
using an inductive
argument yields the bound
\begin{align*}
\| \thetastar\{ \pihat_1, \ldots, \pihat_d \} - \thetastar \|^2_2 \leq  \Biggl(
\sum_{j = 1}^d \sqrt{P_j} \Biggr)^2 \leq d \cdot \sum_{j =1}^d P_j.
\end{align*}
Thus, our strategy to establish claim~\eqref{eq:clm-error-perm} will be to
establish the sufficient claim
\begin{align} \label{eq:suff-clm-perm}
P_j \leq 2(n_1 - k^*) \cdot
\begin{cases}
U \cdot T_j \quad &\text{on the event } \Espace_1 \cap \Espace_2 \\
T^2_j \quad &\text{pointwise.}
\end{cases}
\end{align}
We establish this claim for $j = 1$; the general proof is identical.
Letting
 $\nbar \defn n_1^{d-1}$ and recalling our assumption $\pistar_j = \id$ for all
 $j \in [d]$, we have
\begin{align*}
P_1 &= \| \thetastar\{
\pihat_1, \id, \ldots, \id \} - \thetastar\{ \id, \id, \ldots,
\id \} \|_2^2 \\
&= \sum_{i_1 = 1}^{n_1} \; \; \sum_{(i_2, \ldots, i_{d})
\in
\lattice_{d - 1, \nbar} }\left(\thetastar( \pihat_1(i_1) , i_2, \ldots, i_
{d} ) - \thetastar( i_1 , i_2, \ldots,
i_{d} ) \right)^2.
\end{align*}
We now split the proof into the two cases of equation~\eqref{eq:suff-clm-perm}.

\smallskip

\noindent \underline{Case 1:} In this case, we work on the event
$\Espace_1 \cap \Espace_2$. By Claim~\ref{clm:graph}, we know that
on this event, we have $G'_j = G_j$. Consequently, for any $\pihat_1 \in \Fspace
(\blhat_1)$, we have $\pihat_1(k) < \pihat_1(\ell)$ if
\begin{align} \label{eq:conditions-threshold-perm}
\tauhat_{1}(\ell) - \tauhat_1(k) &> 8 \sqrt{\log n} \cdot n^{\frac{1}{2}(1 -
1/d)}
\qquad
\text{ or} \\
\max_{i_2, \ldots, i_d}\; \left[ Y (\ell, i_2, \ldots, i_d) - Y (k,
i_2, \ldots, i_d) \right] &> 8 \sqrt{\log n}. \notag
\end{align}
As a consequence of the second condition, for any fixed tuple $(i_2, \ldots,
i_d)$, we have $\pihat_1(k) < \pihat_1(\ell)$ whenever $Y(\ell,
i_2, \ldots,i_{d} ) - Y( k, i_2, \ldots, i_{d} ) > 8\sqrt{\log n}$.
Applying Lemma~\ref{lem:per-num} from the appendix with the
substitution $a = \thetastar(\cdot, i_2, \ldots, i_d)$,  $b = Y(\cdot, i_2, \ldots, i_d)$ and
$\tau = 8 \sqrt{\log n}$ yields the bound
\begin{align*}
\max_{i_1, \ldots, i_d \in [n_1]} |\thetastar( \pihat_1(i_1) , i_2, \ldots, i_{d} ) - \thetastar(i_1,
i_2, \ldots,i_{d} )| \leq 2 \| Y - \thetastar \|_{\infty} + 8 \sqrt{\log n} = U.
\end{align*} 
We may now apply H{\"o}lder's inequality to obtain for each $i_1 \in [n_1]$ that
\begin{align}
&\sum_{(i_2, \ldots, i_{d})
\in
\lattice_{d - 1, \nbar} }\left(\thetastar( \pihat_1(i_1) , i_2, \ldots, i_
{d}) - \thetastar( i_1 , i_2, \ldots,
i_{d} ) \right)^2 \notag \\
&\qquad \qquad \qquad \qquad\leq U \cdot \sum_{(i_2, \ldots, i_{d})
\in
\lattice_{d - 1, \nbar} }\left| \thetastar( \pihat_1(i_1) , i_2, \ldots, i_
{d} ) - \thetastar ( i_1 , i_2, \ldots,
i_{d} ) \right| \notag \\
&\qquad \qquad \qquad \qquad\stackrel{\3}{=} U \cdot \Biggl| \sum_{(i_2,
\ldots, i_
{d})
\in
\lattice_{d - 1, \nbar} } \bigg\{ \thetastar( \pihat_1(i_1), i_2, \ldots, i_
{d} ) - \thetastar ( i_1, i_2, \ldots,
i_{d} ) \bigg\} \Biggr| \notag \\
&\qquad \qquad \qquad \qquad = U \cdot | \taustar_1 (\pihat_1
(i_1)) - \taustar_1
(i_1) |, \label{eq:split-point}
\end{align}
where step $\3$ follows since the set of scalars 
\[ 
\Big\{ \thetastar( \pihat_1(i_1),
i_2, \ldots, i_
{d} ) - \thetastar(i_1, i_2, \ldots,
i_{d} ) \Big\}_{(i_2, \ldots, i_{d})
\in
\lattice_{d - 1, \nbar}}
\]
all have the same sign by our monotonicity assumption $\thetastar \in \Mspace
(\lattice_{d, n})$.  

Let $\mathcal{I}$ denote the set containing all
indices $i_1$ for which $\taustar_1 (\pihat_1(i_1)) - \taustar_1(i_1) $ is
non-zero. Since there is an indifference set of size $k^1_{\max}$ along the
first dimension, a non-zero value can only occur
if either $i_1$ or $\pihat_1(i_1)$ belong to the $n_1 - k^1_{\max}$ indices that
are not in the largest indifference set. Consequently, we obtain
$|\mathcal{I}| \leq 2(n_1 - k^1_{\max}) \leq 2(n_1 - k^*)$.
Therefore, by the first part of equation~\eqref{eq:conditions-threshold-perm}, we have
\begin{align}
\sum_{i_1 = 1}^{n_1} | \taustar_1 (\pihat_1(i_1)) - \taustar_1 
(i_1) | &= \sum_{i_1 \in \Ispace} | \taustar_1 (\pihat_1(i_1))
- \taustar_1 (i_1) | \notag \\
&\stackrel{\4}{\leq} 2 (n_1 - k^*) \cdot \left( 2 \| \tauhat_1 - \taustar_1 \|_
{\infty} + 8 \sqrt{\log
n} \cdot n^{\frac{1}{2}(1 - 1/d)} \right) \notag \\
&= 2(n_1 - k^*) \cdot T_1, \label{eq:final-bd}
\end{align}
where step $\4$ follows by applying Lemma~\ref{lem:per-num} once again, but now
to the scores.
This completes the proof of the first case in equation~\eqref{eq:suff-clm-perm}.

\smallskip

\noindent \underline{Case 2:} In this case, our goal is to prove a pointwise
bound. In order to do so, we appeal to the properties of our
algorithm: by construction, the edges of the graph $G_1$ are always consistent
with the conditions imposed by the pairwise statistics $\Deltahat^{{\sf
sum}}_1$, so that for any $\pihat_1 \in \Fspace(\blhat_1)$, we have $\pihat_1(k)
< \pihat_1(\ell)$ if
$
\tauhat_1(\ell) - \tauhat_1(k) > 8 \sqrt{\log n} \cdot n^{\frac{1}{2}(1 - 1/d)}.
$
We now repeat the reasoning from before to obtain the (crude) sequence of bounds
\begin{align} 
&\sum_{(i_2, \ldots, i_{d})
\in
\lattice_{d - 1, \nbar} }\left(\thetastar( \pihat_1(i_1) , i_2, \ldots, i_
{d}) - \thetastar( i_1 , i_2, \ldots,
i_{d} ) \right)^2 \label{eq:first-bd-case-two} \\
&\qquad \qquad \qquad \qquad\leq \Biggl( \sum_{(i_2, \ldots, i_{d})
\in
\lattice_{d - 1, \nbar} }\left| \thetastar( \pihat_1(i_1) , i_2, \ldots, i_
{d} ) - \thetastar ( i_1 , i_2, \ldots,
i_{d} ) \right| \Biggr)^2 \notag \\
&\qquad \qquad \qquad \qquad = \Biggl| \sum_{(i_2,
\ldots, i_
{d})
\in
\lattice_{d - 1, \nbar} } \thetastar( \pihat_1(i_1), i_2, \ldots, i_
{d} ) - \thetastar ( i_1, i_2, \ldots,
i_{d} ) \Biggr|^2 \notag \\
&\qquad \qquad \qquad \qquad = | \taustar_1 (\pihat_1
(i_1)) - \taustar_1(i_1) |^2. \notag
\end{align}
Proceeding exactly as before then yields the bound
\begin{align} \label{eq:final-bd-case-two}
\sum_{i_1 = 1}^{n_1} | \taustar_1 (\pihat_1
(i_1)) - \taustar_1(i_1) |^2 \leq 2(n_1 - k^*) \cdot T_1^2,
\end{align}
which establishes the second case of equation~\eqref{eq:suff-clm-perm}.
\qed

\paragraph*{Bound on averaging error}
In order to prove a bound on the averaging error, we first set up some
notation and terminology to write the averaging error in a manner that is
very similar to the permutation error bounded above. Then the proof follows
from the arguments above.

First, note that the
averaging operator can be equivalently implemented by sequentially averaging the
entries along one dimension at a time. Let us make this precise with some
ancillary definitions. 
Let $\Aspacehat_j(\theta)$ denote the average of $\theta
\in \real_{d, n}$ along dimension $j$ according to
the partition specified by $\blhat_j$, i.e., for each $i_1, \ldots, i_d \in [n_1]$,
\begin{align*}
\Aspacehat_j(\theta) (i_1, \ldots, i_d) \defn \frac{1}{|\blhat_j(i_j)|} \sum_{\ell
\in
\blhat_j(i_j)} \theta (i_1, \ldots, i_{j - 1}, \ell, i_{j+1},
\ldots, i_d).
\end{align*}
As a straightforward consequence of the linearity
of the
averaging operation, we have
\[
\Aspace(\theta; \blhat_1, \ldots, \blhat_d ) = \Aspacehat_1 \circ \cdots \circ
\Aspacehat_d (\theta) \qquad \text{ for each } \; \theta \in \real_{d, n}.
\]
Consequently, we may peel off the first
dimension from the error of interest to write
\begin{align}
&\|\Aspace(\thetastar;
\blhat_1, \ldots, \blhat_d ) - \thetastar \|_2 \notag \\
&\qquad \qquad \qquad \leq \|\Aspacehat_1 \circ \cdots
\circ
\Aspacehat_d (\thetastar) - \Aspacehat_2 \circ \cdots \circ
\Aspacehat_d (\thetastar) \|_2 + \| \Aspacehat_2 \circ \cdots \circ \Aspacehat_d
(\thetastar) -\thetastar \|_2 \notag \\
&\qquad \qquad \qquad= \|\Aspacehat_1 (\thetahat_{2:d}) - \thetahat_{2:d} \|_2 + \| \thetahat_{2:d}
-\thetastar \|_2, \label{eq:two-terms-ave}
\end{align}
where we have let $\thetahat_{2:d} \defn \Aspacehat_2 \circ \cdots \circ
\Aspacehat_d (\thetastar)$. Note the similarity between equations
\eqref{eq:two-terms-ave} and~\eqref{eq:two-terms}. 
Indeed, if we now
write $P'_j$ for the squared error peeled along the $j$-th dimension with $P'_1
=  \|\Aspacehat_1 (\thetahat_{2:d}) - \thetahat_{2:d} \|^2_2$, then peeling the
error along the remaining dimensions
using an inductive argument, we obtain (exactly as before)
\begin{align*}
\|\Aspace(\thetastar;
\blhat_1, \ldots, \blhat_d ) - \thetastar \|_2^2 \leq  \Biggl(
\sum_{j = 1}^d \sqrt{P'_j} \Biggr)^2 \leq d \cdot \sum_{j =1}^d P'_j.
\end{align*}
We now claim that with the random variables $U$ and $T_j$ defined exactly as
before, we have the
(identical) bound
\begin{align} \label{eq:suff-clm-ave}
P'_j \leq (n_1 - k^*) \cdot
\begin{cases}
U \cdot T_j \quad &\text{on the event } \Espace_1 \cap \Espace_2 \\
T^2_j \quad &\text{pointwise,}
\end{cases}
\end{align}
from which the proof of Lemma~\ref{lem:approx-error} for the averaging term
follows identically.

Let us now establish the bound~\eqref{eq:suff-clm-ave} for $j = 1$, for which we
require some ancillary definitions. For an ordered partition $\bl = (S_1, \ldots, S_L)$ and
index $i \in [n_1]$, recall the notation $\sigma_{\bl}(i)$ as the index~$\ell$
of the set $S_\ell \ni i$. Let $\bl(i) = S_{\sigma_{\bl}(i)}$ denote the block
containing index $i$. In a slight abuse of notation, let $\Pspace_{V}$ denote the set of all
permutations on a set $V \subseteq [n_1]$, and let $\Jspace(\bl)$ denote the set
of all permutations $\pi \in \Pspace_{n_1}$ such that $\pi(i) \in \bl(i)$ for
all~$i \in [n_1]$. Note that any permutation in the set
$\Jspace(\bl)$ is given by compositions of individual permutations in
$\Pspace_V$ for $V \in \bl$.

With this notation, we have for each $\theta \in \real_{d, n}$, the sequence of bounds
\begin{align}
\| \Aspacehat_1(\theta) - \theta \|_2^2 &= \sum_{i_2, \ldots, i_d} \sum_{i_1 =
1}^{n_1} \left(\frac{1}{|\blhat_1
(i_1)|}
\sum_{\ell \in
\blhat_1(i_1)} \theta (\ell, i_2, \ldots, i_d) - \theta(i_1, \ldots, i_d)
\right)^2 \notag \\
&= \sum_{i_2, \ldots, i_d} \sum_{V \in \blhat_1} \sum_{i_1 \in V} \left(\frac{1}
{|V|}
\sum_{\ell \in V} \theta (\ell, i_2, \ldots, i_d) - \theta(i_1, \ldots,
i_d) \right)^2 \notag \\
&\stackrel{\1}{\leq} \sum_{i_2, \ldots, i_d} \sum_{V \in \blhat_1} \max_{\pi'
\in
\Pspace_V } \sum_{i_1 \in V}
\left( \theta (\pi'(i_1), i_2, \ldots, i_d) - \theta(i_1, \ldots, i_d)
\right)^2 \notag \\
&= \sum_{V \in \blhat_1} \max_{\pi'
\in
\Pspace_V } \sum_{i_1 \in V} \sum_{i_2, \ldots, i_d}
\left( \theta (\pi'(i_1), i_2, \ldots, i_d) - \theta(i_1, \ldots, i_d)
\right)^2 \notag \\
&\stackrel{\2}{=} \max_{\pi \in \Jspace(\blhat_1) } \sum_{i_1, \ldots, i_d} ( \theta
(\pi(i_1), i_2, \ldots, i_d) - \theta(i_1, i_2, \ldots, i_d))^2. 
\label{eq:Jensen}
\end{align}
Here, step $\1$ follows from Lemma~\ref{lem:rearrange} in the appendix, and step $\2$
follows since the maximization over permutations $\pi \in \Jspace(\blhat_1)$ can
be carried out by individual maximizations over permutations on each block of the partition.

It is also useful to note that $\thetahat_{2:d}$ enjoys some additional
structure. In particular, the
estimate $\thetahat_{2:d}$ satisfies some properties that are straightforward
to verify:
\begin{enumerate}
\item For any $\thetastar \in \real_{d, n}$, the slices
along the first dimension have the same sum as $\thetastar$: i.e., for each
index $\ell
\in [n_1]$, we have
\begin{subequations} \label{eq:prop-ave}
\begin{align} \label{eq:prop-ave-one}
\sum_{j = 1}^d \sum_{i_j = 1}^{n_j} \; \thetahat_{2:d}(i_1, \ldots, i_d)
\cdot \ind{i_1 = \ell} = \sum_
{j= 1}^d \sum_{i_j = 1}^{n_j} \; \thetastar (i_1, \ldots, i_d) \cdot 
\ind{i_1 = \ell} = \taustar_1(\ell),
\end{align}
where the final equality holds by definition~\eqref{eq:pop-scores}.
\item If $\thetastar \in \Mspace(\lattice_{d, n})$, then its monotonicity
property is preserved along the first dimension: i.e., for each pair of indices
$1 \leq k \leq \ell \leq n_1$, we
have
\begin{align} \label{eq:prop-ave-two}
\thetahat_{2:d}(k, i_2, \ldots, i_d) \leq  \thetahat_{2:d}(\ell, i_2, \ldots,
i_d) \quad \text{ for all } i_2, \ldots, i_d \in [n_1].
\end{align}
\end{subequations}
\end{enumerate}
With these properties in hand, we are now ready to establish the proof of
claim~\eqref{eq:suff-clm-ave}. First, use equation~\eqref{eq:Jensen} and let
 $\nbar \defn n_1^{d-1}$ to obtain the pointwise bound
\begin{align} \label{eq:det-bd-ave}
P'_1  \leq \max_{\pi_1
\in \Jspace(\blhat_1)} \sum_{i_1 = 1}^{n_1} \; \; \sum_{(i_2, \ldots, i_{d})
\in
\lattice_{d - 1, \nbar} }\left(\thetahat_{2:d}( \pi_1(i_1) , i_2, \ldots, i_
{d} ) - \thetahat_{2:d}( i_1, i_2, \ldots, i_{d} ) \right)^2.
\end{align}
We now establish the two cases of equation~\eqref{eq:suff-clm-ave} separately.

\smallskip

\noindent \underline{Case 1:} In this case, we work on the event
$\Espace_1 \cap \Espace_2$, in which case the estimated blocks obey the
conditions~\eqref{eq:conditions-threshold-perm}; in particular, two indices
$k, \ell$ are placed in the same block of $\blhat_1$ if and only if
\begin{align} \label{eq:conditions-threshold-ave}
|\tauhat_1 (k) - \tauhat_1 (\ell)| &\leq 8 \sqrt{\log n} \cdot n^{\frac{1}{2}(1
- 1/d)} \quad \text{ and } \\ 
\max_{i_2, \ldots, i_d} \; |Y(k, i_2, \ldots,
i_d) - Y(\ell, i_2, \ldots, i_d)| &\leq 8 \sqrt{\log n}. \notag
\end{align}
Since the averaging operation
is
$\ell_\infty$-contractive, we
have, for each $\pi_1 \in \Jspace(\blhat_1)$, the sequence of bounds
\begin{align*}
&\max_{i_1, \ldots, i_d \in [n_1]} |\thetahat_{2:d}( \pi_1(i_1) , i_2, \ldots, i_
{d} ) - \thetahat_{2:d}( i_1, i_2, \ldots, i_{d})| \\
&\quad \leq \max_{i_1, \ldots, i_d \in [n_1]} \bigg\{ |\thetastar(\pi_1
(i_1) , i_2, \ldots, i_{d} ) - \thetastar(i_1, i_2, \ldots,
i_{d} )| \bigg\} \\
&\quad \leq \max_{i_1, \ldots, i_d \in [n_1]} \bigg\{ | \thetastar(\pi_1(i_1) , i_2, \ldots, i_{d} ) - Y(\pi_1(i_1)
, i_2,
\ldots, i_{d} )| \\
&\qquad \qquad \qquad \qquad + | Y(i_1, i_2, \ldots, i_{d} ) - \thetastar(i_1, i_2, \ldots,
i_{d} ) |  \\
&\qquad \qquad \qquad \qquad + | Y(\pi_1(i_1) , i_2,
\ldots, i_{d} ) - Y(i_1, i_2, \ldots, i_{d} )| \bigg\} \\
&\quad \stackrel{\2}{\leq}  2\| Y - \thetastar \|_{\infty} + 8\sqrt{\log n} = U,
\end{align*}
where step $\2$ follows from the second condition~\eqref{eq:conditions-threshold-ave}.

Thus, we have
\begin{align}
\frac{P'_1}{U} &\leq \max_{\pi_1 \in \Jspace
(\blhat_1)} \;
\sum_{i_1 = 1}^{n_1} \; \sum_{(i_2, \ldots, i_{d})
\in \lattice_{d - 1, \nbar} }\left| \thetahat_{2:d}( \pi_1(i_1) , i_2, \ldots,
i_
{d} ) - \thetahat_{2:d} ( i_1, i_2, \ldots, i_{d} ) \right| \notag \\
&\stackrel{\3}{=} \max_{\pi_1 \in \Jspace(\blhat_1)}\;
\sum_{i_1 = 1}^{n_1} \; \left| \sum_{(i_2,
\ldots, i_{d}) \in \lattice_{d - 1, \nbar} } \bigg\{ \thetahat_{2:d}( \pi_1(i_1) , i_2,
\ldots, i_{d} ) - \thetahat_{2:d} ( i_1, i_2, \ldots, i_{d} ) \bigg\} \right| \notag \\
&\stackrel{\4}{=} \max_{\pi_1 \in \Jspace(\blhat_1)} \; \sum_{i_1 = 1}^
{n_1} \; 
| \taustar_1 (\pi_1(i_1)) - \taustar_1(i_1) |,
\label{eq:split-point-ave}
\end{align}
where step $\3$ follows by the monotonicity property~\eqref{eq:prop-ave-two} and
step $\4$ from property~\eqref{eq:prop-ave-one}.

Now for each $\pi_1 \in \Jspace(\blhat_1)$, we have
\begin{align}
| \taustar_1 (\pi_1(i_1)) - \taustar_1(i_1) | &\leq | \taustar_1 (\pi_1
(i_1)) - \tauhat_1(\pi(i_1)) | + | \tauhat_1 (i_1) - \taustar_1(i_1) | + |
\tauhat_1 (\pi_1(i_1)) - \tauhat_1(i_1) | \notag \\
&\stackrel{\5}{\leq} 2 \| \tauhat_1 - \taustar_1 \|_\infty + 8 \sqrt{\log n}
\cdot n^{
\frac{1}{2}
(1 - 1/d)} = T_1, \label{eq:richard}
\end{align}
where step $\5$ is guaranteed by the first condition
\eqref{eq:conditions-threshold-ave}.
Since there are at most $2(n_1 - k^*)$ indices in the
sum~\eqref{eq:split-point-ave} that are non-zero, we have
\begin{align*}
P'_1 \leq 2 U \cdot (n_1 - k^*) \cdot T_1,
\end{align*}
and this completes the proof of the first case of equation~\eqref{eq:suff-clm-ave}.

\smallskip

\noindent \underline{Case 2:} In this case, our goal is to establish a pointwise
bound. Once again, by construction, the estimated ordered
partitions are always consistent with the pairwise statistics $\Deltahat^{{\sf
sum}}_1$, so that two indices $k, \ell$ are placed within the same block of
$\blhat_1$ if and only if
\begin{align} \label{eq:conditions-ave-pointwise}
|\tauhat_1 (k) - \tauhat_1 (\ell)| \leq 8 \sqrt{\log n} \cdot n^{\frac{1}{2}(1
- 1/d)}.
\end{align}
Consequently, proceeding from equation~\eqref{eq:det-bd-ave} and using the same
properties as before, we have
\begin{align*}
P'_1 &\leq \max_{\pi_1
\in \Jspace(\blhat_1)} \sum_{i_1 = 1}^{n_1} \; \; \left( \sum_{(i_2, \ldots, i_
{d})
\in
\lattice_{d - 1, \nbar} }| \thetahat_{2:d}( \pi_1(i_1) , i_2, \ldots, i_
{d} ) - \thetahat_{2:d}( i_1, i_2, \ldots, i_{d} ) | \right)^2 \\
&= \max_{\pi_1 \in \Jspace(\blhat_1)}\;
\sum_{i_1 = 1}^{n_1} \; \left| \sum_{(i_2,
\ldots, i_{d}) \in \lattice_{d - 1, \nbar} } \bigg\{ \thetahat_{2:d}( \pi_1(i_1) , i_2,
\ldots, i_{d} ) - \thetahat_{2:d} ( i_1, i_2, \ldots, i_{d} ) \bigg\} \right|^2
\\
&= \max_{\pi_1 \in \Jspace(\blhat_1)} \; \sum_{i_1 = 1}^
{n_1} \; 
| \taustar_1 (\pi_1(i_1)) - \taustar_1(i_1) |^2.
\end{align*}
By equation~\eqref{eq:richard} and the fact that there are at most $2(n_1 - k^*)$ indices in the
sum~\eqref{eq:split-point-ave} that are non-zero, we obtain
\begin{align*}
P'_1 \leq 2 (n_1 - k^*) \cdot T^2_1,
\end{align*}
and this completes the proof of the second case of equation
\eqref{eq:suff-clm-ave}.
\qed

\subsection{Proof of Proposition~\ref{prop:borda-worst-case}} \label{sec:pf-thm1}


At the heart of the proposition lies the following lemma, which bounds
the $\ell_2$ error as a sum of approximation and estimation errors.

\begin{lemma} \label{lem:oracle}
There is a universal positive constant $C$ such that for all $\thetastar
\in \Mspace_{\perm}(\lattice_
{d,
n}) \cap \infball(1)$, we have
\begin{align*}
\Rspace_n ( \thetahatplug, \thetastar ) \leq C \biggl( n^{-1/d} \log^{5/2} n +
\frac{d}{n} \sum_{j = 1}^d \EE \left[ \| \tauhat_j - \taustar_j \|_1 \right]
\biggr).
\end{align*}
\end{lemma}
Taking this lemma as given for the moment, the proof of the proposition is
straightforward. The random variable $\tauhat_j(k) - \taustar_j(k)$ is the sum of
$n^{1 - 1/d}$ independent standard Gaussians, so that
\begin{align*}
\EE[ | \tauhat_j(k) - \taustar_j(k) | ] = \sqrt{\frac{2}{\pi}} \cdot n^{\frac{1}{2}
(1 - 1/d)} \text{ for each } k \in [n_1], \; j \in [d].
\end{align*}
Summing over both $k \in [n_1]$ and $j \in [d]$ and normalizing, we have
\begin{align*}
\frac{d}{n} \sum_{j = 1}^d \EE \left[ \| \tauhat_j - \taustar_j \|_1 \right]
\leq
C d^2 n^{-\frac{1}{2}(1 - 1/d)},
\end{align*}
as required.
\qed

\noindent It remains to prove Lemma~\ref{lem:oracle}.

\subsubsection{Proof of Lemma~\ref{lem:oracle}}

In order to lighten notation in this section, we use the convenient
shorthand $\thetahat \equiv \thetahatplug$ and
$\pihat_j \equiv \pihatcount_j$ for each $j \in [d]$.
Assume without loss of generality that $\pistar_1 = \cdots = \pistar_d = \id$,
so that \mbox{$\thetastar \in \Mspace (\lattice_{d, n})$}.
Let $\thetatil$ denote
the projection of the tensor
$\thetastar\{\pihat_1, \ldots, \pihat_d \} + \epsilon$ onto the set $\Mspace (\lattice_{d,
n}; \pihat_1, \ldots, \pihat_d)$. With this setup, we have
\begin{align}
\| \thetahat - \thetastar \|_2 &\leq \| \thetahat - \thetatil \|_2 + \| \thetatil
- \thetastar\{ \pihat_1, \ldots, \pihat_d\} \|_2 + \| \thetastar\{ \pihat_1,
\ldots, \pihat_d \} - \thetastar \|_2 \notag \\
&\stackrel{\1}{\leq} \| \thetastar + \epsilon - ( \thetastar\{ \pihat_1, \ldots,
\pihat_d \} +
\epsilon) \|_2 + \| \thetatil
- \thetastar\{ \pihat_1, \ldots, \pihat_d\} \|_2 \notag \\
& \qquad \qquad \qquad \qquad \qquad \qquad \qquad \qquad \qquad + \| \thetastar\{ \pihat_1,
\ldots, \pihat_d \} - \thetastar \|_2 \notag \\
&= \underbrace{\| \thetatil - \thetastar\{ \pihat_1, \ldots, \pihat_d\} \|_2}_
{\text{estimation error}} + \underbrace{ 2 \cdot \|
\thetastar\{ \pihat_1, \ldots, \pihat_d \} - \thetastar \|_2}_{
\text{approximation error}}. \label{eq:borda-decomp}
\end{align}
Here, step $\1$ follows since $\thetatil$ is
the projection of the tensor
$\thetastar\{\pihat_1, \ldots, \pihat_d \} + \epsilon$ onto the convex set $\Mspace (\lattice_{d,
n}; \pihat_1, \ldots, \pihat_d)$, and an $\ell_2$-projection onto a convex set is always
non-expansive~\eqref{contract-convex}. We
now bound the estimation and approximation error terms separately.

\paragraph*{Bounding the estimation error} The key difficulty here is that the
estimated permutations $\pihat_1, \ldots, \pihat_d$ \emph{depend} on the noise
tensor $\epsilon$. Similarly to before, we handle this dependence by
establishing a
uniform result that holds simultaneously over all choices of permutations. In
particular, letting $\thetahat_{\pi_1, \ldots, \pi_d}$ denote the $\ell_2$
projection of the tensor $\thetastar \{\pi_1, \ldots, \pi_d\} + \epsilon$ onto
the closed, convex
set $\Mspace (\lattice_{d,
n}; \pi_1, \ldots, \pi_d)$, we claim that for each $\thetastar \in
\Mspace(\lattice_{d, n}) \cap \infball(1)$, we have
\begin{align} \label{eq:unif-perm}
\EE \left[ \max_{\pi_1, \ldots, \pi_d} \|
\thetahat_{\pi_1, \ldots, \pi_d} - \thetastar \{\pi_1, \ldots, \pi_d\} \|_2^2
\right] \leq
C n^{1 - 1/d} \log^{5/2} n.
\end{align}
Since $\thetatil = \thetahat_{\pihat_1, \ldots, \pihat_d}$, equation
\eqref{eq:unif-perm} provides a bound on the expected estimation error that is of the
claimed order.

Let us now prove claim~\eqref{eq:unif-perm}. For each fixed
tuple of
permutations $(\pi_1, \ldots, \pi_d)$, combining 
Proposition~\ref{cor:lse}(b) with Lemma~\ref{lem:vdG-W}
yields the tail bound
\begin{align*}
\Pr \left\{ \| \thetahat_{\pi_1, \ldots, \pi_d} - \thetastar \{\pi_1, \ldots,
\pi_d\} \|_2^2
\geq C \cdot n^{1 - 1/d} \log^{5/2} n + 4u
\right\} \leq \exp \left\{ - u\right\} \text{ for all } u \geq 0.
\end{align*}
Taking a
union bound over all 
$
\prod_{j = 1}^d n_j! \leq \exp(d n_1 \log n_1) = \exp
(n_1 \log n)
$
permutations and setting $u = C n_1 \log n + u'$ for a sufficiently large
constant $C$, we obtain that for every $u' \geq 0$,
\begin{align} 
&\Pr \left\{ \max_{\pi_1, \ldots, \pi_d} \|
\thetahat_{\pi_1, \ldots, \pi_d} - \thetastar \{\pi_1, \ldots, \pi_d\} \|_2^2
\geq C(n^{1 - 1/d} \log^{5/2} n + n_1 \log n) + u' \right\} \label{eq:perm-tail} \\
&\qquad \qquad \qquad \qquad \qquad \qquad \qquad \qquad \qquad \qquad \leq e^{-cu'}. \notag 
\end{align}
Finally, note that for each $d \geq 2$, we have $n_1 \leq n^{1- 1/d}$
and integrate the tail bound~\eqref{eq:perm-tail} to complete the proof of claim
\eqref{eq:unif-perm}. 

\paragraph*{Bounding the approximation error} Our bound on the approximation
error proceeds very similarly to the steps~\eqref{eq:two-terms}--\eqref{eq:split-point}, 
so we sketch the key differences.
First, we have the decomposition
\begin{align*}
&\|
\thetastar\{ \pihat_1, \ldots, \pihat_d \} - \thetastar \|_2  \\
&\qquad \qquad \qquad \leq \| \thetastar\{
\pihat_1, \id, \ldots, \id \} - \thetastar\{ \pistar_1, \id, \ldots,
\id \} \|_2 + \| \thetastar\{ \pistar_1, \pihat_2, \ldots,
\pihat_d \} - \thetastar \|_2.
\end{align*}
But since $\thetastar \in \infball(1)$, now each scalar $\thetastar( \pihat_1
(i_1) , i_2, \ldots, i_
{d} ) - \thetastar(\pistar_1(i_1) , i_2, \ldots,
i_{d} )$ is bounded in the range $[-2, 2]$.
Letting
 $\nbar \defn n_1^{d-1}$, H{\"o}lder's inequality yields
\begin{align*}
&\sum_{(i_2, \ldots, i_{d})
\in
\lattice_{d - 1, \nbar} }\left(\thetastar( \pihat_1(i_1) , i_2, \ldots, i_
{d}) - \thetastar( i_1 , i_2, \ldots,
i_{d} ) \right)^2 \notag \\
&\qquad \qquad \qquad \qquad\leq 2 \cdot \sum_{(i_2, \ldots, i_{d})
\in
\lattice_{d - 1, \nbar} }\left| \thetastar( \pihat_1(i_1) , i_2, \ldots, i_
{d} ) - \thetastar ( i_1 , i_2, \ldots,
i_{d} ) \right| \notag \\
&\qquad \qquad \qquad \qquad = 2 \cdot \Biggl| \sum_{(i_2,
\ldots, i_
{d})
\in
\lattice_{d - 1, \nbar} } \bigg\{ \thetastar( \pihat_1(i_1), i_2, \ldots, i_
{d} ) - \thetastar ( i_1, i_2, \ldots,
i_{d} ) \bigg\} \Biggr| \notag \\
&\qquad \qquad \qquad \qquad = 2 \cdot | \taustar_1 (\pihat_1
(i_1)) - \taustar_1
(i_1) |, 
\end{align*}
where, as before, the first equality follows since the set of scalars 
\[ 
\Big\{ \thetastar( \pihat_1(i_1),
i_2, \ldots, i_
{d} ) - \thetastar(i_1, i_2, \ldots,
i_{d} ) \Big\}_{(i_2, \ldots, i_{d})
\in
\lattice_{d - 1, \nbar}}
\]
all have the same sign by our monotonicity assumption $\thetastar \in \Mspace
(\lattice_{d, n})$.  

Now,
\begin{align*}
\sum_{i_1 = 1}^{n_1}  | \taustar_1 (\pihat_1(i_1)) - \taustar_1 
(i_1) | &\leq \sum_{i_1 = 1}^{n_1} \bigg\{ | \tauhat_1 (\pihat_1(i_1)) -
\taustar_1 
(i_1) | + | \tauhat_1 (\pihat_1(i_1)) - \taustar_1 
(\pihat_1(i_1)) | \bigg\} \\
&\stackrel{\2}{\leq} \sum_{i_1 = 1}^{n_1} \bigg\{ | \tauhat_1 (i_1) -
\taustar_1 
(i_1) | + | \tauhat_1 (\pihat_1(i_1)) - \taustar_1 
(\pihat_1(i_1)) | \bigg\} \\
&= 2 \| \tauhat_1 - \taustar_1 \|_1
\end{align*}
where step $\2$ follows from the rearrangement inequality for the $\ell_1$
norm~\citep{vince1990rearrangement}, since $\tauhat_1$ and $\taustar_1$ are
sorted in increasing order along the permutations $\pihat_1$ and $\pistar_1 =
\id$, respectively. Putting together the
pieces, we
have shown that
\begin{align*}
\|
\thetastar\{ \pihat_1, \ldots, \pihat_d \} - \thetastar \|_2 \leq \sqrt{ 4 \|
\tauhat_1 - \taustar_1 \|_1 } + \| \thetastar\{ \pistar_1, \pihat_2, \ldots,
\pihat_d \} - \thetastar \|_2.
\end{align*}
Proceeding inductively, we have
\begin{align*}
\|
\thetastar\{ \pihat_1, \ldots, \pihat_d \} - \thetastar \|^2_2 \leq \biggl(
\sum_{j = 1}^d \sqrt{ 4 \|
\tauhat_j - \taustar_j \|_1 } \biggr)^2 \leq 4 d \sum_{j =1}^d \| \tauhat_j -
\taustar_j \|_1,
\end{align*}
and this provides a bound on the approximation error that is of the claimed
order.
\qed

\subsection{Proof of Theorem~\ref{thm:adapt-lb}}
We handle the case where the projection in Definition~\ref{def:pp} is onto
unbounded tensors; the
bounded case follows identically. 
For each tuple of
permutations $\pi_1, \ldots,
\pi_d$, define the estimator
\begin{align*}
\thetahat_{\pi_1, \ldots, \pi_d} \defn \argmin_{\theta \in \Mspace(\lattice_{d,
n}; \pi_1, \ldots, \pi_d)} \| Y - \theta \|_2.
\end{align*}
By definition, any permutation-projection based estimator must equal $\thetahat_
{\pi_1, \ldots, \pi_d}$ for some choice of permutations $\pi_1, \ldots, \pi_d
\in \Pspace_{n_1}$. Our strategy will thus be to lower bound the risk $\min_
{\pi_1, \ldots, \pi_d} \| \thetahat_{\pi_1,
\ldots, \pi_d} - \thetastar \|^2_2$ for a particular choice of $\thetastar$.
To that end, let us analyze the risk of the individual estimators around the
point $\thetastar = 0$. Define the positive scalar $t_0$ via
\begin{align*}
t_0 &\defn \argmax_{t \geq 0} \left\{ \EE \sup_{\theta \in \Mspace
(\lattice_{d, n}; \pi_1, \ldots, \pi_d) \cap \mathbb{B}_2(t)} 
\inprod{\epsilon}{\theta} - t^2/2 \right\} \\
&= \argmax_{t \geq 0} \left\{ \EE \sup_{\theta \in \Mspace
(\lattice_{d, n}) \cap \mathbb{B}_2(t)} 
\inprod{\epsilon}{\theta} - t^2/2 \right\} \\
&\stackrel{\1}{=} \argmax_{t \geq 0} \left\{ t \cdot \EE \sup_{\theta \in \Mspace
(\lattice_{d, n}) \cap \mathbb{B}_2(1)} 
\inprod{\epsilon}{\theta} - t^2/2 \right\} \\
&= \EE \sup_{\theta \in \Mspace
(\lattice_{d, n}) \cap \mathbb{B}_2(1)} 
\inprod{\epsilon}{\theta},
\end{align*}
where step $\1$ follows since since $\Mspace
(\lattice_{d, n})$ is a cone, so for each $\theta \in \Mspace
(\lattice_{d, n}) \cap \mathbb{B}_2(t)$, we may write $\theta = t \theta'$
for some $\theta' \in \Mspace
(\lattice_{d, n}) \cap \mathbb{B}_2(1)$.

Applying~\citet[Theorem 1.1]{Cha14} yields that for each tuple $\pi_1, \ldots,
\pi_d$, we have
\begin{align} \label{eq:u-bound-Chat}
\Pr \left\{ \left| \| \thetahat_{\pi_1, \ldots, \pi_d} - \thetastar \|_2 -  t_0
\right|
\geq u \sqrt{t_0} \right\} \leq 3 \exp \left( -\frac{u^4}{32(1 + u / 
\sqrt{t_0})^2}
\right) \text{ for each } u \geq 0.
\end{align}
Furthermore, applying~\citet[Proposition 5]{han2019} yields the lower bound
\begin{align} \label{eq:lower-bound-t}
t_0 \geq c_d \cdot n^{1/2 - 1/d} \quad \text{ for each } d \geq 3.
\end{align}
Substituting the value $u = \sqrt{t_0} / 2$ into the bound
\eqref{eq:u-bound-Chat} and using the lower bound~\eqref{eq:lower-bound-t} on
$t_0$ yields, for each fixed tuple of permutations $\pi_1, \ldots, \pi_d$, the
high probability bound
\begin{align*}
\Pr \biggl\{ \| \thetahat_{\pi_1, \ldots, \pi_d} - \thetastar \|_2 \leq c_d
 \cdot
n^{1/2 -1/d} \biggl\} \leq 3 \exp\left\{ - c'_d \cdot n^{1 -2/d} \right\},
\end{align*}
where the pair $(c_d, c'_d)$ are different constants that depends on $d$ alone.
Applying a union bound over all choices of permutations now yields
 \begin{align*}
\Pr \left\{ \min_{\pi_1, \ldots, \pi_d} \| \thetahat_{\pi_1, \ldots, \pi_d} -
\thetastar \|_2^2 \leq c_d \cdot
n^{1 -2/d} \right\} \leq 3\exp\left\{ - c'_d \cdot n^{1 -2/d} + d n_1 \log n_1
\right\}.
\end{align*}
Now for each $d \geq 4$, there is a large enough constant $C_d > 0$ depending on
$d$ alone such that if $n \geq C_d$, then $c'_d \cdot n^{1 -2/d} \geq 2d
n_1
\log n_1$. Consequently, if $n \geq C_d$, then
\begin{align*}
\Pr \bigg\{ \min_{\pi_1, \ldots, \pi_d} \| \thetahat_{\pi_1,
\ldots, \pi_d} - \thetastar \|_2^2 &\geq c_d \cdot
n^{1 -2/d} \bigg\} \geq 1/2 \quad \text{ and } \\
\EE \big[ \min_{\pi_1,\ldots, \pi_d}
\| \thetahat_{\pi_1,
\ldots, \pi_d} - \thetastar \|_2^2 \big] &\geq c_d \cdot
n^{1 -2/d}
\end{align*}
for a sufficiently small constant $c_d > 0$ depending only on $d$. Thus, any
permutation-projection based estimator $\thetahat$ must satisfy
\begin{align*}
\EE \left[ \| \thetahat -
\thetastar \|_2^2
\right] \geq c_d \cdot
n^{1 -2/d}.
\end{align*}

On the other hand, we have
$\thetastar
\in \Mspace^{\kfull_0, \sset_0}_{\perm} (\lattice_{d, n})$ with $\sset_0 = 
(1, \ldots, 1)$ and $\kfull_0 = ((n_1), \ldots, (n_1))$. Thus, 
$s(\thetastar) = 1$ and $k^*(\thetastar) = n_1$, and Proposition
\ref{prop:adapt-funlim} yields the upper bound
$\minimax_{d, n}(\kfull_0,
\sset_0) \lesssim 1/n$.
Thus,
the
adaptivity index of any permutation-projection based estimator $\thetahat$ must
satisfy
\begin{align} \label{eq:LSE-projection}
\adapt(\thetahat) \geq \adapt^{\kfull_0, \sset_0}(\thetahat) \geq c_d
\cdot n^{1 - 2/d},
\end{align}
and this completes the proof.
\qed


\section{Discussion} \label{sec:discussion}

We considered the problem of estimating a multivariate isotonic regression function on the lattice from noisy observations that were also permuted along each coordinate, and established several results. In this section, we summarize these results, and discuss some related and open questions.

\vspace{1mm}

\noindent {\it Summary of results.} First, we showed that unlike in the bivariate case, computationally efficient estimators are able to achieve the minimax lower bound for estimation of bounded tensors in this class. Second, when the tensor is also structured, in that it is piecewise constant on a $d$-dimensional partition with a small number of blocks, we showed that the fundamental limits of adaptation are still nonparametric. Third, by appealing to the hypergraph planted clique conjecture, we also argued that the adaptivity index of polynomial time estimators is significantly poorer than that of their inefficient counterparts. The second and third phenomena are both significantly different from the case without unknown permutations. Fourth, we introduced a novel Mirsky partition estimator that was simultaneously optimal both in worst-case risk and adaptation, while being computable in sub-quadratic time. This procedure also enjoys better adaptation properties than existing estimators when $d = 2$ (see Appendix~\ref{app:adapt}), and its computational complexity adapts to structure in the underlying tensor. Our results for the Mirsky partition estimator are particularly surprising given that a large class of natural estimators does not exhibit fast adaptation in the multivariate case. Finally, we also established risk bounds and structural properties (see Appendix~\ref{app:iso}) for natural isotonic regression estimators without unknown permutations.

\vspace{1mm}

\noindent {\it Dependence on signal strength.} Let us briefly comment on a particular facet of our results that was not emphasized in Section~\ref{sec:funlims}: The dependence of the derived rates on the signal-to-noise ratio of the problem. In particular, suppose that the true signal $\thetastar \in \infball(r)$ for some positive scalar $r$. Then how do our results in Propositions~\ref{prop:full-funlim} and~\ref{cor:lse} change? By carefully repeating the steps in the respective proofs, it can be shown that for all $d \geq 2$, the BLSE over the class of isotonic functions with unknown permutations achieves the worst-case risk bound
\begin{align*}
\sup_{\thetastar \in \Mspace_{\perm} 
(\lattice_{d, n}) \cap
\infball(r)}
\Rspace_n (\thetahatblse, \thetastar) \leq C \cdot ( r n^{-1/d} \log^{2} n + n^{- (1 - 1/d)} \log n).
\end{align*}
On the other hand, the LSE (without boundedness constraints) is shown by our techniques (for all $d \geq 2$) to achieve the risk bounds
\begin{subequations} \label{eq:our-lse-bd}
\begin{align} 
&\sup_{\thetastar \in \Mspace_{\perm}(\lattice_{d, n}) \cap \infball(r)}
\Rspace_n\bigl(\thetahatlse(\Mspace_{\perm}(\lattice_{d, n}), Y) ,
\thetastar\bigr) \\
&\qquad \qquad \qquad   \qquad \qquad \qquad \leq C \cdot (\sqrt{\log n} + r) \cdot (r n^{-1/d} \log^{2} n + n^{- (1 - 1/d)} \log n) \notag
\end{align}
and
\begin{align} \label{eq:iso-lse-our-bd}
\sup_{\thetastar \in \Mspace(\lattice_{d, n}) \cap \infball(r)} \Rspace_n\bigl
(\thetahatlse(\Mspace(\lattice_{d, n}), Y), \thetastar\bigr) \leq C \cdot \left\{ (\sqrt{\log n} + r) \cdot r n^{-1/d} \log^{2} n + n^{-1} \right\}
\end{align}
\end{subequations}
over the sets defined with and without unknown permutations, respectively.

It is instructive to compare the latter bound~\eqref{eq:iso-lse-our-bd} on the vanilla isotonic regression estimator with the one that can be derived from the proof of~\citet[Theorem 1]{han2019}. There, the authors show the worst case bound
\begin{align} \label{eq:han-lse-bd}
\sup_{\thetastar \in \Mspace(\lattice_{d, n}) \cap \infball(r)} \Rspace_n\bigl
(\thetahatlse(\Mspace(\lattice_{d, n}), Y), \thetastar\bigr) \leq C \cdot (r n^{-1/d} \log^{4} n + n^{-2/d} \log^8 n).
\end{align}
Comparing the bounds~\eqref{eq:iso-lse-our-bd} and~\eqref{eq:han-lse-bd} in terms of their dependence on $r$, we see that our bound~\eqref{eq:iso-lse-our-bd} is sharper in the regime $r \to 0$, since the error floor is much smaller: $n^{-1} \ll n^{-2/d} \log^8 n$. On the other hand, the bound~\eqref{eq:han-lse-bd} is better when $r$ is very large, i.e., growing with $n$. Since both bounds are on the same estimator, one can combine them to obtain the guarantee
\begin{align*}
&\sup_{\thetastar \in \Mspace(\lattice_{d, n}) \cap \infball(r)} \Rspace_n\bigl
(\thetahatlse(\Mspace(\lattice_{d, n}), Y), \thetastar\bigr) \\
&\qquad \qquad \qquad \qquad  \qquad \leq C \cdot \left( r n^{-1/d} \log^{5/2} n \right) \cdot \min \left\{\sqrt{\log n} + r, \log^2 n \right\} + Cn^{-1},
\end{align*}
which inherits the favorable properties of both bounds.

\vspace{1mm}

\noindent {\it Open questions.} Our work raises many interesting questions from both the modeling and theoretical standpoints. From a modeling perspective, the isotonic regression model with unknown permutations should be viewed as just a particular nonparametric model for tensor data. There are many ways one may extend these models. For instance, taking a linear combination of $k > 1$ tensors in the set $\Mspace_{\perm}(\lattice_{d, n})$ directly generalizes the class of nonnegative tensors of (canonical polyadic) rank~$k$. Studying such models would parallel a similar investigation that was conducted in the case $d = 2$ for matrix estimation~\citep{shah2019low}. It would also be interesting to incorporate latent permutations within other multidimensional nonparametric function estimation tasks that are not shape constrained; a similar study has been carried out in the case $d = 2$ in the context of graphon estimation~\citep{GaoLuZho15}. In the case $d \geq 3$, the analogous application would be in modeling hypergraphs in a flexible manner, going beyond existing models involving planted partitions~\citep{abbe2013conditional,ghoshdastidar2017consistency}. 

Methodological and theoretical questions also abound. First, note that in typical applications, $n_1$ will be very large, and we will only observe a subset of entries chosen at random. Indeed, when $d = 2$,~\citet{mao2018towards} showed that the fundamental limits of the problem exhibit an intricate dependence on the probability of observing each entry and the dimensions of the tensor. What are the analogs of these results when $d \geq 3$? The second question concerns adaptation. Our focus on indifference sets to define structure in the tensor was motivated by the application to multiway comparisons, but other structures are also interesting to study. For instance, what does a characterization of adaptation look like when there is simply a partition into hyper-rectangles---not necessarily Cartesian products of one-dimensional partitions---on which the tensor is piecewise constant? Such structure has been extensively studied in the isotonic regression literature~\citep{ChaGunSen18,han2019,deng2018isotonic}. What about cases where $\thetastar$ is a nonnegative tensor of rank $1$? It would be worth studying spectral methods for tensor estimation for this problem, especially in the latter case. Finally, an interesting open question is whether the block-isotonic regression estimator of~\citet{fokianos2017integrated} can be employed in conjunction with permutation estimation to yield an estimator that is minimax optimal as well as adaptive. For instance, we could replace step II in the Borda count estimator with the block-isotonic regression estimator~\citep{fokianos2017integrated,deng2018isotonic}, and call this estimator $\thetahat_{\block}$. Note that $\thetahat_{\block}$ is \emph{not} permutation-projection based, so it is possible that it achieves the optimal adaptivity index for polynomial time algorithms while remaining minimax optimal. 
On the other hand, the best existing algorithms for the block-isotonic estimator require time $\order(n^3)$, as opposed to our estimation procedure that runs in time $\ordertil(n^{3/2})$ in the worst-case, and faster if the problem is structured. From a technical perspective, understanding the behavior of the estimator $\thetahat_{\block}$ in our setting is intricately related to the oracle properties of the block-isotonic regression estimator around permuted versions of isotonic tensors.

Finally, let us discuss in more detail two independently interesting questions in shape-constrained estimation.
The first was raised in the context of adaptation in  Section~\ref{sec:adapt-bounded-MP}. 
Can we obtain a  characterization of the minimax risk of estimation over the set $\Mspace^{\kfull, \sset}(\lattice_{d, n}) \cap \infball(1)$ \emph{as a function of} the pair $(\kfull, \sset)$? In spite of multiple investigations of related issues~\citep{chatterjee2015risk,bellec2015sharp,ChaGunSen18,gao2020estimation}, this question is complementary and does not seem to have been addressed (or even asked) in the literature. While a complete answer to this question would make significant progress towards characterizing adaptation in the bounded case (with unknown permutations), the question is one of independent interest even in the case of univariate isotonic regression, as witnessed by the following examples. First, suppose that the partition into $s$ pieces induces blocks of equal sizes, i.e., $k_1 = \cdots = k_s = n/s$. For this pair $(\kfull, \sset)$, existing results (see, e.g.,~\citet{bellec2015sharp}) show that we have
\begin{align} \label{eq:bellec-sharp}
\inf_{\thetahat} \; \sup_{\thetastar \in \Mspace^{\kfull, \sset}(\lattice_{d, n}) \cap \infball(1)} \;
\Rspace_n (\thetahat, \thetastar) \gtrsim \frac{1}{n} \min \bigl(s, n^{1/3} \bigr),
\end{align}
and it can be shown that this bound is matched by the block-wise isotonic estimator that first averages the observations within each block and then performs isotonic regression on the result. Indeed, the vanilla isotonic regression estimator (without averaging within blocks) also achieves the same rate up to a logarithmic factor~\citep{chatterjee2015risk}.
On the other hand, consider the second case in which the pair $(\kfull, \sset)$ satisfies $k_1 = \cdots = k_{s-1} = 1$ and $k_s = n - (s - 1)$. By treating the first $s-1$ entries of the problem as standard isotonic regression and setting the last $n - (s - 1)$ entries (deterministically) to $1$, one can establish the minimax lower bound
\begin{align} \label{eq:second-bound-main}
\inf_{\thetahat} \; \sup_{\thetastar \in \Mspace^{\kfull, \sset}(\lattice_{d, n}) \cap \infball(1)} \;
\Rspace_n (\thetahat, \thetastar) \gtrsim \frac{(s-1)^{1/3} + 1 }{n}.
\end{align}
Once again, this bound can be achieved by the block-wise isotonic regression estimator. Comparing the bounds~\eqref{eq:bellec-sharp} and~\eqref{eq:second-bound-main}, we see that the minimax risk in this case must depend on the values of the block sizes $k_1, \ldots, k_s$, and not just the number of blocks $s$. Characterizing the risk as a function of $\kfull$ and $\sset$ (especially in the general multivariate case) is thus likely to be a challenging problem.

The second question is about measuring adaptation with respect to a larger class of structured tensors. Note that we considered isotonic tensors with piecewise constant structure on a hyper-rectangular partition that was formed by indifference sets along different dimensions, i.e., a Cartesian product of univariate partitions. While our focus on this type of structure was motivated by the application to multi-way comparisons---in which each block of the univariate partition represents a set of items among which we are indifferent---a more general type of structure has been studied in the isotonic regression literature (without unknown permutations), in which we have a general hyper-rectangular partition that is not necessarily a Cartesian product of univariate partitions~\citep{ChaGunSen18,han2019,deng2018isotonic}. An interesting open question is to derive analogs of Proposition~\ref{prop:adapt-funlim} and Theorems~\ref{thm:block-risk} and~\ref{thm:adapt-lb} under this more general notion of structure. The key difference in such a bound is that 
while the tuple $\kfull = (\kset_1, \ldots, \kset_d)$ of indifference set sizes (and the associated functionals $s$ and $k^*$) suffice to characterize structure in the tensor in the setting of the current paper, 
a different set of quantities would be needed to measure complexity in this more general class of tensors. Note that one can always refine a general hyper-rectangular 
partition into a Cartesian product of univariate partitions, though this can increase the number of hyper-rectangular pieces exponentially in the dimension and yield suboptimal rates.

\subsection*{Acknowledgments}

A.P. was partially supported by a Simons--Berkeley research fellowship when part of this work was done. R.J.S. was supported by EPSRC grants EP/P031447/1 and EP/N031938/1. We thank the anonymous referees for their feedback, which improved the scope and presentation of the paper.
\bibliographystyle{abbrvnat}
\bibliography{isotonic_final}

\newpage

\appendix


\section{Adaptation properties in the low-dimensional setting} \label{app:adapt}

In this appendix, we collect some results regarding the adaptation
properties of the CRL and Mirsky partition estimators when $d \in \{2, 3\}$. We
consider both bounded and unbounded parameter spaces.

\subsection{Adaptation of CRL estimator in unbounded case}

Let us begin by writing down the CRL estimator proposed by
\citet{ShaBalWai16-2} in the multivariate case\footnote{There are some
minor differences between the estimator presented here and that of~\citet{ShaBalWai16-2}: for instance, we do not impose a boundedness constraint in
the least squares step of the estimator, and nor do we impose the symmetry
constraints inherent to the SST class~\citep{ShaBalWai16-2}. Finally, step Ib
is only needed in the unbounded case, because the scores $\tauhat_j$ alone are
insufficient for permutation estimation and we also require the (entry-wise)
statistics $\Deltahat^{\max}_j$ computed on each pair of indices.}
$d
\geq
3$. Recall the score vectors $\tauhat_1, \ldots, \tauhat_d$ from equation
\eqref{eq:est-scores}.

\smallskip
\noindent \hrulefill

\noindent \underline{\bf Algorithm: Count-Randomize-Least-Squares 
(CRL) estimator} 
\begin{enumerate}
\item[I.] For each $j \in [d]$: 
\begin{enumerate}
  \item[a.] (Count): Let $\pihatcount_j$ be any permutation along which 
  entries of $\tauhat_j$ are non-decreasing; i.e.,
  \begin{align*}
  \tauhat_j \bigl(\pihatcount_j (k)\bigr) \leq \tauhat_j \bigl(\pihatcount_j 
  (\ell)\bigr) \text{ for all } 1 \leq k \leq \ell \leq n_j.
  \end{align*}
  \item[b.] (Prune): For each pair $k, \ell \in [n_j]$, if 
  \begin{align*}
  \pihatcount_j(k) > \pihatcount_j(\ell) \quad \text{ and } \quad \Deltahat^
  {\max}_{j}(k, \ell) > 8 \sqrt{\log n}, 
  \end{align*}
  then flip the pair $(k, \ell)$ in the ordering $\pihatcount_j$. Let $\pitil_j$
  denote the permutation obtained at the end of this process. 

  If a collision (inconsistent ordering) occurs during this process,
  set $\pihatcrl_j = \pihatcount_j$ and skip step Ic.
  \item[c.] (Randomize): Compute the largest set of indices $T^j_{\max}$ (if there are multiple such sets, choose one of them arbitrarily)
  such that for all $k, \ell \in T^j_{\max}$, we have 
  \begin{align*}
  \Deltahat^{{\sf sum}}_j(k, \ell) \leq 8 \sqrt{\log n} \cdot n^{\frac{1}{2}(1 -
  1/d)} \quad \text{ and } \quad \Deltahat^{\max}_j(k, \ell) \leq 8 \sqrt{\log
  n}.
  \end{align*}
  Choose a
  uniformly random permutation on the set $T^j_{\max}$ independently of the data, and let
  $\pihatcrl_j$ be the
  composition of $\pitil_j$ with this permutation.  
\end{enumerate}
\item[II.] (Least squares): Project the observations onto the class of
isotonic tensors that are consistent with the permutations $\pihatcrl_1,
\ldots, \pihatcrl_d$ to obtain
\begin{align*}
\thetahatcrl \in \argmin_{\theta \in \Mspace(\lattice_{d, n}; \pihatcrl_1,
\ldots, \pihatcrl_d)} \; \ell_n^2 (Y, \theta).
\end{align*}
\end{enumerate}
\noindent \hrulefill
\bigskip

In the case $d = 2$,~\citet{ShaBalWai16-2} showed that under a
definition of adaptivity index involving bounded function classes $\Mspace^
{\kfull, \sset}_{\perm}(\lattice_{d, n}) \cap \infball(1)$ (see equation
\eqref{eq:AI-unbounded} to follow), the CRL estimator has the
smallest adaptivity index attainable by polynomial time procedures.
For completeness, we extend this analysis to our setting, where adaptation
is measured over  a hierarchy of unbounded sets $\Mspace^{\kfull, \sset}_{\perm}
(\lattice_{d, n})$, $\kfull \in \Kfull_{\sset}, \sset \in \lattice_{d, n}$ (see
equation~\eqref{eq:global-AI}). We also handle the case
$d = 3$. 

\begin{proposition} \label{prop:AI-CRL-unbounded}
There is a universal positive constant $C$ such
that
%
\begin{align*}
\sup_{\thetastar \in \Mspace^{\kfull, \sset}_{\perm}(\lattice_{d, n}) }
\Rspace_n
(\thetahatcrl, \thetastar) \leq \frac{C}{n} \cdot
\begin{cases}
s \log^8 n + (n_1 - k^*) \cdot n^{1/4} \log n \quad&\text{ if } d = 2, \\
s^{2/3} \cdot n^{1/3} \log^8 n + (n_1 - k^*) \cdot n^{1/3} \log n &\text{ if } d
= 3.
\end{cases}
\end{align*}
%
Consequently, the adaptivity index~\eqref{eq:global-AI} of the CRL estimator satisfies
\begin{align*}
\adapt(\thetahatcrl)  \leq C \cdot
\begin{cases}
n^{1/4} \log n \quad&\text{ if } d = 2, \\
n^{1/3} \log^8 n &\text{ if } d = 3.
\end{cases}
\end{align*}
\end{proposition}
Comparing Proposition~\ref{prop:AI-CRL-unbounded} with Theorem
\ref{thm:hardness}, we see that for $d = 2, 3$, the CRL estimator
attains (up to a poly-logarithmic factor) the smallest adaptivity index possible
for polynomial time procedures if we assume Conjecture~\ref{conj:HPC}. On the other hand, 
Theorem~\ref{thm:adapt-lb} already showed
that when $d \geq 4$, the CRL estimator is unable to attain the
polynomial time optimal adaptivity index, since the estimator is
permutation-projection based.

\begin{proof}[Proof of Proposition~\ref{prop:AI-CRL-unbounded}]
This proof borrows tools from the proofs of Theorem~\ref{thm:block-risk}(b) and
Proposition~\ref{prop:borda-worst-case}. For the rest of this proof, let
$\pihat_j = \pihatcrl_j$ for
convenience. Also, let $\pistar_1 = \cdots =
\pistar_d = \id$ wlog,
so that we have $\thetastar \in
\Mspace^
{\kfull, \sset}(\lattice_{d, n})$.
The decomposition~\eqref{eq:borda-decomp}
for the Borda count estimator still applies to yield
\begin{align*}
\| \thetahatcrl - \thetastar \|_2 \leq \| \thetatil - \thetastar\{ \pihat_1,
\ldots, \pihat_d\} \|_2 +  2 \cdot \| \thetastar\{ \pihat_1, \ldots, \pihat_d \}
- \thetastar \|_2,
\end{align*}
where $\thetatil$ denotes the projection of the tensor $\thetastar\{ \pihat_1,
\ldots, \pihat_d\} + \epsilon$ onto the set $\Mspace(\lattice_{d, n};
\pihat_1, \ldots, \pihat_d)$. With this decomposition at hand, let us now
bound the two terms separately. Recall the high probability events
$\Espace_1$ and $\Espace_2$ from equations~\eqref{eq:hp-events}.

\noindent \underline{Estimation error bound}: Similarly to the proof of Claim~\ref{clm:graph} in Section~\ref{sec:claim-section}, it can be shown that on the event $\Espace_1 \cap \Espace_2$, there are no collisions in the permutation estimation step Ib of the CRL
estimator. We prove the estimation error bound
\begin{align} \label{eq:est-error-cases}
\| \thetatil - \thetastar\{ \pihat_1,
\ldots, \pihat_d\} \|^2_2 \leq
\begin{cases}
C \left\{ n \cdot \left( \frac{s}{n} \right)^{2/d} \log^8 n + (n_1 - k^*) \log n
\right\}&\text{ on 
} \Espace_1 \cap \Espace_2 \cap \Espace \\
\| \epsilon \|_2^2 \quad &\text{ pointwise,}
\end{cases}
\end{align}
where $\Espace$ is a third event (to be defined shortly) such that $\Pr \{ \Espace_1 \cap \Espace_2
\cap \Espace \} \geq 1 - 5n^{-7}$. 
Now applying Lemma
\ref{lem:RV} from the appendix
yields
\begin{align*}
\EE \left[\| \thetatil - \thetastar\{ \pihat_1,
\ldots, \pihat_d\} \|^2_2 \right] &\leq C \left\{ n \left( \frac{s}{n} \right)^
{2/d} \log^8 n  \! + \! (n_1 \! - \! k^*) \log n \right\} + \sqrt{5n^{-7}} \cdot \sqrt{ \EE[ \| \epsilon
\|_2^4 ] } \\
&\leq C \left\{ n \left( \frac{s}{n} \right)^
{2/d} \log^8 n \! + \! (n_1 \! - \! k^*) \log n \right\}.
\end{align*}

Let us now proceed to the proof of the estimation error bound
\eqref{eq:est-error-cases}. Note that the second case is a direct consequence of
the fact that $\thetatil$ is a projection onto the convex set $\Mspace(\lattice_{d, n}; \pihat_1, \ldots, \pihat_d)$. We thus dedicate
the rest of this proof to the proof of the first case.
For a fixed tuple of
permutations $\pi_1,
\ldots, \pi_d$, let $\thetahat_{\pi_1, \ldots, \pi_d}$ denote the projection of the tensor $\thetastar\{ \pi_1,
\ldots, \pi_d\} + \epsilon$ onto the set $\Mspace(\lattice_{d, n};
\pi_1, \ldots, \pi_d)$. Applying~\citet[Theorem 3]{han2019}, note
that for each $\thetastar \in \Mspace^{\kfull, \sset}(\lattice_{d, n})$, we
obtain
\begin{align*}
\frac{1}{n} \cdot \EE \left[ \|\thetahat_{\pi_1, \ldots, \pi_d} - \thetastar\{ \pi_1, \ldots,
\pi_d \} \|_2^2 \right] \leq C \left( \frac{s}{n} \right)^{2/d} \log^8 \left(
\frac{Cn}{s} \right).
\end{align*}
Combining this expectation bound with Lemma~\ref{lem:vdG-W}(b) and adjusting constants appropriately, we
obtain the high probability bound
\begin{align*}
&\Pr \left\{ \|\thetahat_{\pi_1, \ldots, \pi_d} - \thetastar\{ \pi_1, \ldots,
\pi_d \} \|_2^2 \geq C \left( n \cdot \left( \frac{s}{n} \right)^{2/d} \log^8 n  + 
(n_1 - k^*) \log n \right) \right\} \\
&\qquad \qquad \qquad \qquad \qquad \qquad \qquad \qquad \qquad \qquad \qquad  \leq \exp \left\{ - 8 (n_1 -
k^* + 1) \log n \right\}.
\end{align*}
Owing to the randomization step of the algorithm, 
at least $|T^j_{\max}|$ indices are
shuffled in a \emph{data-independent} manner along dimension $j$.
Similarly to Lemma~\ref{lem:block-structure}, it can be shown that
on the event $\Espace_1 \cap \Espace_2$, we have
$|T^j_{\max}| \geq k^j_{\max}$. Thus, applying Lemma~\ref{lem:num-partitions}
yields that the number of
problematic
data-dependent
partial orders along dimension $j$ is
at most $\exp( 3(n_1 - k^*) \log n_1)$. Taking
a union bound over at most $\exp( 3d(n_1 - k^*) \log n_1) = \exp( 3(n_1 - k^*)
\log n)$ such partial
orders, we have 
\begin{align*}
&\Pr \left\{ \max_{\pi_1, \ldots, \pi_d} \|\thetahat_{\pi_1, \ldots, \pi_d} - \thetastar\{ \pi_1, \ldots,
\pi_d \} \|_2^2 \geq C \left( n \cdot \left( \frac{s}{n} \right)^{2/d} \log^8 n  + 
(n_1 - k^*) \log n \right) \right\} \\
& \qquad \qquad \qquad \qquad \qquad \qquad \qquad \qquad\qquad \qquad\leq \exp \left\{ - 5 (n_1 -
k^* + 1) \log n \right\}.
\end{align*}
Letting $\Espace$ denote the complement of this event,
we have shown that the event $\Espace_1 \cap \Espace_2 \cap \Espace$ occurs with probability at least $1 - 5n^{-7}$ and that
on this event, we have
\begin{align*}
\|\thetahat_{\pihat_1, \ldots, \pihat_d} - \thetastar\{ \pihat_1, \ldots,
\pihat_d \} \|_2^2 \leq  C \left( n \cdot  \left( \frac{s}{n} \right)^{2/d} \log^8 n  + 
(n_1 - k^*) \log n \right).
\end{align*}

\noindent \underline{Approximation error bound}:  Peeling off the first
dimension from
the error as usual, we have
\begin{align*}
\| \thetastar\{\pihat_1, \ldots, \pihat_d \}
- \thetastar \|_2 \leq \| \thetastar\{\pihat_1, \id, \ldots, \id \} - \thetastar
\|_2 + \| \thetastar\{\id, \pihat_2, \ldots, \pihat_d \} - \thetastar \|_2.
\end{align*}
Now recall the random variables $T_j$ and
$U$~\eqref{eq:TU-RV} and the high probability events~\eqref{eq:hp-events}. We claim that the
square of the first term may be
bounded as
\begin{align} \label{eq:one-dim-crl}
\| \thetastar\{\pihat_1, \id, \ldots, \id \} - \thetastar
\|^2_2  \lesssim (n_1 - k^*) \cdot
\begin{cases}
U \cdot T_1 \quad &\text{on the event } \Espace_1 \cap \Espace_2 \\
T_1^2 \quad &\text{pointwise.}
\end{cases}
\end{align}
The rest of the argument is completed exactly as before (see the proof of
Lemma~\ref{lem:approx-error}), and so we focus on establishing the two cases of
equation \eqref{eq:one-dim-crl}.

\smallskip

\noindent \underline{Case 1}: On the event $\Espace_1 \cap \Espace_2$,
 there are no collisions in step Ib of the CRL estimator, and so we have
$\pihat_1 (k) < \pihat_1(\ell)$ if
\begin{align*}
\tauhat_{1}(\ell) - \tauhat_1(k) &> 8 \sqrt{\log n} \cdot n^{\frac{1}{2}(1 -
1/d)}
\qquad
\text{ or} \\
\max_{i_2, \ldots, i_d} \; Y(\ell, i_2, \ldots, i_d) - Y(k, i_2, \ldots, i_d) &>
8 \sqrt{\log n}.
\end{align*}
Thus, proceeding exactly as in equations~\eqref{eq:two-terms}--\eqref{eq:final-bd}, we have
\begin{align*}
\| \thetastar\{\pihat_1, \id, \ldots, \id \} - \thetastar
\|^2_2 \leq U \cdot (n_1 - k^*) \cdot T_1.
\end{align*}

\noindent \underline{Case 2}: Let us now establish the pointwise bound. Just as before, it still holds that
$\pihat_1 (k) < \pihat_1(\ell)$ if
\begin{align*}
\tauhat_{1}(\ell) - \tauhat_1(k) > 8 \sqrt{\log n} \cdot n^{\frac{1}{2}(1 -
1/d)},
\end{align*}
since even after the pruning step, the estimated permutation is consistent with
the Borda counts. Therefore, we may proceed exactly as in equations
\eqref{eq:first-bd-case-two}--\eqref{eq:final-bd-case-two} to complete the proof
of this case.
\end{proof}

\subsection{Adaptation of the Mirsky partition estimator in the bounded case} \label{proof:prop4}
In this section, we prove Proposition~\ref{prop:AI-MP-bounded} from the main text.

\begin{proof}[Proof of Proposition~\ref{prop:AI-MP-bounded}]
The proof of this proposition follows almost immediately from some of our
earlier calculations. First, the decomposition~\eqref{eq:decomposition} still
applies to yield
\begin{align*}
\| \thetahatblock - \thetastar \|_2 \leq \| \blop( \blop (\thetastar; \Bhat) +
\epsilon; \Bhat ) -
\blop
(\thetastar;
\Bhat) \|_2 + 2 \| \blop(\thetastar; \Bhat) - \thetastar \|_2.
\end{align*}
Applying Lemma~\ref{lem:est-error} to the first term along with a union
bound as in the proof of Theorem~\ref{thm:block-risk}(b) (see Remark~\ref{rem:stronger-bound}) 
yields the bound
\begin{align*}
\| \blop( \blop (\thetastar; \Bhat) + \epsilon; \Bhat ) - \blop(\thetastar;
\Bhat) \|_2^2 &\lesssim \min \left\{ s + (n_1 - k^* + 1) \log n, n^{1/2} \log^
{5/2}
n
\right\} \\
&\lesssim \min \left\{ s, n^{1/2} \log^{5/2} n
\right\}
\end{align*}
with probability at least $1 - 4n^{-4}$. Note also
that when $d = 2$, we have $s \leq (n_1 - k^* + 1)^2$. Applying the
inequality $\min\{ a^2, b^2 \} \leq ab$ valid for any two positive scalars $a$
and $b$, we obtain
\begin{align*}
\| \blop( \blop (\thetastar; \Bhat) + \epsilon; \Bhat ) - \blop(\thetastar;
\Bhat) \|_2^2 &\lesssim (n_1 - k^* + 1) \cdot n^{1/4} \log^{5/4} n
\end{align*}
with probability at least  $1 - 4n^{-4}$. At the same time,
Lemma~\ref{lem:approx-error} still applies to yield the approximation error
bound
\begin{align*}
\EE \left[ \| \blop(\thetastar; \Bhat) - \thetastar \|^2_2 \right] \lesssim 
(n_1 - k^*) \cdot n^{1/4} \log n
\end{align*}
in expectation,
and a similar bound with high probability. 
In order to obtain the bound in expectation, note that $\| \thetahatblock -
\thetastar \|^2_2 \leq \| \epsilon \|_2^2$ pointwise and apply Lemma
\ref{lem:RV} to obtain
the claimed result on the risk.
For the claimed bounds on the adaptivity index, combine these
bounds with the minimax lower bound 
\begin{align*}
\inf_{\thetahat} \sup_{\thetastar \in \Mspace^{\kfull, \sset}_{\perm}(\lattice_{d, n}) \cap \infball(1) } \mathcal{R}_n (\thetahat, \thetastar) \geq C (n_1 - k^* + 1),
\end{align*}
which was shown in~\citet[Proposition 1]{ShaBalWai16-2}.
\end{proof}

\section{Technical results on the isotonic projection} \label{app:iso}

In this section, we collect some technical results on the isotonic projection
onto piecewise constant hyper-rectangular partitions. This is the operator given
by $\blop(\; \cdot \; ; \bl_1, \ldots,
\bl_d)$, which was defined in equation~\eqref{eq:blop-def}. 
Let us begin by defining some other helpful notation. Let $\Cspace(\lattice_{d,
n};
\bl_1, \ldots, \bl_d)$ denote the set of all tensors in $\Tspace_{d, n}$
that are piecewise constant on the $d$-dimensional blocks specified
by the Cartesian products of one-dimensional partitions
$\bl_1, \ldots, \bl_d$. Define
the operators $\mathcal{P}:
\Tspace_{d, n} \to \Tspace_{d, n}$
and $\mathcal{A}: \Tspace_{d, n} \to \Tspace_{d, n}$ as projection operators
onto the sets $\Mspace
(\lattice_{d, n}; \pi_1, \ldots, \pi_d)$ and $\Cspace(\lattice_{d, n};
\bl_1, \ldots, \bl_d)$, respectively, i.e., for each $\theta \in \Tspace_{d,
n}$, we have
\begin{subequations} \label{eq:ind-projections}
\begin{align}
\mathcal{P}(\theta; \pi_1, \ldots, \pi_d) &\in \argmin_{\theta' \in \Mspace
(\lattice_{d, n}; \pi_1, \ldots, \pi_d)} \| \theta - \theta' \|_2^2, \text{
and } \\
\mathcal{A}(\theta; \bl_1, \ldots, \bl_d) &\in \argmin_{\theta' \in \Cspace
(\lattice_{d, n}; \bl_1, \ldots, \bl_d)} \| \theta - \theta' \|_2^2.
\end{align}
\end{subequations}
Recall our notion of a permutation that is faithful to a one-dimensional ordered
partition from the proof of Lemma~\ref{lem:est-error}(b).
Finally, let $\thetabar_S$ denote the average of the
entries of $\theta \in \Tspace_{d, n}$ on the set
$S \subseteq \lattice_{d, n}$.

Our first technical lemma demonstrates that the operator $\Bspace$ can be
written as a composition of the operators $\Pscript$ and $\Aspace$, i.e., in
order to project onto the class of isotonic tensors that are piecewise
constant on hyper-rectangular blocks given by a $d$-dimensional ordered
partition, it suffices to first average all
entries within each block, and then project the result onto the class of
isotonic tensors whose partial orderings are consistent with the corresponding
one-dimensional ordered partitions.

\begin{lemma}[Composition] \label{lem:composition}
For each $j \in [d]$, let $\pi_j \in \Pspace_{n_1}$ be any permutation
that is faithful to the
ordered partition $\bl_j$. Then, for each $\theta \in \Tspace_{d, n}$, we have
\begin{align*}
\Bspace(\theta; \bl_1, \ldots, \bl_d) = \Pscript ( \; \Aspace
(\theta; \bl_1, \ldots, \bl_d) \; ; \pi_1, \ldots, \pi_d).
\end{align*}
\end{lemma}
\begin{proof}
To fix notation, suppose that $\bl_j$ is a partition of $[n_j]$ into $s_j$
blocks, and that $\prod_{j = 1}^d s_j = s$. Note that the ordered partitions
$\bl_1, \ldots,\bl_d$ induce a
hyper-rectangular partition of the lattice $\lattice_{d, n}$ into $\sbar$
pieces. Index each of these hyper-rectangles
by the corresponding member of the smaller lattice $\lattice_{d,
s_1, \ldots, s_d}$, and for each $x \in \lattice_{d,
s_1 \ldots, s_d}$, let $B_x \subseteq \lattice_{d, n}$ denote the indices of
hyper-rectangle $x$. With this notation, the projection operator for any
$\theta \in \Tspace_{d, n}$ takes the form
\begin{align*}
\Bspace(\theta; \bl_1, \ldots, \bl_d) &\in \argmin_{\mu \in \Mspace(\lattice_{d,
s_1, \ldots, s_d})} \sum_{x \in \lattice_{d, s_1, \ldots, s_d}} \sum_{w \in B_x}
\left(\theta_w - \mu_x \right)^2.
\end{align*}
The inner sum in the objective can be written as
\begin{align}
\sum_{w \in B_x} \left( \theta_w -
\mu_x \right)^2 
&= \sum_{w \in B_x} \left( \theta_w \! - \!
\thetabar_{B_x} \right)^2 +  \sum_{w \in B_x} \left(\thetabar_{B_x} \! - \! \mu_x
\right)^2 + 2 \left(\thetabar_{B_x} \! - \! \mu_x \right) \sum_{w \in B_x}
\left(
\theta_w -
\thetabar_{B_x} \right) \notag \\
&= \sum_{w \in B_x} \left( \theta_w -
\thetabar_{B_x} \right)^2 +  |B_x| \cdot \left(\thetabar_{B_x} - \mu_x
\right)^2, \label{eq:last-term-decomp}
\end{align}
where equation~\eqref{eq:last-term-decomp} follows since $\sum_{w \in B_x}
\left(\theta_w -
\thetabar_{B_x} \right) = 0$. Now, noting that the
first term of inequality~\eqref{eq:last-term-decomp} does not depend
on $\mu$, we have
\begin{align*}
\Bspace(\theta; \bl_1, \ldots, \bl_d) &\in \argmin_{\mu \in \Mspace(\lattice_{d,
s_1,\ldots, s_d})} \sum_{x \in \lattice_{d, s_1,\ldots, s_d}} |B_x| \cdot \left
(\thetabar_{B_x} - \mu_x
\right)^2.
\end{align*}
The proof is completed by noting that $\Aspace(\theta; \bl_1, \ldots, \bl_d)$ is
equal to $\thetabar_{B_x}$ on each block $B_x$, and so the optimization problem
above can be viewed as the projection of $\Aspace(\theta; \bl_1, \ldots, \bl_d)$
onto any set $\Mspace(\lattice_{d, n}; \pi_1, \ldots, \pi_d)$ such that
the permutations $\pi_1, \ldots, \pi_d$ are faithful to the ordered
partitions $\bl_1, \ldots, \bl_d$, respectively.
\end{proof}

\begin{lemma}[$\ell_\infty$-contraction] \label{lem:contraction}
For each $\theta, \theta' \in \Tspace_{d, n}$, ordered partitions $\bl_1,
\ldots, \bl_d$,
and
permutations $\pi_1, \ldots, \pi_d$, we have
\begin{subequations}
\begin{align}
\| \Aspace( \theta; \bl_1, \ldots, \bl_d) - \Aspace( \theta'; \bl_1, \ldots,
\bl_d)
\|_{\infty} &\leq \| \theta - \theta' \|_{\infty} \text{ and } 
\label{eq:mean-contract} \\
\| \Pscript( \theta; \pi_1, \ldots, \pi_d) - \Pscript( \theta'; \pi_1, \ldots,
\pi_d)
\|_{\infty} &\leq \| \theta - \theta' \|_{\infty}. \label{eq:perm-contract}
\end{align}
Consequently,
\begin{align} \label{eq:block-contract}
\| \Bspace( \theta; \bl_1, \ldots, \bl_d) - \Bspace( \theta'; \bl_1, \ldots,
\bl_d)
\|_{\infty} &\leq \| \theta - \theta' \|_{\infty}.
\end{align}
\end{subequations}
\end{lemma}
\begin{proof}
Owing to Lemma~\ref{lem:composition}, equation~\eqref{eq:block-contract} follows
directly from equations~\eqref{eq:mean-contract} and~\eqref{eq:perm-contract}.
Equation~\eqref{eq:mean-contract} is also immediate, since the operator
$\mathcal{A}$ simply averages entries within each partition, and the averaging
operation is trivially $\ell_\infty$-contractive.

The proof of equation~\eqref{eq:perm-contract} is slightly more involved.
First, since the
$\ell_\infty$ norm is invariant to the labeling of the entries of the tensor, it
suffices to establish the result when $\pi_j = \id$ for all $j \in [d]$.
We use the notation $\Pscript(\cdot) \defn \Pscript(\; \cdot \;; \id, \ldots,
\id)$,
for convenience. For each $x \in \lattice_{d, n}$, let $\Lspace(x)$ and
$\Uspace(x)$ denote the collections of lower and upper sets containing $x$,
respectively\footnote{Recall that an upper set $u \subseteq \lattice_{d, n}$ of the lattice
is any set for which if $y \in u$ and $z - y \succeq 0$, then $z \in u$. Lower sets are defined analogously.}.
Recall the min-max characterization of the isotonic projection
\citep[Chapter 1]{robertson1988order}: For each tensor
$a \in \Tspace_{d, n}$ and $x \in \lattice_{d, n}$, we have
\begin{align*}
\Pscript(a)(x) = \min_{L \in \Lspace(x)} \max_{U \in
\Uspace(x)} \abar_{L \cap U}.
\end{align*}
Consequently, for each pair of tensors $a, b \in \Tspace_
{d, n}$, we obtain the sequence of bounds
\begin{align*}
|\Pscript(a)(x) - \Pscript(b)(x)| &= \left| \min_{L \in \Lspace(x)} \max_{U \in
\Uspace(x)} \abar_{L \cap U} - \min_{L \in \Lspace(x)} \max_
{U \in
\Uspace(x)} \bbar_{L \cap U} \right| \\
&\leq \max_{L \in \Lspace(x)} \left| \max_{U \in
\Uspace(x)} \abar_{L \cap U} - \max_
{U \in
\Uspace(x)} \bbar_{L \cap U} \right| \\
&\leq \max_{L \in \Lspace(x)} \max_{U \in
\Uspace(x)} | \abar_{L \cap U} - \bbar_{L \cap U} | \\
&\leq \| a - b \|_{\infty}.
\end{align*}
Since this
holds for all $x \in \lattice_{d, n}$, we have proved the claimed
result.
\end{proof}

As an immediate corollary of equation~\eqref{eq:perm-contract}, we obtain the
following result that may be of independent interest.
\begin{corollary}
The isotonic projection is $\ell_\infty$ contractive, i.e., for any $\theta,
\theta' \in \Tspace_{d, n}$, we have
\begin{align*}
\| \Pscript( \theta; \id, \ldots, \id) - \Pscript( \theta'; \id, \ldots,
\id)
\|_{\infty} &\leq \| \theta - \theta' \|_{\infty}.
\end{align*}
\end{corollary}
To the best of our knowledge, similar results are only known when $d = 1$
\citep{yang2019contraction}.

\section{Ancillary results} \label{app:ancillary}

In this section, we collect several results that are used in multiple
proofs.

\subsection{Basic lemmas for least squares estimators} \label{app:lse}

Our first lemma allows us to bound the expected supremum of a
Gaussian process over a union of sets in terms of the individual
expected suprema. Similar results have
appeared in the literature~\citep{chatterjee2019adaptive,guntuboyina2020adaptive}.
We state a version that can be readily deduced from~\citet[Lemma
D.1]{guntuboyina2020adaptive}.
\begin{lemma} \label{lem:sup-emp}
Let $K \geq 1$, and let $\epsilon$ denote a standard
Gaussian tensor in~$\real_{d, n}$. Suppose that for some positive scalar $t$, we
have $\Theta_1, \ldots, \Theta_K \subseteq \mathbb{B}_2 (t)$. There is a
universal positive constant $C$ such that \\
\noindent (a) The supremum of the empirical process satisfies
\begin{align*}
\Pr \left\{ \max_{ k \in [K]} \sup_{\theta \in \Theta_k} \inprod{\epsilon}
{\theta} \geq \max_{k \in [K]} \EE \left[ \sup_{\theta \in \Theta_k} 
\inprod{\epsilon}{\theta} \right] + C t (\sqrt{\log K} + \sqrt{u}) \right\} \leq
e^{-u} \ \ \text{ for all } u \geq 0.
\end{align*}

\noindent (b) If, in addition, the all-zero tensor is contained in each
individual set $\{\Theta_k \}_{k =
1}^K$, then 
\begin{align*}
\EE\left[ \max_{ k \in [K]} \sup_{\theta \in \Theta_k} \inprod{\epsilon}
{\theta} \right]
\leq \max_{k \in [K]} \EE \left[ \sup_{\theta \in \Theta_k} \inprod{\epsilon}
{\theta}
\right] + C t\sqrt{\log K}.
\end{align*}
\end{lemma}

Our second lemma bounds the supremum of a Gaussian process over a set that is
piecewise constant over known blocks. Recall that for $\theta \in \Tspace_{d,
n}$ and $S \subseteq \lattice_{d, n}$, we let $\theta_S$ denote the sub-tensor
formed by restricting $\theta$ to indices in $S$. In the statement of the lemma,
we also use the notation of stochastic dominance: for a pair of scalar random
variables $(X_1, X_2)$, the relation $X_1 \stackrel{d} {\leq} X_2$ means that
$\Pr\{ X_1 \geq t \} \leq \Pr\{ X_2 \geq t \}$ for each $t \in \real$.
\begin{lemma} \label{lem:const-blocks-ep}
Let $B_1, \ldots, B_s$ denote a (known) partition of the lattice $\lattice_{d,
n}$. Let $\Theta \subseteq \Tspace_{d, n}$ denote a collection of tensors such
that for
each $\theta \in \Theta$ and $\ell \in [s]$, the sub-tensor $\theta_{B_\ell}$ is
constant. Let $\epsilon \in \Tspace_{d, n}$ represent a standard Gaussian
tensor. Then, for each $t \geq 0$, we have
\begin{align*}
\sup_{\theta \in \Theta \cap \mathbb{B}_2(t) } \inprod{\epsilon}{\theta}  
\stackrel{d} {\leq} t \cdot Y_s,
\end{align*}
where $Y^2_s \sim \chi^2_s$.
Consequently, we have
\begin{align*}
\EE \left[ \sup_{\theta \in \Theta \cap \mathbb{B}_2(t) } \inprod{\epsilon}{\theta}  \right]
\leq t \sqrt{s}.
\end{align*}
\end{lemma}
\begin{proof}
We focus on proving the first claim, since the second claim follows immediately
from
it by Jensen's inequality. For each $S \subseteq \lattice_{d, n}$, we write
$\thetabar_S \defn \frac{1}
{|S|} \sum_{x \in
S} \theta_x$. By definition of the set $\Theta$, each sub-tensor $\theta_{B_{\ell}}$ is constant. Consequently, we have the decomposition
\begin{align*}
\inprod{\epsilon}{\theta} = \sum_{\ell \in [s]} \sum_{x \in B_\ell} \epsilon_{x}
\cdot \theta_x = \sum_{\ell \in [s]} \sqrt{|B_\ell|} \cdot \thetabar_{B_{\ell}}
\cdot 
\frac{\sum_
{x \in B_{\ell}} \epsilon_x}{ \sqrt{|B_{\ell}|}}.
\end{align*}
Now define the $s$-dimensional vectors $\epstil$ and $v(\theta)$ via
\begin{align*}
\epstil_\ell \defn \frac{\sum_
{x \in B_{\ell}} \epsilon_x}{ \sqrt{|B_{\ell}|}} \quad \text{ and } \quad [v
(\theta)]_{\ell}
\defn \sqrt{|B_{\ell}|} \cdot \thetabar_{B_{\ell}}, \quad \text{ for each } \;
\ell
\in [s].
\end{align*}
By construction, the vector $\epstil$ consists of standard Gaussian entries, and
we also have 
\[
\| v(\theta) \|_2^2 = \sum_{\ell} |B_{\ell}| \cdot (\thetabar_{B_{\ell}})^2 = \sum_{\ell} \| \theta_{B_{\ell}} \|_2^2 = \| \theta \|_2^2
\] 
for each $\theta \in \Theta$.
Applying the Cauchy--Schwarz inequality then yields
\begin{align*}
\sup_{\theta \in \Theta \cap \mathbb{B}_2 (t)} \inprod{\epsilon}{\theta} \leq
\sup_
{ \substack{v \in \real^s \\ \| v \|_2 \leq t} } \inprod{\epstil}{v} \leq t
\cdot \| \epstil \|_2,
\end{align*}
as desired.
\end{proof}

Our third lemma follows almost directly from
\citet[Theorem 13.5]{wainwright2019high}, after a little bit of algebraic
manipulation.
In order to state the lemma, we require a few preliminaries.

\begin{definition} \label{def:star-shaped}
A set $\mathcal{C}$ is star-shaped if for all $\theta \in \Cspace$
and $\alpha \in [0, 1]$, the inclusion $\alpha \theta \in \Cspace$ holds. We say
that $\mathcal{C}$ is additionally non-degenerate if it does not consist solely
of the zero element.
\end{definition}
Let $\epsilon$ denote a standard Gaussian in $\real_{d, n}$, and suppose that
the set 
$\Theta \subseteq \real_{d, n}$ is star-shaped and non-degenerate.
Let $\Deltahat \in \real_{d,
n}$ denote a (random) tensor satisfying the pointwise inequality
\begin{align*}
\| \Deltahat \|_2^2 \leq \sup_{ \substack{\Delta \in \Theta \\ \| \Delta \|_2
\leq \| \Deltahat \|_2 }} \inprod{\epsilon}{\Delta},
\end{align*}
and for each $t \geq 0$, define the random variable
\begin{align*}
\xi(t) = \sup_{ \substack{\Delta \in \Theta \\ \| \Delta \|_2
\leq t }} \inprod{\epsilon}{\Delta}.
\end{align*}
Let $t_n$ denote the smallest positive solution to the critical inequality
\begin{align*}
\EE [\xi(t)] \leq \frac{t^2}{2}.
\end{align*}
Such a solution always exists provided $\Theta$ is star-shaped and
non-degenerate; this can be shown
via a standard rescaling argument 
(see~\citet[Lemma 13.6]{wainwright2019high}).

We are now ready to state a high probability bound on the error $\| \Deltahat
\|_2^2$.
\begin{lemma} \label{lem:star-shaped}
Under the setup above, there is a pair of universal positive constants $(c, C)$
such that
\begin{align*}
\Pr \left\{ \| \Deltahat \|_2^2 \geq C t_n^2 + u \right\} \leq \exp\left\{ -
c u \right\} \quad \text{ for all } \quad u \geq 0.
\end{align*}
Consequently,
\begin{align*}
\EE[ \| \Deltahat \|_2^2 ] \leq C (t_n^2 + 1).
\end{align*}
\end{lemma}
\begin{proof}
Applying~\citet[Theorem 13.5]{wainwright2019high} and
rescaling appropriately yields the bound
\begin{align*}
\Pr \left\{ \| \Deltahat \|_2^2 \geq 16 t_n \cdot \delta \right\} \leq \exp\left
\{- \frac{\delta t_n}{2} \right\} \quad \text{ for all } \quad \delta \geq t_n.
\end{align*}
Now note that $t_n > 0$, and that for any $u \geq 0$, we may set $\delta = t_n +
\frac{u}{16 t_n}$. This yields the bound
\begin{align*}
\Pr \left\{ \| \Deltahat \|_2^2 \geq 16 t^2_n + u \right\} \leq
\exp\left
\{- \frac{t^2_n}{2} \right\} \cdot \exp\left\{- \frac{u}{32} \right\} \leq
\exp\left\{- \frac{u}{32} \right\}.
\end{align*}
The bound on the expectation follows straightforwardly by integrating the tail
bound.
\end{proof}

Our fourth lemma is an immediate corollary of
\citet[Theorem 2.1]{van2017concentration}, and shows that the
error of a least squares estimator---recall our notation from equation
\eqref{eq:lse-defn}---over a convex set concentrates around its
expected value. 
\begin{lemma} \label{lem:vdG-W}
Let $\epsilon$ denote a standard Gaussian tensor in $\Tspace_{d, n}$, and let $K
\subseteq \Tspace_{d,
n}$ denote a closed convex set. For a fixed tensor $\thetastar \in K$, let
$\thetahat =
\thetahatlse(K, \thetastar + \epsilon)$. Then for each $u \geq 0$: \\
\noindent (a) The $\ell_2$ norm of the error satisfies the two-sided tail bound
\begin{align*}
\Pr \left\{ \big| \| \thetahat - \thetastar \|_2 - \EE [ \| \thetahat -
\thetastar \|_2 ] \big| \geq \sqrt{2u} \right\} \leq e^{-u}.
\end{align*}

\noindent (b) The squared $\ell_2$ norm of the error satisfies the one-sided
tail bound
\begin{align*}
\Pr \left\{ \| \thetahat - \thetastar \|_2^2 \geq 2 \EE [ \| \thetahat -
\thetastar \|_2^2 ] + 4u \right\} \leq e^{-u}.
\end{align*}
\end{lemma}
\begin{proof}
Part (a) of the lemma follows directly by rescaling the terms in~\citet[Theorem
2.1]{van2017concentration}. Part (b) of the lemma follows from part (a) by
noting that if
$\| \thetahat - \thetastar \|_2 - \EE [ \| \thetahat -
\thetastar \|_2 ]| \leq \sqrt{2u}$, then 
\begin{align*}
\| \thetahat - \thetastar \|_2^2 \leq 2 \left( \EE [ \| \thetahat -
\thetastar \|_2 ] \right)^2 + 2 \cdot 2u \leq 2 \EE [ \| \thetahat -
\thetastar \|_2^2 ]  + 4u. 
\end{align*}
\end{proof}

Finally, we state the following variant of Chatterjee's variational
formula~\citep{Cha14}, due to~\citet[Lemma 6.1]{FlaMaoRig16}. Unlike in
Chatterjee's original result and specific improvements
\citep[e.g.,][]{chatterjee2019adaptive}, 
this lemma applies to least squares estimators
that are obtained via projections onto closed (not necessarily convex) sets.

\begin{lemma} \label{lem:variational}
Let $\mathcal{C}$ be a closed subset of $\Tspace_{d, n}$, and let $\epsilon \in
\Tspace_{d, n}$ denote a (not necessarily Gaussian) noise tensor.
For $\thetastar \in \mathcal{C}$, let $\thetahat \defn \thetahatlse (\mathcal{C},
\thetastar + \epsilon)$ denote the least squares estimator.
Define a function $f_{\thetastar}: [0, \infty) \to {\real}$ by
\[
f_{\thetastar}(t) = \sup_{ \substack{ \theta \in \mathcal{C} \\ \| \theta -
\thetastar \|_2 \le t } }
\inprod{\epsilon}{\theta - \thetastar } - \frac{t^2}{2}.
\]
If there exists a scalar $t_{*} > 0$ such that $f_{\thetastar}(t) < 0$ for all
$t \ge t_{*}$,
then we have
$\|\thetahat - \thetastar\|_2 \le t_{*}$.
\end{lemma}


\subsection{Results on average-case reductions} \label{app:hardness}

In this section, we present the Gaussian rejection kernel from~\citet{brennan2018reducibility} as a concrete algorithm below for completeness. We also
collect a lemma about the low total variation distortion property that it enjoys
when transforming a Bernoulli random variable into a Gaussian random variable.

\medskip

\noindent \hrulefill

\noindent \underline{{\bf Algorithm Gaussian rejection kernel} $\textsc{rk}(\rho, B)$}

\smallskip

\noindent \textit{Parameters}: Input $B \in \{0, 1\}$, Bernoulli probabilities $0 < q < p \le 1$, Gaussian mean $\rho$, number of iterations $T$, let $\varphi_\rho(x) = \frac{1}{\sqrt{2\pi}} \cdot \exp\left(- \frac{1}{2}(x - \rho)^2 \right)$ denote the density of $\NORMAL(\rho, 1)$.
\begin{itemize}
\item Initialize $z \gets 0$.
\item Until $z$ is set or $T$ iterations have elapsed:
\begin{itemize}
\item[(1)] Sample $z' \sim \NORMAL(0, 1)$ independently.
\item[(2)] If $B = 0$ and the condition
$$p \cdot \varphi_0(z') \ge q \cdot \varphi_{\rho}(z')$$
holds, then set $z \gets z'$ with probability $1 - \frac{q \cdot \varphi_\rho(z')}{p \cdot \varphi_0(z')}$.
\item[(3)] If $B = 1$ and the condition
$$(1 - q) \cdot \varphi_\rho(z' + \rho) \ge (1 - p) \cdot \varphi_0(z' + \rho)$$
holds, then set $z \gets z' + \rho$ with probability $1 - \frac{(1 - p) \cdot \varphi_0(z' + \rho)}{(1 - q) \cdot \varphi_\rho(z' + \rho)}$.
\end{itemize}
\end{itemize}
\noindent \emph{Output} $z$.

\noindent \hrulefill

\begin{lemma}[Gaussian Rejection Kernels -- Lemma 5.4 in \citet{brennan2018reducibility}] \label{lem:BB}
Suppose that $\rho \in (0, 1)$ satisfies
\begin{align} \label{eq:rho-def}
\rho = \frac{\log 2}{2 \sqrt{6 (d + 1) \log n_1 + 2\log 2}}.
\end{align}
Then the map $\textsc{rk}$ with $T = \left\lceil 6 (\log 2)^{-1} (d + 1) \log n_1 \right\rceil$ iterations can be computed in time polynomial in $n_1^d$ and satisfies
\begin{align*}
\TV\left(\textsc{rk}(\rho, \BER(1)), \NORMAL(\rho, 1) \right) &= O\left(n_1^{-(d+1)}\right) \; \text{ and } \\
\TV\left(\textsc{rk}(\rho, \BER(1/2)), \NORMAL(0, 1) \right) &= O\left(n_1^{-(d+1)}\right).
\end{align*}
\end{lemma}

In the next lemma, we state a straightforward consequence of Lemma~\ref{lem:BB} to
the tensor case. Abusing notation slightly, we let $\NORMAL (\mu, I)$ denote a normal
distribution having tensor valued mean $\mu \in \Tspace_{d, n}$, and independent entries of unit variance.

\begin{lemma} \label{lem:GRK-tensor}
Let $Y \in \Tspace_{d, n}$ be a tensor with independent Bernoulli entries having mean either $1$ or $1/2$, and let $\mu \defn \EE[Y]$. With $\rho$ as in equation~\eqref{eq:rho-def}, define the tensor $\overline{\mu}\defn 2 \rho (\mu - \frac{1}{2} \cdot \mathbbm{1}_{d, n} )$. Let $\textsc{rk}(Y)$ denote the law of the random tensor obtained by applying the Gaussian rejection kernel to each entry of $Y$ with the choice of parameters given in Lemma~\ref{lem:BB}. Then,
\begin{align*}
\TV(\textsc{rk}(Y), \NORMAL(\overline{\mu}, I)) \leq C / n_1.
\end{align*} 
\end{lemma}

\subsection{Some other useful lemmas}

We first state a useful (deterministic) lemma regarding permutations, which
generalizes~\citet[Lemma A.10]{mao2018towards}.
\begin{lemma}
\label{lem:per-num}
Let $\{a_i\}_{i=1}^{n}$ be a non-decreasing sequence of real
numbers, let $\{b_i\}_{i=1}^{n}$ be a sequence of real numbers, and let $\tau$
be a positive scalar.
If $\pi$ is a permutation in $\Pspace_{n}$ such that $\pi(i) <
\pi(j)$ whenever $b_j - b_i > \tau$, then
$|a_{\pi(i)} - a_i| \leq \tau + 2\| b - a \|_\infty$ for all $i \in [n]$. Here,
we have defined the vectors $a = (a_1, \ldots, a_n)$ and $b = (b_1, \ldots,
b_n)$.
\end{lemma}

\begin{proof}
The proof is by contradiction.
Letting $\Delta = b - a$, assume that $a_j
- a_{\pi(j)} > \tau + 2 \| \Delta
\|_\infty$ for some index $j \in [n]$. Since $\pi$ is a bijection, there must
exist---by the pigeonhole principle---an index $i \leq \pi(j)$ such that $\pi(i)
\geq \pi(j)$. Hence, we have
\[
b_j - b_i = a_j - a_i + \Delta_j - \Delta_i \geq a_j - a_{\pi(j)} + \Delta_j -
\Delta_i > \tau + 2 \| \Delta \|_\infty - 2 \| \Delta \|_\infty,
\] 
which contradicts the assumption that $\pi(i) < \pi(j)$ whenever $b_j - b_i >
\tau$. 

On the other hand, suppose that $a_{\pi(j)} - a_j > \tau + 2 \| \Delta
\|_\infty$ for some $j \in [n]$. Since $\pi$ is a bijection, there must
exist an index $i \geq \pi(j)$ such that $\pi(i) \leq \pi(j)$. In this case, we
have
\[
b_i - b_j = a_i - a_j + \Delta_i - \Delta_j \geq a_{\pi(j)} - a_{j} +
\Delta_i -
\Delta_j > \tau,
\] 
which also leads to a contradiction.
\end{proof}

Our next technical lemma is a basic result about random variables. 
\begin{lemma} \label{lem:RV}
Let $(X,Y,Z)$ denote a triple of real-valued random variables defined on a
common probability space, with $X^2 \leq Z^2$ almost surely. Let $\Espace$ be a
measurable event such that on the event $\Espace$, we have $X^2 \leq Y^2$. Then,
\begin{align*}
\EE[X^2] \leq \EE[Y^2] + \sqrt{\EE [Z^4]} \cdot \sqrt{\Pr\{ \Espace^c\}}.
\end{align*}

\end{lemma}
\begin{proof}
Since $X^2 \leq Y^2$ on $\Espace$, we have $X^2 \leq Y^2 \ind{\Espace} + X^2 
\ind{\Espace^c}$. Consequently,
\begin{align*}
\EE[X^2] &\leq \EE[Y^2 \ind{\Espace}] + \EE[X^2 \ind{\Espace^c}] \\
&\leq \EE[Y^2] + \EE[Z^2 \ind{\Espace^c}] \\
&\leq \EE[Y^2] + \sqrt{\EE[Z^4]} \cdot \sqrt{\EE[\ind{\Espace^c}]},
\end{align*}
where the final inequality is an application of the Cauchy--Schwarz inequality.
\end{proof}

Our third lemma is an elementary type of rearrangement inequality. For a
permutation $\pi \in \Pspace_n$ and vector $v \in \real^n$, we use the notation
$v\{\pi\}$ to denote the vector formed by permuting the entries of $v$ according
to $\pi$, so that $v\{\pi\} = (v_{\pi(1)}, \ldots, v_{\pi(n)})$. Let
$\ones_n$ denote the $n$-dimensional all-ones vector.

\begin{lemma} \label{lem:rearrange}
Let $v \in \real^n$ with $\vbar = \left( \frac{1}{n} \sum_{i = 1}^n v_i \right)
\cdot \ones_n$. Then, we have
\begin{align*}
\| v - \vbar \|_2^2 \leq \max_{\pi \in \Pspace_n} \| v - v\{\pi\} \|_2^2.
\end{align*}
\end{lemma}
\begin{proof}
First note that
$\vbar = \frac{1}{|\Pspace_n|} \sum_{\pi \in \Pspace_n} v\{\pi\}$, so that we
have
\begin{align*}
\| v - \vbar \|_2^2 = \left\| v - \frac{1}{|\Pspace_n|} \sum_{\pi \in \Pspace_n}
v\{\pi\} \right\|_2^2 \stackrel{\1}{\leq} \frac{1}{|\Pspace_n|} \sum_{\pi \in
\Pspace_n} \| v - v\{\pi\} \|_2^2 \leq \max_{\pi \in \Pspace_n} \| v - v
\{\pi\}
\|_2^2,
\end{align*}
where step $\1$ follows from Jensen's inequality. 
\end{proof}
Next, we state an elementary lemma that bounds the number of distinct
one-dimensional ordered partitions satisfying certain conditions. Recall that
$\Partition_L$ denotes the set of all one-dimensional ordered partitions of the
set $
[n_1]$ consisting of
exactly $L$ blocks. Also recall that $\Partition^{\max}_{k}$ denotes all
one-dimensional ordered partitions of $[n_1]$ in which the largest block has
size at
least $k$.
\begin{lemma} \label{lem:num-partitions}
For each $n_1 \geq 2$, the following statements hold: \\
\noindent (a) For any $L \in [n_1]$, we have
\begin{align*}
\left| \bigcup_{\ell = 1}^L \Partition_\ell \right| = L^{n_1}, \quad  \text{ and
so } \quad
|\Partition | = \left| \bigcup_{\ell = 1}^{n_1} \Partition_\ell \right| = (n_1)^
{n_1}.
\end{align*}
\noindent (b) For any $k^* \in [n_1]$, we have
\begin{align*}
|\Partition^{\max}_{k^*} | \leq (n_1)^{3 (n_1 - k^*)}.
\end{align*}
\end{lemma}
\begin{proof}
The second claim of part (a) follows from the first. In order to prove the first
claim, note that each $i \in [n_1]$ can be placed into any one of the $L$ blocks,
and each different choice yields a different element of the set $\cup_{\ell=1}^L
\Partition_\ell$.

We now proceed to prove part (b). First, note that the claim is
immediately true whenever $k^* \geq n_1 - 1$, since
\begin{align*}
|\Partition^{\max}_{n_1} | = 1 \quad \text{ and } \quad |\Partition^{\max}_{n_1
- 1} | =
2n_1.
\end{align*}
Consequently, we focus our proof on the case $k^* \leq n_1 - 2$, in which case
$n_1 - k^* + 1 \leq \frac{3}{2} (n_1 - k^*)$. Suppose we are
interested in
bounding the number of one-dimensional ordered partitions in which the largest
block has size at
least $k$ and the number of blocks is at most $s_1$. Then there are $
\binom{n_1}{k}$ distinct ways of choosing the first $k$ elements of the largest
block and
$s_1$ ways
of choosing the position of the largest block. After having done this, the
remaining $n_1 - k$ elements of $[n_1]$ can be placed in any of the $s_1$
blocks. Finally, note that $s_1 \leq n_1 - k + 1$, so that the number of such
one-dimensional ordered partitions is bounded above by
\begin{align*}
\binom{n_1}{k^*} \cdot s_1 \cdot 
s_1^{n_1 - k} \leq \binom{n_1}{k^*} \cdot (n_1 - k + 1)^{n_1 - k + 1}.
\end{align*}
Choosing $k = k^*$ yields the bound
\begin{align*}
|\Partition^{\max}_{k^*} | &\leq \binom{n_1}{k^*} \cdot (n_1 - k^* + 1)^{n_1 -
k^* + 1} \\
&= \frac{n_1!}{(k^*)!} \cdot (n_1 - k^* + 1) \cdot \frac{(n_1 - k^* + 1)^{n_1 -
k^* + 1}}{(n_1 - k^* + 1)!} \\
&\leq \frac{n_1!}{(k^*)!} \cdot \sqrt{n_1 - k^* + 1} \cdot e^{n_1 - k^* + 1}
\cdot 
(2 \pi)^{-1/2} \\
&\leq (n_1)^{n_1 - k^*} \cdot \sqrt{n_1 - k^* + 1} \cdot e^{n_1 - k^* + 1}
\cdot 
(2 \pi)^{-1/2}
\end{align*}
where the second inequality uses the bound $n! \geq \left(\frac{n}{e}\right)^n
\sqrt{2 \pi n}$ given by Stirling's approximation. Now note that for each $n_1
\geq 2$, we have $e / \sqrt{2 \pi} \leq \sqrt{n_1}$. Combining this with the
bound $n_1 - k^* + 1 \leq n_1$ and putting together the pieces, we have
\begin{align*}
|\Partition^{\max}_{k^*} | &\leq (n_1)^{n_1 - k^* + 1} \cdot e^{n_1 - k^*} \leq
(n_1)^{3(n_1 - k^*)},
\end{align*}
where the final inequality is a consequence of the bounds $n_1 - k^* + 1 \leq
\frac{3}{2}(n_1 - k^*)$ and $e \leq (n_1)^{3/2}$ for each $n_1 \geq 2$.
\end{proof}

Finally, we state an elementary lemma about the concentration of hypergeometric random variables, whose proof is provided for completeness.
\begin{lemma} \label{lem:HG}
Suppose that $d$ is a fixed positive integer that divides $N$, and that $X \sim \Hyp(N/d, N, K)$. There is a universal positive constant $c$ such that
\begin{align*}
\Pr \{ \big| X - K / d \big| \geq K / (2d) \} \leq 2e^{-c K / d}.
\end{align*}
In fact, it suffices to take $c = 1/9$.
\end{lemma}

\begin{proof}
Define the random variable $Y \sim \BIN(N/d, K/N)$, and note that $\EE[X] = \EE[Y] = K / d =: \mu$. Also use the shorthand $p \defn K/ N$ and $n \defn N / d$. Markov's inequality applied to the moment generating function yields that for each $\delta > 0$, we have
\begin{align*}
\Pr\{ X \geq (1 + \delta) \mu \} \leq \inf_{\lambda > 0} \; \frac{\EE[e^{\lambda X}]}{e^{\lambda (1 + \delta) \mu}} &\stackrel{\1}{\leq} \inf_{\lambda > 0} \; \frac{\EE[e^{\lambda Y}]}{e^{\lambda (1 + \delta) \mu}} \\
&= \inf_{\lambda > 0} \; \frac{(p e^\lambda + (1 - p))^n}{e^{\lambda (1 + \delta) \mu}} \\
&\stackrel{\2}{\leq} \inf_{\lambda > 0} \; \frac{\exp( \mu(e^\lambda - 1))}{e^{\lambda (1 + \delta) \mu}} \\
&\stackrel{\3}{\leq} \left( \frac{e^\delta}{(1 + \delta)^{1 + \delta}} \right)^{\mu} \\
&\stackrel{\4}{\leq} e^{-\delta^2 \mu / (2 + \delta)}.
\end{align*}
Here, step $\1$ is a classical result due to~\citet[Theorem 4]{hoeffding} (see also~\citet[Lemma 1.1]{bardenet2015concentration}). Step $\2$ uses the inequality $1 + x \leq e^x$ to deduce that 
\begin{align*}
(p e^{\lambda} + (1 - p))^n = (p (e^{\lambda} - 1) + 1)^n \leq \exp(p (e^{\lambda} - 1) n) = \exp(\mu (e^{\lambda} - 1) ).
\end{align*}
Step $\3$ follows from setting $\lambda = \log (1 + \delta)$, and step $\4$ follows
from using the inequality $1 + \delta \geq e^{2 \delta / (2 + \delta)}$, which holds for all $\delta \geq 0$.

Proceeding identically for the lower tail yields the bound
\begin{align*}
\Pr\{ X \leq (1 - \delta) \mu \} \leq \left( \frac{e^{-\delta}}{(1 - \delta)^{1 - \delta}} \right)^{\mu} \stackrel{\5}{\leq} e^{- \delta^2 \mu / 2} \quad \text{ for all } 0 < \delta \leq 1,
\end{align*}
where step $\5$ follows from the inequality $(1 - \delta)^{2(1 - \delta)} \geq e^{-\delta(2 - \delta)}$, which holds for all $\delta \in [0, 1]$.

Putting together the two tail bounds with the substitution $\delta = 1/2$ completes the proof of the lemma.
\end{proof}

\end{document}